\documentclass[11pt,twoside]{amsart}

\usepackage{verbatim}
\usepackage{amssymb}
\usepackage{amsbsy}
\usepackage{amscd}
\usepackage{amsmath}
\usepackage{amsthm}
\usepackage[mathscr]{eucal}

\newtheorem{thm}[subsection]{Theorem}
\newtheorem{subthm}[subsubsection]{Theorem}
\newtheorem{prop}[subsection]{Proposition}
\newtheorem{lemma}[subsection]{Lemma}
\newtheorem{sublemma}[subsubsection]{Lemma}
\newtheorem{cor}[subsection]{Corollary}
\newtheorem{claim}[subsubsection]{Claim}

\theoremstyle{definition}
\newtheorem{defn}[subsection]{Definition}
\newtheorem{subdefn}[subsubsection]{Definition}

\newtheorem{example}[subsection]{Example}
\newtheorem{subcase}[subsubsection]{Case}

\newtheorem{rem}[subsection]{Remark}
\newtheorem{subrem}[subsubsection]{Remark}
\newtheorem{summary}[subsection]{Summary}

\title[Compactification 
of the moduli space of abelian varieties]
{Compactification  by GIT-stability of the moduli space of abelian varieties}

\author[I. Nakamura]{Iku Nakamura}

\address{Department of Mathematics, 
Hokkaido University,Sapporo, 060-0810}

\email{nakamura@math.sci.hokudai.ac.jp}



\newcommand\op{\operatorname}

\newcommand\s{\smallskip}
\newcommand\Aut{\op{Aut}}

\newcommand\CL{\op{CL}}

\newcommand\Del{\op{Del}}
\newcommand\depth{\op{depth}}
\newcommand\Div{\op{Div}}
\newcommand\dAut{\op{Aut}^{\dagger}}
\newcommand\ddAut{\op{Aut}^{\dagger 0}}
\newcommand\diag{\op{diag}}

\newcommand\End{\op{End}\,}
\newcommand\Ext{\op{Ext}}
\newcommand\formal{\op{for}}

\newcommand\Gal{\op{Gal}}

\newcommand\GL{\op{GL}}
\newcommand\Hilb{\op{Hilb}}
\newcommand\hilb{\op{Hilb}}
\newcommand\Hom{\op{Hom}}
\newcommand\id{\op{id}}
\newcommand\inv{\op{inv}}
\newcommand\main{\op{main}}
\newcommand\Mor{\op{Mor}}

\newcommand\NL{\op{NL}}
\newcommand\norm{\op{norm}}
\newcommand\Obj{\op{Obj}}

\newcommand\PGL{\op{PGL}}

\newcommand\Pic{\op{Pic}}
\newcommand\Proj{\op{Proj}\,}
\newcommand\rank{\op{rank}}
\newcommand\red{\op{red}}

\newcommand\Sing{\op{Sing}\,}
\newcommand\SL{\op{SL}}
\newcommand\Spec{\op{Spec}\,}
\newcommand\Spf{\op{Spf}\,}
\newcommand\Spl{\op{Sp}}
\newcommand\sq{\op{sq}}

\newcommand\sqap{\op{sqap}}

\newcommand\Sym{\op{Sym}}

\newcommand\toric{\op{toric}}

\newcommand\val{\op{val}}
\newcommand\suniv{\op{su}}
\newcommand\univ{\op{univ}}
\newcommand\Weil{\op{Weil}}

\newcommand{\wR}{{\widetilde R}}

\newcommand{\dTheta}{\Theta^{\dagger}}
\newcommand{\wOmega}{{\widetilde\Omega}}
\newcommand{\dOmega}{\Omega^{\dagger}}
\newcommand{\bA}{{\bold A}}
\newcommand{\bC}{{\bold C}}
\newcommand{\bF}{{\bold F}}
\newcommand{\bG}{{\bold G}}
\newcommand{\bH}{{\bold H}}

\newcommand{\bP}{{\bold P}}
\newcommand{\bQ}{{\bold Q}}
\newcommand{\bR}{{\bold R}}

\newcommand{\bV}{{\bold V}}
\newcommand{\bZ}{{\bold Z}}

\newcommand{\cA}{{\cal A}}
\newcommand{\cC}{{\cal C}}
\newcommand{\cD}{{\cal D}}
\newcommand{\cG}{{\cal G}}
\newcommand{\cH}{{\cal H}}
\newcommand{\cL}{{\cal L}}
\newcommand{\cM}{{\cal M}}
\newcommand{\bcM}{\overline{\cal M}}
\newcommand{\cO}{{\cal O}}
\newcommand{\cP}{{\cal P}}
\newcommand{\cQ}{{\cal Q}}
\newcommand{\cS}{{\cal S}}
\newcommand{\cV}{{\cal V}}
\newcommand{\cW}{{\cal W}}
\newcommand{\cX}{{\cal X}}

\newcommand{\cZ}{{\cal Z}}

\newcommand{\barx}{{\bar x}}

\newcommand{\barAP}{\overline{AP}}

\newcommand{\barcAP}{\overline{\cA\cP}}

\ifx\undefined\cal
\newcommand{\cal}{\mathcal}
\fi

\newcommand{\theoremtextend}{
\hfill\begin{picture}(4,8)
\put(0,0){\line(1,0){4}}
\put(0,0){\line(0,1){8}}
\put(4,0){\line(0,1){8}}
\put(0,8){\line(1,0){4}}
\end{picture}
}

\begin{document}
\hspace*{-2cm}
\thanks{Research was supported by the Grant-in-aid 
(No. 23224001 (S)) for Scientific Research, 
JSPS}

\begin{abstract}
 The moduli space $\cM_g$
of nonsingular projective curves of genus $g$ 
is compactified into the moduli 
$\bcM_g$ of Deligne-Mumford 
stable curves of genus $g$. 
We compactify in a similar way  
the moduli space of 
abelian varieties by adding some mildly degenerating limits of 
abelian varieties. 
A typical case is the moduli space of 
Hesse cubics. Any Hesse cubic is GIT-stable in the sense that 
its $\SL(3)$-orbit is closed in the semistable locus, and conversely
any GIT-stable planar cubic is one of Hesse cubics. 
Similarly in arbitrary dimension,  the moduli
space of abelian varieties is compactified by adding 
only GIT-stable limits of abelian varieties 
(\S~\ref{sec:related topics}). \par
Our moduli space 
is a projective ``fine'' 
moduli space of possibly degenerate abelian schemes 
{\it with non-classical non-commutative 
level structure} 
 over $\bZ[\zeta_{N},1/N]$  
for some $N\geq 3$.  
The objects at the boundary are singular schemes, called PSQASes, 
projectively stable quasi-abelian schemes.
\end{abstract}

\maketitle
\setcounter{tocdepth}{1}
\tableofcontents

\section{Introduction}
The moduli of stable curves, 
the so-called Deligne-Mumford compactification,  
compactifies the moduli of nonsingular curves :   
\begin{equation*}
\begin{aligned}
\op{the}&\ \text{moduli of smooth curves}\\
   &=\text{the set of all isomorphism classes of smooth curves}\\
   &\subset  \text{the set of all isomorphism classes of stable
curves}\\
   &
=\text{the Deligne-Mumford compactification $\bcM_g$}
\end{aligned}
\end{equation*}

The moduli of stable curves is known to be a projective scheme, 
while the moduli of nonsingular curves is a Zariski open subset of it.  
\par
Our problem is to do the same for moduli of smooth abelian varieties. 
We find certain natural limits of smooth abelian varieties 
similar to stable curves to compactify the moduli. In other words, 
we will construct a new compactification $SQ_{g,K}$, the moduli of 
some possibly degenerate abelian varieties with some extra structure, 
which contains the moduli of smooth abelian varieties with 
similar extra structure as a Zariski open subset. This will complete 
the following diagram : 
\begin{equation*}
\begin{aligned}
\text{the}&\ \text{moduli of smooth AVs ($=$ abelian varieties)}\\
   &=\{\text{smooth polarized AVs + extra structure}\}/\op{isom.}\\ 
   &\subset  \{\text{smooth polarized AVs or}\\
   &\hskip 18pt   \text{singular polarized degenerate AVs 
             + extra structure}\}/\op{isom.}\\
   &=\ \text{the new compactification $SQ_{g,K}$}
\end{aligned}
\end{equation*}


The compactification problem of the moduli space of abelian varieties 
has been studied by many people : 
\begin{enumerate}
\item[$\circ$] Satake compactification, Igusa monoidal transform of it
\item[$\circ$] Mumford toroidal compactification 
(\cite[(1975)]{AMRT10})
\item[$\circ$] Faltings-Chai arithmetic compactification 
(arithmetic version of Mumford compactification) \cite[(1990)]{FC90}
\end{enumerate}

These are the compactifications which had been known before 1995 
when the author restarted the research of compactifications. 
These are compactifications as spaces, 
not as the moduli of compact objects.  
In this article, we  are going to construct {\em a natural compactification,} 
in fact, {\em projective,} {\em as 
the ``fine/coarse'' moduli space of compact geometric objects}, where  
\begin{enumerate}
\item[$\circ$] the moduli space contains the moduli space 
of abelian varieties as a dense Zariski open subset, 
\item[$\circ$] it is compact, which amounts to collecting enough limits,
\item[$\circ$] it is separated, which amounts to choosing the minimum possible among the above.
\end{enumerate}

The following are the works closely related to the subject;
first of all,  the works of Mumford \cite{Mumford66}, 
\cite{Mumford67} and \cite{Mumford72} during 1966--1972, 
though they do not focus on compactifications directly. 
After 1975 there appeared  Nakamura \cite{Nakamura75} and Namikawa 
\cite{Namikawa76}, closely related to this article. \par
After 1999 there appeared several works on the subject: 
\cite{AN99},  \cite{Nakamura99}, 
\cite{Alexeev02}, \cite{Ols08} and \cite{Nakamura10}. 
By modifying \cite{Nakamura75}, Nakamura \cite{Nakamura99} and 
\cite{Nakamura10} study two kinds of compactifications of 
the moduli space of abelian varieties (with no zero section specified and 
with no semi-abelian scheme action assumed). 
Meanwhile, 
Alexeev \cite{Alexeev02} and Olsson \cite{Ols08} 
study the complete moduli spaces of certain 
schemes with semi-abelian scheme action.\par
Now we shall explain  
how we choose our compactification $SQ_{g,K}$, 
which will explain why 
the title of this article refers to GIT-stability. \par 
Let $H$ be 
a finite Abelian group, 
and $V:=V_H$ the unique irreducible 
representation of the Heisenberg group $\cG_H$ of weight one. 
Let $\bP(V)$ be the projective space of $V$, 
and $X:=\Hilb^{\chi}_{\bP(V)}$ 
the Hilbert scheme parameterizing closed subschemes of $\bP(V)$ 
with Hilbert polynomials $\chi(n)=n^g|H|$. 
According to GIT, our problem of compactifying the moduli space is, 
{\em very roughly speaking}, reduced to 
studying the quotient 
$X_{ss}/\!/\SL(V)$ where 
$X_{ss}$ denotes the semistable locus of $X$  
with respect to 
$\SL(V)$. 
GIT tells us, {\em set-theoretically},  
\begin{equation}\label{eq:closed orbits}
X_{ss}/\!/\SL(V)=\text{the set of all closed orbits in $X_{ss}$}.
\end{equation} 
See Section~\ref{sec:related topics}.
This scenario has to be modified a little. 
In an appropriately modified scenario,  
the LHS of (\ref{eq:closed orbits}) is 
the moduli space $SQ_{g,K}$, the compactification 
in the title of this article, 
while the RHS of (\ref{eq:closed orbits}) is 
just the set of isomorphism classes of 
our degenerate abelian schemes 
PSQASes $(Q_0,\cL_0)$ 
with $\cG_H$-action.  See Section~\ref{sec:PSQAS TSQAS}, 
Theorem~\ref{thm:fine moduli SQgK} and 
Theorem~\ref{thm:stability of PSQAS}. It should  be mentioned that 
$SQ_{g,K}$ is the fine moduli scheme 
for families of PSQASes over {\em reduced} base schemes, hence $SQ_{g,K}$ 
itself is also reduced.\par

This note is based on our lectures 
with the same title delivered 
at Kyoto university during  June 11--13, 2013. 
It overlaps the report 
\cite{Nakamura04} on the same topic in many respects, 
though the note includes also the recent progress of the topic. 
In this note, we give simple proofs for the major results of 
\cite{Nakamura99} and \cite{Nakamura10}, 
assuming known rather general results. 
We also tried to include (elementary or less elementary) 
proofs of the well-known related facts 
whose proofs  are hard to find in the literature. As a whole we tried to make 
our presentation more accessible than \cite{Nakamura99}, keeping the atmosphere of the lecture as much as possible. 
\footnote{This article will appear in the Proceedings  
of ASPM for Mukai 60 conference, Kyoto, 2013.}
\par
In what follows throughout this article, 
we always consider a finite abelian group 
$H=\bigoplus_{i=1}^g(\bZ/e_i\bZ)$, where 
$e_i|e_{i+1}$, and we write $N=|H|=\prod_{i=1}^ge_i$ 
and $K=K_H=H\oplus H^{\vee}$\ 
($H^{\vee}$ : the dual of $H$). 
We call such $H$ simply {\em a finite Abelian group.} 
We also call $K$ {\em a finite symplectic Abelian group.}
We also let $\cO=\cO_N=\bZ[1/N,\zeta_N]$ where 
$\zeta_N$ is a primitive $N$-th root of unity. \par
The article is organized as follows. \par
Section~\ref{sec:Hesse cubics} reviews 
the classical moduli theories of Hesse cubics with 
Neolithic level-3 structure or with classical level-3 structure.\par 
Section~\ref{sec:noncomm level} gives a new interpretation of  
the moduli theories in Section~\ref{sec:Hesse cubics} 
in a non-commutative way, and then  
explains a new moduli theory of Hesse cubics 
with level-$G(3)$ structure, where $G(3)$ 
is a non-commutative group, the Heisenberg group. 
This is the model theory for all the rest. 
The major purpose of this article is to 
explain its higher dimensional analogue. 
See Subsec.~\ref{subsec:for constructing separated moduli}. 
\par
In Section~\ref{sec:PSQAS TSQAS} 
we introduce two kinds $(P_0,\cL_0)$ and $(Q_0,\cL_0)$  of 
 nice degenerate abelian schemes in arbitrary dimension 
to compactify the moduli space of abelian varieties. 
Theorem~\ref{thm:refined stable reduction} 
gives an intrinsic description of 
those degenerate schemes 
$(P_0,\cL_0)$ and $(Q_0,\cL_0)$, where $P_0$ is always reduced, while 
$Q_0$ can be nonreduced. \par
A more direct definition of 
those  degenerate schemes will be given 
in Sections~\ref{sec:PSQAS dim one} and \ref{sec:PSQAS general case}. 
Especially we give a complete proof of the part 
$Q_{\eta}\simeq P_{\eta}\simeq G_{\eta}$ of 
Theorem~\ref{thm:refined stable reduction}. We will give 
two-dimensional and three-dimensional examples of PSQASes. 
We will also explain how 
a naive classical level-$n$ structure results in a 
nonseparated moduli. 
\par
 Section~\ref{sec:G action G linearization} 
reviews a rather general theory 
about $G$-action and $G$-linearization. 
We give various definitions and constructions and 
show their equivalence or compatibility. 
In Section~\ref{sec:moduli AgK SQgK}, 
we give a definition of level-$\cG_H$ structure and 
define a quasi-projective (resp. projective)
scheme $A_{g,K}$ (resp. $SQ_{g,K}$) when $e_1\geq 3$. 
We show that any geometric point of $A_{g,K}$ 
(resp. $SQ_{g,K}$) is a nonsingular level-$\cG_H$ 
PSQAS (resp. a level-$\cG_H$ PSQAS) and vice versa. 
In Section~\ref{sec:moduli psqas} we formulate the moduli functor of 
smooth (resp. flat) PSQASes over $\cO_N$-schemes 
(resp. reduced $\cO_N$-schemes). 
We will prove the representability of these functors 
by  $A_{g,K}$ (resp. $SQ_{g,K}$) in the respective category. \par
In Sections~\ref{sec:SQtoric} and \ref{sec:moduli of TSQAS} 
we see that there exists 
the coarse moduli algebraic space $SQ_{g,K}^{\toric}$ 
of level-$\cG_H$ TSQASes. 
This has been proved in \cite{Nakamura10} when $e_1\geq 3$.
We generalize it here to the case $e_1\leq 2$.  
There is a bijective morphism from 
$SQ_{g,K}^{\toric}$ onto $SQ_{g,K}$ if $e_1\geq 3$. 
In Sections~\ref{sec:SQtoric} and \ref{sec:moduli of TSQAS} many of 
the definitions, constructions and proofs are 
given in parallel to Sections~\ref{sec:moduli AgK SQgK} 
and \ref{sec:moduli psqas}, 
which we often omitted to avoid overlapping.
 \par 
In Section~\ref{sec:The morphisms to Alexeev's complete moduli space} 
we briefly report our recent results without proofs. 
We define a morphism $\sqap$ from $SQ_{g,K}^{\toric}\times U$ to 
Alexeev's complete moduli $\overline{AP}_{g,d}$ for a nonempty 
Zariski open  subset $U$ of $\bP^{N-1}=\bP(V_H)$. We see
that $\sqap$ restricted to 
$SQ_{g,K}^{\toric}\times \{u\}$ for any $u\in U$ is injective: 
in fact, it is almost a closed immersion.
We also see that  
$SQ_{g,1}^{\toric}$ is isomorphic to 
the main (reduced) component $\overline{AP}_{g,1}^{\main}$ 
of $\overline{AP}_{g,1}$, the closure in $\overline{AP}_{g,1}$ 
of the moduli of abelian torsors.  We emphasize 
that it is nontrivial to define {\em a well-defined morphism $\sqap$} because
singular TSQASes have a lot of continuous automorphisms.\par
In Section~\ref{sec:related topics} 
we explain the set of all closed orbits and 
GIT stability of PSQASes. We also mention a few related topics. 
\par

We tried to give complete proofs 
to Sections~\ref{sec:G action G linearization}-\ref{sec:moduli psqas} and
to Theorem~\ref{thm:fine/coarse moduli AgK} 
(especially for the case $e_1\leq 2$) 
in Section~\ref{sec:moduli of TSQAS},  
relying in part on \cite{Nakamura99} and \cite{Nakamura10}. 
In the other sections we only survey mainly
\cite{Nakamura99}, \cite{Nakamura10} and \cite{Nakamura14}.

\smallskip

{\bf Acknowledgements.} 
We are very grateful to 
Professors V.~Alexeev, J.-B.~Bost, A.~Fujiki,  
K.~Hulek, L.~Illusie, M.~Ishida, 
A.~King,   J.~McKay, Y.~Mieda, Y.~Odaka, T.~Shioda, and L.~Weng 
for their interest and advice on our works. 
Inspired by their advice, 
we have changed some of the presentations and 
especially the formulation of the functors 
$\cA_{g,K}$ and $\cS\cQ_{g,K}$, 
though we are not sure that it is the final form. 
We also thank K.~Sugawara for constant collaboration and support.

\section{Hesse cubics}
\label{sec:Hesse cubics}

Here we will start with a simple example.

\subsection{Hesse cubics}\label{subsec:Hesse cubics}
Let $k$ be  
any ring which contains $1/3$ and $\zeta_3$, the primitive cube root of unity.
A Hesse cubic curve is a curve in $\bP^2_k$ defined by 
\begin{equation}
C(\mu)\ :\ x^3_0+x^3_1+x^3_2-3\mu x_0x_1x_2=0 
\end{equation}
for some $\mu\in k$, or $\mu=\infty$ (in which case we understand that
$C(\infty)$ is the curve 
defined by $x_0x_1x_2=0$).  We see
\begin{enumerate}
\item[(i)] 
$C(\mu)$ is nonsingular elliptic 
for $\mu\neq\infty, 1, \zeta_3, \zeta_3^2$, 
\item[(ii)] 
$C(\mu)$ is a 3-gon for $\mu=\infty, 1, \zeta_3, \zeta_3^2$,
\item[(iii)] any $C(\mu)$ 
contains $K$, 
which is independent of $\mu$, 
$$K=\left\{[0,1,-\zeta_3^k], [-\zeta_3^k,0,1], 
[1,-\zeta_3^k,0]; k=0,1,2\right\},$$ 
\item[(iv)] $K$ is identified with 
the group of $3$-division points by choosing $[0,1,-1]$ as the zero, 
so $K\simeq (\bZ/3\bZ)^2$ as groups,
\item[(v)] if $k=\bC$, any Hesse cubic is 
the image of a complex torus $E(\omega):=\bC/\bZ+\bZ\omega$  
by (slightly modified) theta functions $\vartheta_k$ of level 3 
(see Subsec.~\ref{subsec:theta functions}), and then 
$K$ is the image of the $3$-division points 
$\langle \frac{1}{3},\frac{\omega}{3}\rangle$ of $E(\omega)$. 
\end{enumerate}

\subsection{Theta functions}
\label{subsec:theta functions}
We will explain 
Subsec.~\ref{subsec:Hesse cubics}~(v) in more detail. 
First let us recall standard (resp. modified) 
theta functions of level 3 
on $E(\omega)$ :
\begin{align*}\theta_{k}(\omega,z)&
=\sum_{m\in \bZ} q^{(3m+k)^2}w^{3m+k},\quad \text{resp.}\\
\vartheta_k(\omega,z)&=\theta_{k}(\omega,z+\frac{1-\omega}{2})
\end{align*}where $q=e^{2\pi i\omega/6}$, 
$w=e^{2\pi iz}$. They satisfy the transformation relation :
\begin{align*}
\theta_k(\omega,z+\frac{a+b\omega}{3})&=\zeta_3^{ak}(q^bw)^{-b}
\theta_{k+b}(\omega,z),\\
\vartheta_k(\omega,z+\frac{a+b\omega}{3})&=\zeta_3^{ak}(q^{b-3}(-w))^{-b}
\vartheta_{k+b}(\omega,z).
\end{align*}

We define a mapping $\vartheta:E(\omega)\to \bP^2$ by  
\begin{gather*}
\vartheta(\omega,z):=[\vartheta_0,\vartheta_1,\vartheta_2].
\end{gather*}

Let us check the second half 
of Subsec.~\ref{subsec:Hesse cubics}~(v). For it, 
we rewrite 
\begin{align*}
\vartheta_0(\omega,z)&=\sum_{m\in\bZ}q^{9m^2-9m}(-w)^{3m},\\
\vartheta_1(\omega,z)&=\sum_{m\in\bZ}q^{9m^2-3m-2}(-w)^{3m+1},\\
\vartheta_2(\omega,z)&=\sum_{m\in\bZ}q^{9m^2+3m-2}(-w)^{3m+2}.
\end{align*}

Then we check  $\vartheta(\omega,\frac{\ell}{3})=[0,1,-\zeta_3^{\ell}]$ and 
$\vartheta(\omega,\frac{\omega}{3})=[1,-1,0]$. 
First we prove $\vartheta_0(\omega,\frac{\ell}{3})=0$. In fact, we see 
\begin{align*}
\vartheta_0(\omega,\frac{\ell}{3})&
=\sum_{m\in\bZ}q^{9m^2-9m}(-1)^{3m}\\
&=\sum_{m\in\bZ}q^{9(-m+1)^2-9(-m+1)}(-1)^{3(-m+1)}\\
&=\sum_{m\in\bZ}q^{9m^2-9m}(-1)^{-3m+3}
=-\vartheta_0(\omega,\frac{\ell}{3}),
\end{align*}whence $\vartheta_0(\omega,\frac{\ell}{3})=0$. Moreover
\begin{align*}
\vartheta_1(\omega,\frac{\ell}{3})&=\zeta_3^{\ell}
\sum_{m\in\bZ}q^{9m^2-3m-2}(-1)^{3m+1},\\
\vartheta_2(\omega,\frac{\ell}{3})&=\zeta_3^{2\ell}
\sum_{m\in\bZ}q^{9m^2+3m-2}(-1)^{3m}\\
&=\zeta_3^{2\ell}
\sum_{m\in\bZ}q^{9m^2-3m-2}(-1)^{3m}=
-\zeta_3^{\ell}
\vartheta_1(\omega,\frac{\ell}{3}).
\end{align*} 

$\vartheta(\omega,\frac{\omega}{3})=[1,-1,0]$ is proved similarly.

\subsection{The moduli space of Hesse cubics ---  
the Stone-age (Neolithic) level structure}\label{subsec:Stone age level}
With the same notation as in Subsec.~\ref{subsec:Hesse cubics}, 
consider the moduli space $SQ^{\NL}_{1,3}$ of the pairs $(C(\mu),K)$
over any ring $k\ni 1/3$ and $\zeta_3$. 
\par
\begin{subdefn}Any pair $(C(\mu),K)$ is called 
{\em a Hesse cubic with Neolithic level-3 structure.} 
Let $(C(\mu),K)$ and $(C(\mu'),K)$ be two pairs of Hesse cubics with 
Neolithic level-3 structure. 
We define $(C(\mu),K)\simeq (C(\mu'),K)$ to be isomorphic  
if there exists an isomorphism 
$f:C(\mu)\to C(\mu')$ with $f_{|K}=\id_K$. 
\end{subdefn}

\begin{claim}Let $SQ^{\NL}_{1,3}$ be the set of 
isomorphism classes of \linebreak 
$(C(\mu),K)$, and $A^{\NL}_{1,3}$ 
the subset of $SQ^{\NL}_{1,3}$ consisting of smooth $C(\mu)$. Then 
\begin{enumerate}
\item[(i)] 
if $(C(\mu),K)\simeq (C(\mu'),K)$, then $\mu=\mu'$, 
\item[(ii)] $SQ^{\NL}_{1,3}$ has a natural scheme structure:
$$SQ^{\NL}_{1,3}\simeq \bP^1_k=\Proj k[\mu_0,\mu_1],$$
\item[(iii)] 
this compactifies the moduli $A^{\NL}_{1,3}$ of smooth Hesse cubics:
$$A^{\NL}_{1,3}\simeq \Spec k[\mu,\frac{1}{\mu^3-1}],\quad \mu=\mu_1/\mu_0,$$
where $A^{\NL}_{1,3}(k)=\{C(\mu); \text{smooth}, \mu\in k\}$ if $k$ is  
a closed field, 
\item[(iv)] the universal Hesse cubic over 
$SQ^{\NL}_{1,3}$ is given by
\begin{equation}\label{eq:universal Hesse cubic}
 \mu_0(x^3_0+x^3_1+x^3_2)-3\mu_1 x_0x_1x_2=0. 
\end{equation}
\end{enumerate}
\end{claim}

\begin{proof}[Proof of (i)]
We prove (i). Suppose we are given an
isomorphism 
$$f : (C(\mu),K)\simeq (C(\mu'),K).$$  
Since any 3 points $x,y$ and $z\in K$ 
with $x+y+z=0$ are on a line $\ell_{x,y,z}$ of $\bP^2$,  
we have $\ell_{x,y,z}\cap C(\mu)=\{x,y,z\}$ and 
$f^*\ell_{x,y,z}=\ell_{x,y,z}$ 
as divisors of $C(\mu)$. 
Hence $f$ is given by  a $3\times 3$ matrix $A$. \par
We shall prove that $A$ is a scalar and $f=\id$.
In fact, any line $\ell_{x,y}$ connecting two points 
$x, y\in K$ is fixed by $f$. Since 
the line $x_0=0$ connects  $[0,1,-1]$ and $[0,1,-\zeta_3]$, 
it is fixed by $f$. 
Similarly the lines $x_1=0$ and $x_2=0$ are fixed by $f$, 
whence $f^*(x_i)=a_ix_i$ $(i=0,1,2)$ for 
some $a_i\neq 0$. Thus $A$ is diagonal. 
Since $[0,1,-1]$ and $[-1,0,1]$ are fixed, 
we have $a_0=a_1=a_2$, hence $A$ is scalar and $f=\id$, $\mu=\mu'$. 
\par
We do not give proofs of (ii)-(iv) here because there are 
complicated arguments to prove rigorously.  
\end{proof}

\subsection{The moduli space of smooth cubics ---  
 classical level structure}\label{subsec:classical level}
Consider the (fine) moduli space 
of smooth cubics over an algebraically closed field $k\ni 1/3$. 
\begin{subdefn}
Let $K=(\bZ/3\bZ)^{\oplus 2}$, $e_i$ a standard basis of $K$. 
Let $e_K:K\times K\to\mu_3$ 
be a standard symplectic form of $K$: in other words, $e_K$ is 
(multiplicatively) alternating 
and bilinear  such that
$$e_K(e_1,e_2)=e_K(e_2,e_1)^{-1}=\zeta_3,\ e_K(e_i,e_i)=1.$$
\end{subdefn}

Let $C$ be a smooth cubic with zero $O$, 
$C[3]=\ker(3\id_C)$ the group of 3-division points and 
$e_C$ the Weil pairing of $C$ (see \cite[pp.~95--102]{Silverman86}), that is, 
$$e_C:C[3]\times C[3]\to \mu_3\quad\text{alternating nondegenerate bilinear,}$$
(see \ref{subsec:noncommutative level}~(v)). By \cite[pp.~294--295]{Mumford66},
there exists 
a symplectic (group) isomorphism 
$$\iota: (C[3],e_C)\to (K,e_K).$$  

In what follows, we identify $C(\mu)[3]$ with $K$ by
\begin{equation}\label{eq:3 divisions}
O=[0,1,-1],\ e_1=[0,1,-\zeta_3],\ e_2=[1,-1,0].
\end{equation}

\begin{subdefn}The triple $(C,C[3],\iota)$ is called {\em a (planar) cubic 
with classical level-3 structure.} We define
$(C,C[3],\iota)\simeq (C',C'[3],\iota')$ to be isomorphic iff 
there exists an isomorphism $f:C\to C'$ such that
 $f_{|C[3]}:C[3]\to C'[3]$ is a symplectic (group) isomorphism subject to 
 $\iota'\cdot f=\iota$.
\end{subdefn}

\begin{claim}
\label{claim:A13}
Let $A^{\CL}_{1,3}$ be the set of 
isomorphism classes of \linebreak $(C,C[3],\iota)$. Then 
\begin{enumerate}
\item[(i)] any $(C,C[3],\iota)$ is isomorphic to $(C(\mu),C(\mu)[3],\iota)$ 
for a unique $\mu$,
\item[(ii)]
$(C(\mu),K,\id_K)\in A^{\CL}_{1,3}$ via (\ref{eq:3 divisions}), and 
\begin{align*}
A^{\CL}_{1,3}
&=\{(C(\mu),K,\id_K); \text{a smooth Hesse cubic}\}\\
&\simeq\Spec k[\mu,\frac{1}{\mu^3-1}],
\end{align*}
\item[(iii)] we define   
$SQ^{\CL}_{1,3}$ to be the union of $A^{\CL}_{1,3}$ and 
3-gons in Subsec.~\ref{subsec:Hesse cubics}~(ii) :
\begin{align*}
SQ^{\CL}_{1,3}:&=\{(C,C[3],\iota);C\ 
\text{smooth elliptic or a 3-gon}\}/\text{isom.}\\
&=\{(C(\mu),K,\id_K); \text{a Hesse cubic}\}\\
&\simeq\Proj k[\mu_0,\mu_1],
\end{align*} 
\item[(v)]$A^{\CL}_{1,3}\simeq A^{\NL}_{1,3}$ and 
$SQ^{\CL}_{1,3}\simeq SQ^{\NL}_{1,3}$ over $k$.\theoremtextend
\end{enumerate}
\end{claim}
\begin{proof}[Proof of (i)] 
We prove the uniqueness of $\mu$. 
Suppose that $$f:(C(\mu),K,\id_K)\to 
(C(\mu'),K,\id_K)$$ is an isomorphism. Then $f\in\GL(3)$. 
Since $\id_K\cdot f_{|K}=\id_K$ by $\iota'\cdot f=\iota$, 
we have $f_{|K}=\id_K$.  
Hence $f=\id\in\PGL(3)$, $\mu=\mu'$ by 
Subsec.~\ref{subsec:Stone age level}~(iv). 
See also Lemma~\ref{lemma:G(3) uniqueness} 
and Lemma~\ref{sublemma:uniqueness of rigid PSQAS}.
\end{proof}

\section{Non-commutative level structure} 
\label{sec:noncomm level}
\subsection{For constructing a separated moduli}
\label{subsec:for constructing separated moduli}
If we keep naively using the same definition 
of level structures as in Subsec.~\ref{subsec:classical level} 
in higher dimension, then the complete 
moduli will be roughly the moduli 
of the triples $(Z,\ker(\lambda(L)),\iota_Z)$ 
similar to $(C,C[3],\iota)$ 
$$\iota_Z:\ker(\lambda(L))\simeq K\ \text{for some $K$}.\quad$$
However then we will 
have nonseparated moduli spaces in general. 
The details will 
be explained in Subsec.~\ref{subsec:nonseparatedness}. \par
To construct a separated moduli, we need to find outside $C$ 
an alternative for $C[3]$ embedded in $C$. 
The group $C[3]$, hence 
$x\in K=(\bZ/3\bZ)^{\oplus 2}$ acts on $C$ by translation $T_x:C\to C$.
Though the action of $K$ on $C$ 
cannot be lifted to $L$ as an action of the group $K$, 
the action of any individual element $x$ of $K$ can be lifted to 
a line bundle automorphism $\tau_x$ of $L$.
In general $\tau_x$ and $\tau_y$ $(x,y\in K)$
do not commute so $T_x\mapsto \tau_x$ fails to be a group homomorphism.  
However it turns out that the non-commutative group 
generated by all individual liftings $\tau_x$ plays 
the role of an alternative for $C[3]$ embedded in $C$. 
This leads us to the notion of {\em a level-$G(3)$ structure}, say, 
{\em a non-commutative level structure,} where $G(3)$ is the Heisenberg group associated to $K$.
\begin{subrem} 
Since any elliptic curve with level-$G(3)$ structure 
has a section over $\bZ[\zeta_3,1/3]$ 
by \cite{NT01}, the level-$G(3)$ structure is 
a $\vartheta$-structure  of 
\cite[II, p.78]{Mumford67} and vice versa.
A level $\cG_H$ (or $G_H$-)structure is not always a $\vartheta$-structure 
by \cite{NT01} when $H=\bZ/n\bZ$ for $n$ even in 
Definition~\ref{defn:Heisenberg group}. 
\end{subrem}

\begin{defn}Let $k$ be an algebraically 
closed field $k\ni 1/3$. 
Then 
\begin{enumerate}
\item[(i)]
let $C$ be any smooth cubic with zero $O$, and 
$L:=O_C(1)$ the hyperplane bundle. 
Let $\lambda(L):C\to C^{\vee}:=\Pic^0(C)\simeq C$ be 
the map $x\to T_x^*L\otimes L^{-1}$, called the polarization morphism,  
where we see $\lambda(L)=3\id_C$, 
\item[(ii)]
let $K:=C[3]=\ker\lambda(L)\simeq (\bZ/3\bZ)^{\oplus 2}$, and
$e_K:K\times K\to \mu_3$ the Weil pairing of $C$. 
If $C=C(\mu)$ 
and $O=[0,1,-1]\in C(\mu)$. 
Then $K=\ker(\lambda(L))$ is 
the same as in Subsec. 2.1 (iii).
\end{enumerate}
\end{defn}

\subsection{Non-commutative interpretation of Hesse cubics}
\label{subsec:noncommutative level}
First we shall re-interpret the group $C[3]$ 
of 3-division points  of 
Hesse cubics 
in the non-commutative way as follows. \par
Any  translation $T_x$ by $x\in K$ is lifted 
to $\gamma_x\in \GL(V)$, 
so that $$e_K(x,y)=[\gamma_x,\gamma_y]\in\mu_3,$$
where $V=H^0(C,O_C(1))
=H^0(\bP^2,O_{\bP^2}(1))$. To be more precise, 
\begin{enumerate}
\item[(i)] we define $\sigma$ and $\tau$ by 
 $\sigma(x_k)=\zeta^k_3x_k$, $\tau(x_k)=x_{k+1}$\ $(k=0,1,2)$, where
their matrix forms are given by
\begin{align*}
\sigma&=\begin{pmatrix}
1&0&0\\
0&\zeta_3&0\\
0&0&\zeta_3^2
\end{pmatrix},\quad 
\tau=\begin{pmatrix}
0&0&1\\
1&0&0\\
0&1&0
\end{pmatrix},
\end{align*}
\item[(ii)] $\sigma$ is induced from the translation by $1/3$ 
because $x_k=\theta_k$ by Subsec.~2.1~(v) and 
$$\theta_k(z+1/3)=\zeta_3^k\theta_k(z),$$ 
\item[(iii)]
$\tau$ is induced from the translation by  $\omega/3$
because $$[\theta_0,\theta_1,\theta_2](z+\omega/3)
=[\theta_1,\theta_2,\theta_0](z),$$
\item[(iv)] $[\sigma,\tau]=\zeta_3$, that is, 
$\sigma$ and $\tau$ do not commute,
\begin{align*}
\sigma\tau&=\begin{pmatrix}
0&0&1\\
\zeta_3&0&0\\
0&\zeta_3^2&0
\end{pmatrix},\quad 
\tau\sigma=\begin{pmatrix}
0&0&\zeta_3^2\\
1&0&0\\
0&\zeta_3&0
\end{pmatrix}.
\end{align*}
\end{enumerate}

\begin{lemma}\label{lemma:G3 irreducible}
Let $G(3):=\langle \sigma,\tau\rangle$ be 
the group generated by $\sigma$ and $\tau$. 
Then it is a finite group of order 27. Let 
$V=H^0(\bP^2,O_{\bP^2}(1))=\{x_0,x_1,x_2\}$. Then 
$V$ is an irreducible $G(3)$-module of weight one, where 
''weight one'' means that $a\in \mu_3$ (center) acts by $a\id_V$.
\end{lemma}
\begin{proof}The first assertion is clear. 
See \cite[Proposition~3, p.~309]{Mumford66} or
\cite[Lemma~4.4]{Nakamura10} for the second assertion.
\end{proof}

The action of $G(3)$ on $H^0(C,L)$ 
is a special case of more general 
Schr\"odinger representations defined below.

\begin{defn}\label{defn:Heisenberg group} We define
$G(K)=G_H$ (resp. $\cG(K)=\cG_H$) 
to be {\em the Heisenberg group 
(finite resp. infinite)} and 
$U_H$ {\em the Schr\"odinger representation of $G_H$} as follows:
\begin{align*}
H&=H(e):=\bigoplus_{i=1}^g(\bZ/e_i\bZ), e_i|e_{i+1}, 
N=|H|=\prod_{i=1}^ge_i,  \\
K&=H\oplus H^{\vee}, e_{\min}(K)=e_{\min}(H):=e_1,\\
G_H&=\{(a,z,\alpha);a\in \mu_N, z\in H, \alpha\in H^{\vee}\},\\
\cG_H&=\{(a,z,\alpha);a\in \bG_m, z\in H, \alpha\in H^{\vee}\},\\
(a,z,\, &\alpha)\cdot (b,w,\beta)=(ab\beta(z),z+w,\alpha+\beta),\\
V:&=V_H=\cO_N[H^{\vee}]=\bigoplus_{\mu\in H^{\vee}}\cO_N\, v(\mu),\\
U_H&(a,z,\alpha) v(\gamma)=a\gamma(z)v(\alpha+\gamma).
\end{align*}  Here $\cO=\cO_N=\bZ[\zeta_N,1/N]$, and 
$v(\mu)$ ($\mu\in H^{\vee})$ is 
a free $\cO_N$-basis of $V_H$.  The group homomorphism  
$U_H$, from $G_H$ or $\cG_H$ to $\End(V)$, 
is called {\em Schr\"odinger representation}.
We note 
\begin{align*}
1\to &\mu_N\to G_H\to K\to 0 \quad \text{(exact)}\\
1\to &\bG_m\to \cG_H\to K\to 0 \quad \text{(exact)}.
\end{align*}
\end{defn}

\begin{example}\label{example:Hesse P(VH)}
For Hesse cubics, $\cO:=\bZ[\zeta_3,1/3]$,
 $H=H^{\vee}=\bZ/3\bZ$,  we identify $G(3)$ with $G_H$; to be precise, 
$G(3)=U_H(G_H)$ and  
\begin{gather*}
\sigma=U_H(1,1,0),\quad \tau=U_H(1,0,1),\quad N=3.\\
V_H=\cO[H^{\vee}]=
 \bigoplus_{k=0}^2\cO\cdot v(k).
\end{gather*} 

Let $\bP^2=\bP(V_H)$. Then 
$V_H$ is identified with $H^0(C,O_C(1))=H^0(\bP^2,O_{\bP^2}(1))$ 
by the map $v(k)\mapsto x_k$ in Lemma~\ref{lemma:G3 irreducible}.  
\end{example}

\begin{lemma}
\label{lemma:irred repres UH} 
$V_H$ is an irreducible $\cG_H$-$\cO_N$-module 
(an irreducible $G_H$-$\cO_N$-module)
of weight one, unique up to equivalence. Any 
$\cG_H$-$\cO_N$-module $W$ 
(resp. any $G_H$-$\cO_N$-module) of finite rank is a direct sum of 
$V_H$ if $W$ is of weight one: that is, any element $a$ in
the center $\bG_m$ (resp. $\mu_N$) acts on $W$ by 
scalar multiplication $a\id_W$.
\end{lemma}
\begin{proof}See Lemma~\ref{sublemma:weight 1 mod N module} 
and \cite[Lemma~4.4]{Nakamura10}.
\end{proof}

\begin{lemma}\label{lemma:lemma of Schur}$($Schur's lemma$)$
Let $R$ be a commutative algebra with $1/N$ and $\zeta_N$. 
Let $V_1$ and $V_2$ be $R$-free $G_H$-modules 
of finite rank of weight one. If $V_1$ and $V_2$ are 
irreducible $G_H$-modules, and if 
$f:V_1\to V_2$ and $g:V_1\to V_2$ are 
$\cG_H$-isomorphisms, then there exists 
a unit $c\in R^{\times}$ such that $f=cg$.
\end{lemma}
\begin{proof}
See \cite[Lemma~4.5]{Nakamura10}.
\end{proof}

\subsection{New formulation of the moduli problem}
Let $k$ be any ring such that $k\ni \zeta_3, 1/3$ 
and $K=(\bZ/3\bZ)^{\oplus 2}$. 
Let $C$ be any smooth cubic, $L=O_C(1)$ 
the line bundle viewed as a scheme over $C$. 
By \cite[p.~295]{Mumford66} (see also \cite[Lemma~7.6]{Nakamura99}) 
the pair $(C,L)$ of schemes 
 has a $G(3)$-action lifting the translation action by $C[3]$
$$\tau:G(3)\times (C,L)\to (C,L).$$
Using this $G(3)$-action, we define new level-3 structure.
In a word, 
\begin{enumerate}
\item[$\circ$] classical level-3 structure 
$=$ to fix the 3-division points $K$
\item[$\circ$] new level-3 structure $=$ to fix the matrix form 
of the action of $G(3)$ on 
$V\simeq H^0(C,L)$.
\end{enumerate}
\begin{defn} We define   
$(C,\psi,\tau)$ to be {\em a (planar) cubic with level-$G(3)$ structure} 
(or {\em a level-$G(3)$ cubic}) if 
\begin{enumerate}
\item[(i)] $(C,L)$ is a planar cubic with $L=O_C(1)$,
\item[(ii)] $\tau$ is a $G(3)$-action of weight one on the pair $(C,L)$: 
that is, $\tau(a)$ acts by $(\id_C,a\id_L)$ for $a\in\mu_3$, 
the center of $G(3)$, 
\item[(iii)] $\psi:C\to \bP(V_H)$ is the inclusion, and 
 $$(\psi,\Psi):(C,L)\to (\bP(V_H),\bH)$$ 
is  a $G(3)$-equivariant morphism by $\tau$ 
where 
$\bH$ is the hyperplane bundle of $\bP(V_H)$ 
and $\Psi:L=\psi^*\bH\to \bH$ the natural bundle morphism.
That is, 
\begin{equation}\label{eq:G3 equivariance}
(\psi,\Psi)\circ\tau(g)=S(g)\circ(\psi,\Psi)\ \text{for any $g\in G(3)$}
\end{equation}
 with the notation 
in Subsec.~\ref{subsec:Linearization of O(1)}. \theoremtextend
\end{enumerate}
\end{defn}

In what follows, we denote $(\psi,\Psi)$ 
simply by $\psi$ if no confusion is possible 
because $\Psi$ is uniquely determined by $\psi$. 
We denote (\ref{eq:G3 equivariance}) by
\begin{equation}\label{eq:simple equivariance}
\psi\tau(g)=S(g)\psi,\ \text{or}\ \psi\tau=S\psi.
\end{equation}

\begin{defn}
Two cubics $(C,\psi,\tau)$ and $(C',\psi',\tau')$ 
with level-$G(3)$ structure 
are defined to be isomorphic iff 
there exists an isomorphism 
$$(f,F):(C,L)\to (C',L')
$$
such that 
\begin{enumerate}
\item[(i)]$\psi'\cdot (f,F) = \psi$,  
\item[(ii)]$(f,F)$ is a $G(3)$-isomorphism, that is, 
$(f,F)\tau(g)=\tau'(g)(f,F)$ for any $g\in G(3)$.
\end{enumerate}
\end{defn}

\begin{lemma}\label{lemma:G(3) uniqueness}
Any Hesse cubic $(C(\mu),i,U_H)$ with $i$ the inclusion of $C$ into $\bP(V_H)$ 
is a level-$G(3)$ cubic.  
Moreover any level-$G(3)$ cubic $(C,\psi,\tau)$ is isomorphic  
to a unique Hesse cubic 
$(C(\mu),i,U_H)$.
\end{lemma}
\begin{proof} Let $\bP^2$ be $\bP(V_H)$ and 
$\bH$ the hyperplane bundle of $\bP^2$. 
$U_H$ induces an action on 
$H^0(\bP^2,O_{\bP^2}(1))=V_H$ 
by Claim~\ref{claim:defn of rhoL}, 
which we denote 
by $H^0(U_H,O_{\bP^2}(1))$. This is the same 
as the action $U_H$ on $V_H$ in Definition~\ref{defn:Heisenberg group}. 
In fact, by Subsec.~\ref{subsec:Linearization of O(1)} and 
Remark~\ref{rem:rho=rho}, $U_H$ induces an action of $G(3)$ 
on the pair $(\bP^2,\bH)$, which also induces an action of $G(3)$ on 
$H^0(\bP^2,O_{\bP^2}(1))=V_H$. 
This is the same as $U_H$ as is shown  
in Remark~\ref{rem:rho=rho}. 
\par
Let $O_{C(\mu)}(1)=O_{\bP^2}(1)\otimes O_{C(\mu)}$ and 
$\bH_{C(\mu)}=\bH\times_{\bP^2} C(\mu)$.  
Since $C(\mu)$ is $G(3)$-stable, 
$G(3)$ acts on the pair $(C(\mu),\bH_{C(\mu)})$ 
by Claim~\ref{claim:G equiv embedding}. 
Denoting the action of $G(3)$ on 
$\bH_{C(\mu)}$ by the same letter $U_H$, we see that 
$(C(\mu),i,U_H)$ is a level-$G(3)$ structure. \par
Hence $H^0(C(\mu),O_{C(\mu)}(1))$ 
admits a $G(3)$-action, 
which we denote by 
$H^0(U_H,O_{C(\mu)}(1))$.  
Since $H^0(C(\mu),O_{C(\mu)}(1))=H^0(\bP^2,O_{\bP^2}(1))=V_H$ 
by restriction, 
we can identify 
$H^0(U_H,O_{C(\mu)}(1))$ with $H^0(U_H,O_{\bP^2}(1))$ on $V_H$ 
in a canonical manner. 
Thus we have a canonical identification
$$H^0(U_H,O_{C(\mu)}(1))=H^0(U_H,O_{\bP^2}(1))=U_H. 
$$

By Lemma~\ref{sublemma:uniqueness of rigid PSQAS}, 
any $(C,\psi,\tau)$ is isomorphic  
to some Hesse cubic 
$(C(\mu),i,U_H)$. Here we prove the uniqueness of it only. 
This is a new proof of Claim~\ref{claim:A13}~(ii).
Suppose $(C(\mu),i,U_H)\simeq (C(\mu'),i,U_H)$. Let 
$h:C(\mu)\to C(\mu')$ be a $G(3)$-isomorphism. 
Since $h$ is linear (as is shown easily), $h$ induces an automorphism 
of $(\bP^2,O_{\bP^2}(1))$ (also denoted $h$) 
so that we have a commutative diagram
\begin{equation*}
\CD
H^0(\bP^2,O_{\bP^2}(1))=V_H @>{H^0(h^*)}>> 
H^0(\bP^2,O_{\bP^2}(1))=V_H \\
@VV{\left|\right|}V  @VV{\left|\right|}V \\
H^0(C(\mu'),O_{C(\mu')}(1)) @>{H^0(h^*)}>> H^0(C(\mu),O_{C(\mu)}(1)), \\
@VV{H^0(U_H(g),O_{C(\mu')}(1))}V  @VV{H^0(U_H(g),O_{C(\mu)}(1))}V \\
H^0(C(\mu'),O_{C(\mu')}(1)) @>{H^0(h^*)}>> H^0(C(\mu),O_{C(\mu)}(1)), \\
\endCD
\end{equation*}whence
$$H^0(U_H(g),O_{C(\mu)}(1))H^0(h^*)=H^0(h^*)H^0(U_H(g),O_{C(\mu')}(1))$$ 
for any $g\in G(3)$. By canonically 
identifying $H^0(U_H, O_{C(\mu)}(1))$ with $U_H$ on $V_H$, 
we have 
$$U_H(g)H^0(h^*)=H^0(h^*)U_H(g)\in\End(V_H)$$  
for any $g\in G(3)$, where we also regard 
$H^0(h^*)\in\End(V_H)$. 
Since $U_H$ is irreducible,  
$H^0(h^*)$ is a scalar by Schur's lemma. 
Hence $H^0(h^*)=\id_{V_H}\in\PGL(V_H)$, $h=\id_{\bP(V_H)}$, $C(\mu)=C(\mu')$, 
$\mu=\mu'$. 
\end{proof}
\begin{rem}In the proof of Lemma~\ref{lemma:G(3) uniqueness}, we 
canonically identified all the vector spaces 
involved to simplify the argument. 
This argument will be made much clearer 
by using $\rho(\phi,\tau)$ in Definitions~\ref{subdefn:rho(phi,tau)} 
and ~\ref{subdefn:generalization of Hesse cubic}.  
See Lemma~\ref{sublemma:uniqueness of rigid PSQAS}.
\end{rem}

\begin{prop}Over $\bZ[\zeta_3,1/3]$, 
\begin{align*}
SQ_{1,3}:&=\{(C,\psi,\tau);\text{a level-$G(3)$ cubic}\}/\text{\rm isom.}\\
&=\{(C(\mu),i,U_H)\}/\text{\rm isom.}=\{\mu\in\bP^1\}.
\end{align*}
\end{prop}
\begin{proof}Clear from Lemma~\ref{lemma:G(3) uniqueness} and 
Lemma~\ref{sublemma:uniqueness of rigid PSQAS}.
\end{proof}

It is this   
{\em level-$G(3)$ structure}
 that we can generalize into higher dimension so that we may 
 obtain a separated moduli. 

\begin{rem}
Suppose $k$ is algebraically closed with $1/3$. Let $K=(\bZ/3\bZ)^{\oplus 2}$. 
Let $C$ be any cubic, and 
$C[3]=\ker(3\id_C)$ by choosing the zero $O\in C(k)$. Any level-$G(3)$ 
structure $(C,\phi,\tau)$ 
gives rise to a classical level-3 structure $(C,C[3],\iota)$ as follows. 
First we note 
$$C[3]=G(3)\cdot O. 
$$

Let $\pi:G(3)\to K=G(3)/[G(3),G(3)]$ be the natural homomorphism. 
We define $\iota:K\to C$ by 
$$\iota(g\cdot O):=\pi(g).$$ Then $(C,C[3],\iota)$ is a 
classical level-3 structure. In fact, 
since $e_K(x,y)=[\gamma_x,\gamma_y]$ for a lifting $\gamma_x$ of $x$, 
we have $e_K(1/3,\omega/3)=[\sigma,\tau]=\zeta_3$. Hence $\pi$ defines 
a symplectic isomorphism $\iota:C[3]\to K$.  Thus we see
\begin{align*}
SQ_{1,3}(k)&=SQ^{\CL}_{1,3}(k).
\end{align*}
By 
\cite{NT01} $SQ_{1,3}\simeq SQ^{\CL}_{1,3}$ over $\bZ[1/3,\zeta_3]$.
See \cite{NT01} for the detail.
\end{rem}

\section{PSQAS and TSQAS}
\label{sec:PSQAS TSQAS}

\subsection{Goal}
Our goal of constructing a compactification 
of the moduli space of abelian varieties is achieved by  
\begin{enumerate}
\item[(i)] finding limit objects 
(two kinds of nice degenerate abelian schemes called PSQAS and TSQAS) 
(Theorems~\ref{thm:stable reduction} and 
~\ref{thm:refined stable reduction}),
\item[(ii)] 
constructing the moduli $SQ_{g,K}$ as a projective scheme 
(Section~\ref{sec:moduli AgK SQgK}),
\item[(iii)] proving that any point of $SQ_{g,K}$ is the isomorphism 
class of a nice degenerate abelian scheme (PSQAS)
$(Q_0,\phi_0,\tau_0)$ with level-$\cG_H$ structure 
(Section~\ref{sec:moduli AgK SQgK}, 
Theorems~\ref{thm:SQgK} and \ref{thm:fine moduli SQgK}).
\end{enumerate}

We recall a basic lemma from \cite{Mumford12}.
\begin{lemma}\label{lemma:chara GH action}
Let $k$ be an algebraically closed field 
with $k\ni 1/N$ and $H$ 
a finite Abelian group with $|H|=N$. 
Let $(A,L)$ be an abelian variety over $k$ 
with $L$ an ample line bundle, $\lambda(L):A\to A^{\vee}$ 
the polarization morphism (sending $x\mapsto T_x^*L\otimes L^{-1}$)
and $\cG(A,L)$ 
the group of bundle automorphisms $g$
of $L$ over $A$ inducing translations of $A$.\par
Suppose $\ker(\lambda(L))\simeq K:=H\oplus H^{\vee}$. 
Then $\cG(A,L)\simeq L^{\times}_{\ker(\lambda(L))}\simeq \cG_H$, 
and any $g\in\cG(A,L)$ induces a translation 
of $A$ by some element of $\ker(\lambda(L))$
where
$L^{\times}$ is the complement of the zero section in 
the line bundle $L$, and $L^{\times}_{\ker(\lambda(L))}$ 
is the pullback (restriction) of it to $\ker(\lambda(L))$.
\end{lemma}
\begin{proof}
See \cite[pp.~294--295]{Mumford66} 
and \cite[pp.~115-117, pp.204-211]{Mumford12}. 
\end{proof}

\subsection{Limit objects}
\label{subsec:limit obj} 
We wish to consider limits of abelian varieties.\par

Let $R$ be a complete discrete valuation ring (CDVR), 
and $k(\eta)$ the fraction field of $R$ and 
$k(0):=R/I$ the residue field. 
Suppose we are given 
an abelian scheme $(G_{\eta},\cL_{\eta})$ over $k(\eta)$ and 
the polarization morphism 
$$\lambda(\cL_{\eta}):G_{\eta}\to 
G^t_{\eta}:=\Pic^0(G_{\eta}).$$ 
Let 
$$K_{\eta}=\ker(\lambda(\cL_{\eta})),\quad  
\cG(K_{\eta}):=
\cG(G_{\eta},\cL_{\eta})\simeq (\cL_{\eta}^{\times})_{|K_{\eta}},$$  
 where $\cG(G_{\eta},\cL_{\eta})$ is by definition the group 
of bundle automorphisms of $\cL_{\eta}$  over $G_{\eta}$
 which induce translations of $G_{\eta}$. 
See Lemma~\ref{lemma:chara GH action}.
 \par

For simplicity, in what follows, we assume 
\begin{equation}\label{eq:char prime to |K|}
\text{the field  
$k(0)$ contains $1/{|K_{\eta}|}$.}
\end{equation}

We apply Lemma~\ref{lemma:chara GH action} to $(G_{\eta},\cL_{\eta})$. 
\begin{lemma}Assume (\ref{eq:char prime to |K|}).
Then by some base change of $R$ if necessary, 
there exists a finite symplectic Abelian group $K$ such that 
the diagram is commutative with exact rows: 
\begin{equation*}
\CD
1@>>> \bG_m @>>> \cG(K_{\eta})@>>> K_{\eta}@>>> 0\\
@. @VV{\id.}V @VV{\simeq}V @VV{\simeq}V @.\\
1@>>> \bG_m @>>> \cG_H@>>> H\oplus H^{\vee}@>>> 0.
\endCD
\end{equation*}
\end{lemma}


\begin{thm}\label{thm:stable reduction}
{\rm (Stable reduction theorem) }
{\rm (\cite{AN99})}
For an abelian scheme $(G_{\eta},\cL_{\eta})$ and 
a polarization morphism $\lambda(\cL_{\eta}):G_{\eta}\to G^t_{\eta}$ over 
$k(\eta)$, there exist a flat projective scheme 
$(P,\cL_P)$ (TSQAS) 
over $R$, by a finite base change if necessary,
such that 
\begin{enumerate}
\item $(P_{\eta},\cL_{\eta})
\simeq (G_{\eta},\cL_{\eta})$,
\item $(P,\cL_P)$ is normal with $\cL_P$ ample, in fact, 
$P$ is explicitly given,
\item $P_0$ is reduced and Gorenstein with trivial dualizing sheaf.
\end{enumerate}
\end{thm}

The following is a refined version of the above. 
\begin{thm}\label{thm:refined stable reduction}
{\rm (Refined stable reduction theorem) }
{\rm (\cite[p.~703]{Nakamura99}, \cite[p.~98]{Nakamura10})}
For an abelian scheme $(G_{\eta},\cL_{\eta})$ and 
a polarization morphism $\lambda(\cL_{\eta}):G_{\eta}\to G^t_{\eta}$ over 
$k(\eta)$ such that $K_{\eta}\simeq K$, there exist flat projective schemes 
$(Q,\cL_Q)$ (PSQAS) and $(P,\cL_P)$ (TSQAS) 
over $R$, by a finite base change if necessary,
such that 
\begin{enumerate}
\item $(Q_{\eta},\cL_{\eta})\simeq (P_{\eta},\cL_{\eta})
\simeq (G_{\eta},\cL_{\eta})$,
\item $(P,\cL_P)$ is the normalization of $(Q,\cL_Q)$,
\item $P_0$ is reduced and Gorenstein with trivial dualizing sheaf,
\item if $e_{\min}(K)\geq 3$, then $\cL_Q$ is very ample,  
\item $(Q,\cL_Q)$ is an \'etale 
quotient of some PSQAS $(Q^*,\cL_{Q^*})$ with  ~\linebreak 
$e_{\min}(\ker\lambda(\cL_{Q^*}))\geq 3$, hence with $\cL_{Q^*}$ very ample, 
\item $\cG(K)$ acts on $(Q,\cL_Q)$ and $(P,\cL_P)$
extending the action of $\cG(K_{\eta})$ on $(G_{\eta},\cL_{\eta})$.
\end{enumerate}
\end{thm}

See Definition~\ref{defn:Heisenberg group} for $e_{\min}$. 
Theorem~\ref{thm:refined stable reduction}~(1) 
is proved in Subsec.~\ref{subsec:proof of isom}. 

We call $(Q_0,\cL_0)$ and $(P_0,\cL_0)$ as follows: 
\begin{enumerate}
\item[$\circ$] $(Q_0,\cL_0)$: PSQAS --- {\em a projectively stable 
quasi-abelian scheme}, which can be nonreduced,
\item[$\circ$] $(P_0,\cL_0)$:  TSQAS --- {\em a torically stable 
quasi-abelian scheme}
($=$ variety), which is always reduced.
\end{enumerate}

\begin{rem}
Theorem~\ref{thm:refined stable reduction}~(2) is rather misleading. 
In the proof of it, 
we never define $P$ to be the normalization of $Q$. 
We only construct $P$ with $P_0$ reduced 
and $P_{\eta}\simeq G_{\eta}$.
The normality of $P$ is a consequence 
of the reducedness of $P_0$ 
by the following well-known Claim.
\end{rem}

\begin{claim}\label{claim:normality}
Let $R$ be a complete discrete valuation ring, $S:=\Spec R$, and
$\eta$ the generic point of $S$. 
Assume that $\pi:Z\to S$ is flat with
$Z_0$ reduced and $Z_{\eta}$  nonsingular. 
Then $Z$ is normal.
\end{claim}
\begin{proof}
See \cite[Lemma~10.3]{Nakamura10}.\end{proof}

\begin{rem}In dimension one, any PSQAS is 
a TSQAS and vice versa, which is either a smooth elliptic 
or an $N$-gon (of rational curves).
Once the moduli of PSQASes (resp. TSQASes) is constructed,  
Theorem~\ref{thm:separatedness} will 
prove that {\em the moduli is separated}, 
and then Theorem~\ref{thm:refined stable reduction} 
will prove that {\em the moduli is proper}.
\end{rem}

\begin{thm}\label{thm:separatedness}
{\rm (Uniqueness \cite{Nakamura99},\cite{Nakamura10})} 
In Theorem~\ref{thm:refined stable reduction},  
$(Q,\cL)$ resp. $(P,\cL)$ is
 uniquely determined by $(G_{\eta},\cL_{\eta})$ if $e_{\min}(K)\geq 3$ 
(resp. in any case).
\end{thm}

See \cite[Theorem~10.4]{Nakamura99} and 
\cite[Theorem~10.4; Claim~2, p.~124]{Nakamura10} 
for the detail when $e_{\min}(H)\geq 3$. 
See Subsec~\ref{subsec:case emin leq 2} for $e_{\min}(H)\leq 2$.

\section{PSQASes in low dimension}\label{sec:PSQAS dim one}
The purpose of this section is to show motivating 
examples in dimension one and two.

\subsection{Hesse cubics and theta functions}
\label{subsec:Hesse cubics and thetas}
Let $R$ be a complete discrete valuation ring (CDVR), 
$I$ the maximal ideal of $R$
and $q$ a generator (uniformizer) of $I$, so $I=qR$. 
For instance, if $R=\bZ_3$, then we can choose 
 $q=3$, and if $R=k[[t]]$, $k$ a field, then $q=t$.  
Let $\theta_k$ 
be the same as in Subsec.~2.1~(iv)
\begin{align*}\theta_k(\omega,z)&=\sum_{m\in \bZ} q^{(3m+k)^2}w^{3m+k}
\end{align*}
Then the power series $\theta_k$ converge $I$-adically.

Now we calculate the limit of $[\theta_0,\theta_1,\theta_2]$ 
as $q$ tends to $0$. \par
First we shall show a computation, which once puzzled us so much.
\begin{align*}
\theta_{0}(q,w)&=\sum_{m\in \bZ} 
q^{9m^2}w^{3m}\\ &=1+q^9w^3+q^{9}w^{-3}+q^{36}w^{6}+\cdots,\\
\theta_{1}(q,w)&=\sum_{m\in \bZ} 
q^{(3m+1)^2}w^{3m+1}\\ &=qw+q^4w^{-2}+q^{16}w^4+\cdots,\\
\theta_{2}(q,w)&=\sum_{m\in \bZ} 
q^{(3m+2)^2}w^{3m+2}\\
&=qw^{-1}+q^4w^2+q^{16}w^{-4}+q^{25}w^{5}+\cdots.  
\end{align*}

Hence in $\bP^2$
\begin{align*}
\lim_{q\to 0}\ [\theta_0,\theta_1,\theta_2](q,w)]=[1,0,0] 
\end{align*}

The elliptic curves converge to one point?  This looks strange. 
The reason why we got the above is that we treated $w$ as a constant.
There is N\'eron model behind this strange phenomenon. 
We cannot explain it in detail here. Instead we show 
how to modify the above computation. 

Let $w=q^{-1}u$ 
for $u\in R\setminus I$ and $\overline{u}=u\,\mod I$. 
Then we have 
\begin{align*}
\theta_{0}(q,q^{-1}u)&=\sum_{m\in \bZ} 
q^{9m^2-3m}u^{3m}\\ &=1+q^6u^3+q^{12}u^{-3}+q^{30}u^{6}+\cdots,\\
\theta_{1}(q,q^{-1}u)&=\sum_{m\in \bZ} 
q^{(3m+1)^2-3m-1}u^{3m+1}\\ &=u+q^6u^{-2}+q^{12}u^4+\cdots,\\
\theta_{2}(q,q^{-1}u)&=\sum_{m\in \bZ} 
q^{(3m+2)^2-3m-2}u^{3m+2}\\
&=q^2u^2+q^2u^{-1}+q^{20}u^{5}+q^{20}u^{-4}+\cdots.  
\end{align*}

Hence in $\bP^2$
\begin{align*}
\lim_{q\to 0}\  [\theta_0,\theta_1,\theta_2]&(q,q^{-1}u)=[1,\overline{u},0] 
\end{align*}

Similarly 
\begin{align*}
\theta_{0}(q,q^{-2}u)&=1+q^3u^3+q^{15}u^{-3}+q^{24}u^{6}+\cdots,\\
\theta_{1}(q,q^{-2}u)&=q^{-1}u+q^{12}u^{-2}+q^{8}u^4+\cdots,\\
\theta_{2}(q,q^{-2}u)&=u^2+q^3u^{-1}+q^{15}u^{5}+q^{24}u^{-4}+\cdots,\\ 
\lim_{q\to 0}\ [\theta_0,\theta_1,\theta_2]&(q,q^{-2}u)
=\lim_{q\to 0}\ [1,q^{-1}u,u^2]=[0,1,0]\quad 
\text{in $\bP^2$.}
\end{align*}

Similarly 
\begin{align*}
\theta_{0}(q,q^{-3}u)&=1+u^3+q^{18}u^{-3}+q^{18}u^{6}+\cdots,\\
\theta_{1}(q,q^{-3}u)&=q^{-2}u+q^{10}u^{-2}+q^{4}u^4+\cdots,\\
\theta_{2}(q,q^{-3}u)&=q^{-2}u^2+q^4u^{-1}+q^{10}u^{5}+q^{28}u^{-4}+\cdots,\\ 
\lim_{q\to 0}\ [\theta_0,\theta_1,\theta_2]&(q,q^{-3}u)
=\lim_{q\to 0}\ [1,q^{-2}u,u^2]=[0,1,\overline{u}]
\quad \text{in $\bP^2$.}
\end{align*}

Let
$w=q^{-2\lambda}u$ (a section over a finite extension of 
$k(\eta)$ for $\lambda\in\bQ$) and $u\in R\setminus I$.
\begin{equation}
\lim_{q\to 0}\ [\theta_0,\theta_1,\theta_2](q,q^{-2\lambda}u) =
\Biggl\{\begin{tabular}{ll}
$[1,0,0]$ & \quad $(\op{if}\ -1/2<\lambda <1/2)$,\\
$[1,\overline{u},0]$& \quad $(\op{if}\ \lambda=1/2)$,\\
$[0,1,0]$ & \quad $(\op{if}\ 1/2<\lambda <3/2)$,\\
$[0,1,\overline{u}]$ &\quad $(\op{if}\ \lambda=3/2)$,\\
$[0,0,1]$ & \quad $(\op{if}\ 3/2<\lambda <5/2).$\\
$[\overline{u},0,1]$ &\quad $(\op{if}\ \lambda=5/2)$,\\
\end{tabular}
\end{equation}

When $\lambda$ ranges in $\bR$, the same calculation shows
that the same limits repeat mod $Y=3\bZ$ because  
\begin{align*}
\lim_{q\to 0}\  [\theta_0,\theta_1,\theta_2](q,q^{6n-a}u)&
=\lim_{q\to 0}\  [\theta_0,\theta_1,\theta_2](q,q^{-a}u).
\end{align*}

Thus we see that 
$\lim_{\tau\to\infty}C(\mu(\tau))$ is the 3-gon $x_0x_1x_2=0$.

\begin{defn}For $\lambda\in X\otimes_{\bZ}\bR$ fixed, 
let $$F_{\lambda}:=a^2 - 2\lambda a\quad (a\in X=\bZ).$$
We define a Delaunay cell 
$$D(\lambda):=\begin{matrix}
\text{the convex closure of all
$a\in X$}\\ 
\text{that attain the minimum of $F_{\lambda}$}
\end{matrix}$$
\end{defn}

By computations we see
\begin{align*}
D(j+\frac{1}{2})&=[j,j+1]:=\{x\in \bR;\ j\leq x\leq j+1\},\\ 
D(\lambda)&=\{j\}\quad(\op{if}\ j-\frac{1}{2}<\lambda<j+\frac{1}{2} ), \\
[\bar\theta_k]_{k=0,1,2}:&= 
\lim_{q\to 0}[\theta_{k}(q,q^{-2\lambda}u))]_{k=0,1,2}\\
\bar\theta_k&=
\begin{cases}
\bar u^{j} &(\op{if}\ j\in D(\lambda)\cap (k+3\bZ))\\
0 &(\op{if}\ D(\lambda)\cap (k+3\bZ)=\emptyset). 
\end{cases}
\end{align*}

For instance 
 $D(\frac{1}{2})\cap (0+3\bZ)=\{0\}$,  
$D(\frac{1}{2})\cap (1+3\bZ)=\{1\}$ and
\begin{equation*}
\lim_{q\to 0}[\theta_{k}(q,q^{-1}u))]
=[\bar\theta_0,\bar\theta_1,\bar\theta_2]
=[\bar u^0,\bar u,0]=[1,\bar u,0].
\end{equation*}

Similarly for any $\lambda=j+(1/2)$, we have an algebraic torus as a limit
\begin{equation*}
\{[\bar u^{j},\bar u^{j+1}]\in\bP^1;\bar u\in\bG_m\}
\simeq \bG_m\ (=\bC^*).
\end{equation*}


\vspace*{-1.2cm}
\vspace*{1.7cm}
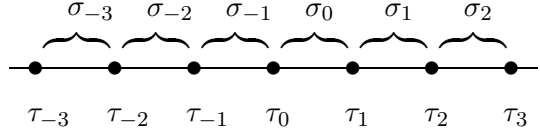
\begin{figure}[ht]
\hspace*{-14.1cm}
\begin{picture}(30,20)(-110,0)
   \put(15,0){\line(-1,0){10}}
   \put(15,0){\circle*{5}}
   \put(15,0){\line(1,0){30}}
   \put(45,0){\circle*{5}}
   \put(45,0){\line(1,0){30}}
   \put(75,0){\circle*{5}}
   \put(75,0){\line(1,0){30}}
   \put(105,0){\circle*{5}}
   \put(105,0){\line(1,0){30}}
   \put(135,0){\circle*{5}}
   \put(135,0){\line(1,0){30}}
   \put(165,0){\circle*{5}}
   \put(165,0){\line(1,0){30}}
   \put(195,0){\circle*{5}}
   \put(195,0){\line(1,0){10}}
\put(27,20){$\sigma_{-3}$}
\put(18,7){$\overbrace{\phantom{\line(1,0){25}}}$}
\put(57,20){$\sigma_{-2}$}
\put(48,7){$\overbrace{\phantom{\line(1,0){25}}}$}
\put(87,20){$\sigma_{-1}$}
\put(78,7){$\overbrace{\phantom{\line(1,0){25}}}$}
\put(117,20){$\sigma_{0}$}
\put(108,7){$\overbrace{\phantom{\line(1,0){25}}}$}
\put(147,20){$\sigma_{1}$}
\put(138,7){$\overbrace{\phantom{\line(1,0){25}}}$}
\put(177,20){$\sigma_{2}$}
\put(168,7){$\overbrace{\phantom{\line(1,0){25}}}$}
\put(12,-20){$\tau_{-3}$}
\put(42,-20){$\tau_{-2}$}
\put(72,-20){$\tau_{-1}$}
\put(102,-20){$\tau_{0}$}
\put(132,-20){$\tau_{1}$}
\put(162,-20){$\tau_{2}$}
\put(192,-20){$\tau_{3}$}
\end{picture}
\vspace*{0.8cm}
\caption{Delaunay decomposition}
\label{fig:Delaunay 3}
\end{figure}

Let $\lambda\in X\otimes\bR$, and 
$\sigma=D(\lambda)$ be a Delaunay cell, and $O(\sigma)$ 
the stratum of $C(\infty)$ consisting of limits of $(q,q^{-{2\lambda}}u)$. If $\sigma$ is one-dimensional, then 
$O(\sigma)=\bC^*$, 
while $O(\sigma)$ is one point if $\sigma$ is zero-dimensional.
Thus we see that
$C(\mu(\infty))$ is a disjoint union of $O(\sigma)$, $\sigma$ 
being Delaunay cells mod $Y$, in other words, it is stratified
in terms of the Delaunay decomposition mod $Y$.\par
Let $\sigma_j=[j,j+1]$ and $\tau_j=\{j\}$.  
Then the Delaunay decomposition 
(resp. the stratification of $C(\infty)$) 
is given in Fig.~1 
(resp. Fig.~\ref{fig:3gon}).


\vspace*{2.2cm}

\begin{figure}[ht]
\hspace*{-13.6cm}
\begin{picture}(40,30)(-163,0)
   \put(15,0){\line(1,0){78}}
   \put(15,0){\circle*{5}}
   \put(15,0){\line(3,5){38}}
   \put(52,63){\circle*{5}}
   \put(90,0){\circle*{5}}
   \put(90,0){\line(-3,5){38}}
\put(-5,-20){$O(\tau_0)$}
\put(66,63){$O(\tau_2)$}
\put(86,-20){$O(\tau_1)$}
\put(40,-20){$O(\sigma_0)$}
\put(78,30){$O(\sigma_1)$}
\put(0,30){$O(\sigma_2)$}
\end{picture}
\vspace*{8mm}
\caption{A 3-gon}\label{fig:3gon}
\end{figure}

\vspace*{0cm}

\subsection{The complex case}
To apply the computation in the last section 
to the moduli problem, we need to 
know the scheme-theoretic limit of the image of $E(\omega)$. \par
Now let us write 
\begin{align*}
\theta_k(q,w)&=\sum_{m\in\bZ}q^{(3m+k)^2}w^{3m+k}=\sum_{m\in\bZ}a(3m+k)w^{3m+k}
\end{align*}where $a(x)=q^{x^2}$ for $x\in X:=\bZ$. Let $Y=3\bZ$. 
Then $\theta_k$ is $Y$-invariant : 
$$\theta_k=\sum_{y\in Y}a(y+k)w^{y+k}.$$

Since the curve $E(\tau)$ is embedded into $\bP_{\bC}^2$ 
by $\theta_k$, we see
\begin{equation}
\label{eq:proj model of E(tau)}
\begin{aligned}
E(\omega)
&=\Proj\bC[x_k, k=0,1,2]/(x_0^3+x_1^3+x_2^3-3\mu(\omega)x_0x_1x_2)\\
&\simeq \Proj\bC[\theta_k\vartheta, k=0,1,2]\\
&=\Proj(\bC[[a(x)w^x\vartheta, x\in X]])^{Y-\op{inv}}
\end{aligned}
\end{equation}where $\vartheta$ is 
a transcendental element of degree one, 
$\deg(x_k)=1$, and 
$\deg(\theta_k)=0$ and $\deg(a(x)w^x)=0$.
Recall that if $U=\Spec A$ is affine, $G$ a finite group acting on
$U$, then 
$$U/G=\Spec A^{\text{$G$-inv}}.$$ 

So we wish to regard $E(\omega)$ as  
\begin{align*}\label{eq:isom 2}
E(\omega)=(\Proj(\bC[[a(x)w^x\vartheta, x\in X]]))/Y.
\end{align*}

Is this really true?  Over $\bC$, $a(x)\in \bC^{\times}$, and  
\begin{equation*}
\bG_m=\Proj \bC[a(x)w^x\vartheta, x\in X],
\end{equation*}In fact,  
the rhs is covered with infinitely many affine $U_k$
\begin{equation*}
U_k=\Spec \bC[a(x)w^x\vartheta/a(k)w^k\vartheta; x\in X]=\Spec \bC[w,w^{-1}]
=\bG_m,
\end{equation*}which is independent of $k$. 
Hence over $\bC$
\begin{equation}\label{eq:isom 2}
\begin{aligned}
E(\omega)
&\simeq \bG_m/w\mapsto q^6w\\
& \simeq \bG_m/\{w\mapsto q^{2y}w; y\in 3\bZ\}\\
& \simeq (\Proj \bC[a(x)w^x\vartheta, x\in X])/Y.
\end{aligned} 
\end{equation}

Thus we see by combining (\ref{eq:proj model of E(tau)}) 
and (\ref{eq:isom 2}) 
\begin{equation}\label{eq:isom 3}
\begin{aligned}
E(\omega)&\simeq \Proj(\bC[[a(x)w^x\vartheta, x\in X]])^{Y-\op{inv}}\\
& \simeq (\Proj \bC[a(x)w^x\vartheta, x\in X])/Y,
\end{aligned}
\end{equation} 
though we should make the convergence of infinite sum 
precise. In fact, this is easily justified 
when $R$ is a CDVR.

\subsection{The scheme-theoretic limit}
\label{subsec:scheme limit}
We define the subring $\wR$ 
of $k(\eta)[w, w^{-1}][\vartheta]$ by
\begin{equation*}
\wR=R[a(x)w^x\vartheta; x\in X]
\end{equation*}
where $a(x)=q^{x^2}$ for $x\in X$, $X=\bZ$,  and 
$\vartheta$ is an indeterminate of degree one, 
where $q$ is the uniformizer of $R$. 
We define 
the action of $Y$ on $\wR$ by
the ring homomorphism
\begin{equation}\label{eq:action of Y}
S_y^*(a(x)w^x\vartheta)=a(x+y)w^{x+y}\vartheta.
\end{equation}where $Y=3\bZ\subset X$. 
Then what does $Z$ look like?
\begin{equation*}
Z=\Proj R[a(x)w^x\vartheta, x\in X]/Y.
\end{equation*}

Let $\cX$ and $U_n$ be  
\begin{align*}
\cX&=\Proj R[a(x)w^x\vartheta, x\in X],\\
U_n&=\Spec R[a(x)w^x/a(n)w^n, x\in X]\\
&=\Spec R[(a(n+1)/a(n))w, (a(n-1)/a(n))w^{-1}]\\
&=\Spec R[q^{2n+1}w, q^{-2n+1}w^{-1}]\\
&\simeq \Spec R[x_n,y_n]/(x_ny_n-q^2),
\end{align*}where $U_n$ and $U_{n+1}$ are glued 
together by
$$x_{n+1}=x_n^2y_n,\ y_{n+1}=x_n^{-1},\ x_n=q^{2n+1}w,\ y_n=q^{-2n+1}w^{-1}.
$$


\begin{figure}[h]
\begin{picture}(30,10)(88,0)
   \put(15,0){\circle*{5}}
   \qbezier(0,8)(15,0)(20,-5)
   \put(45,0){\circle*{5}}
   \qbezier(10,-5)(30,20)(50,-5)
   \put(75,0){\circle*{5}}
   \qbezier(40,-5)(60,20)(80,-5)
   \put(105,0){\circle*{5}}
   \qbezier(70,-5)(90,20)(110,-5)
   \put(135,0){\circle*{5}}
   \qbezier(100,-5)(120,20)(140,-5)
   \put(165,0){\circle*{5}}
   \qbezier(130,-5)(150,20)(170,-5)
   \put(195,0){\circle*{5}}
   \qbezier(160,-5)(180,20)(200,-5)
   \qbezier(190,-5)(195,0)(210,8)
   \end{picture}
\vspace*{0.5cm}
\caption{An infinite chain}
\label{Fig:infinite lines}
\end{figure}
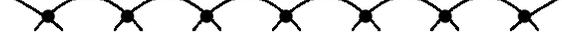


Let $\cX_0:=\cX\otimes_R(R/qR)$ and $V_n=\cX_0\cap U_n$. 
Then $\cX_0$ is an infinite chain of $\bP^1$, 
as in Fig.~3.

The action of the sublattice $Y=3\bZ$ on $\cX_0$ is transfer by 3 components. 
In fact, $S_{-3}$ 
sends $$V_n\overset{S_{-3}}{\rightarrow} V_{n+3}\overset{S_{-3}}
{\rightarrow} V_{n+6}\rightarrow \cdot\cdot\cdot,$$
$$(x_n,y_n)
\overset{S_{-3}}\mapsto (x_{n+3},y_{n+3})=(x_n,y_n)$$ 
so that we have a cycle of 3 rational curves as the quotient $\cX_0/Y$. 
Thus we have the same consequence as 
in Subsec.~\ref{subsec:Hesse cubics and thetas}  
by using theta functions.

\subsection{The partially degenerate case in dimension two}
\label{subsec:PSQAS in dim two partially deg}
We wish to describe any PSQAS in the partially degenerate case 
in dimension two. For simplicity, we shall give it directly by using 
theta functions.   See Subsec.~\ref{subsec:totally degenerate dim two} for 
the totally degenerate case.

\begin{subcase}First we consider the complex case. 
Let 
$$\delta=\diag(\ell,m):=\begin{pmatrix}\ell&0\\
0&m\end{pmatrix},\quad \tau=\begin{pmatrix}\tau_{11}&\tau_{12}\\
\tau_{12}&\tau_{22}\end{pmatrix},\quad \tau_{12}=\tau_{21}.$$

Let $\Lambda$ be the lattice spanned 
by column vectors of $I_2$ and $\tau\delta$, and 
$G_{\eta}$ the abelian variety $\bC^2/\Lambda$. 
We consider the degeneration of 
$G_{\eta}$ as $q:=e^{\pi i\tau_{22}}$ tends to $0$. 
Assume $\ell$ and $m\geq 3$.
Following \cite[Chap.~VII, pp.~77--79]{Siegel62} 
we define for $k=(k_1,k_2)$ $(0\leq k_1\leq \ell-1, 0\leq k_2\leq m-1)$,  
\begin{align*}
\theta_k&=
\sum_{n\in\bZ^2}e^{\pi i{}^t(\delta n+k)\tau(\delta n+k)
+2\pi i{}^t(\delta n+k)z}\\
&=\sum_{n_2\in\bZ}q^{(mn_2+k_2)^2}w^{mn_2+k_2}\vartheta_{k_1}(z_1+(mn_2+k_2)
\tau_{12}),
\end{align*} where $T=\tau\delta$, $W=\delta T$ 
with the notation of \cite{Siegel62}, 
$q=e^{\pi i\tau_{22}}$ and $\vartheta_{k_1}$ 
is a theta function of level $\ell$ of one variable.  Hence
\begin{align}\label{eq:theta partially deg case}
\theta_k&=\sum_{n_2\in\bZ}T^*_{(mn_2+k_2)\tau_{12}}(\vartheta_{k_1})
q^{(mn_2+k_2)^2}w^{mn_2+k_2}.
\end{align}where (\ref{eq:theta partially deg case}) 
is a general form of algebraic theta 
functions in \cite[Theorem~4.10~(3)]{Nakamura99}. 
\end{subcase}

\begin{subcase}
Now we consider the general case. In any algebraic case, we can 
 start with the last form (\ref{eq:theta partially deg case}) 
of theta functions 
by \cite[Theorem~4.10]{Nakamura99}, where $q$ is 
a uniformizer of a CDVR $R$.
In this case, $X=\bZ$, $Y=m\bZ$ and 
the Delaunay decomposition associated 
with this degeneration of abelian surfaces is 
the union of the unit intervals $[j,j+1]$ $(j\in\bZ)$ mod $Y$. \par
Let $H=(\bZ/\ell\bZ)\oplus (X/Y)\simeq (\bZ/\ell\bZ)\oplus(\bZ/m\bZ)$. 
By the theta functions $\theta_k$ we have a 
closed immersion of an abelian variety $G_{\eta}$ to $\bP(V_H)$. 
We compute the limit of the image of $G_{\eta}$ as $q$ tends to $0$. 

By the assumption $\ell\geq 3$, 
$\vartheta_{k_1}$ $(0\leq k_1\leq \ell-1)$ embeds an elliptic curve 
into the projective space $\bP^{\ell-1}$. \par 

Let $w=q^{-2a-1}v$ $(a\in\bZ)$,  $v\in R\setminus I$ and $I=qR$. 
Let $\overline{v}=v\,\mod I$. 
Then we have 
\begin{align*}
\theta_{k_1,a}(q,u,q^{-2a-1}v)&
=q^{-a^2-a}T^*_{a\tau_{12}}\vartheta_{k_1}+\cdots,\\
\theta_{k_1,a+1}(q,u,q^{-2a-1}v)&
=q^{-a^2-a}T^*_{(a+1)\tau_{12}}\vartheta_{k_1}+\cdots,\\
\theta_{k_1,k_2}(q,u,q^{-2a-1}v)&
\equiv 0\ \text{mod}\ q^{-a^2-a+1},\ (k_2\neq a,a+1),
\end{align*}whence 
\begin{align*}
\lim_{q\to 0}[\theta_{k_1,k_2}(q,u,q^{-2a-1}v)]_{(k_1,k_2)\in H}
&=[\theta_{k_1,a}u^a,\theta_{k_1,a+1}u^{a+1}]_{k_1}\\
&=[\underset{k_2=a}{\underbrace{T^*_{a\tau_{12}}\vartheta_{k_1}}},
\underset{k_2=a+1}{\underbrace{(T^*_{(a+1)\tau_{12}}\vartheta_{k_1})v}}]_{k_1}
\end{align*}with zero terms ignored.
In particular, for $w=q^{-1}v$, we have 
\begin{align}\label{eq:Gm bundle}
\lim_{q\to 0}[\theta_{k_1,k_2}(q,u,q^{-1}v)]_{(k_1,k_2)\in H}&
=[\theta_{k_1,0},\theta_{k_1,1}]
=[\underset{k_2=0}{\underbrace{\vartheta_{k_1}}},
\underset{k_2=1}{\underbrace{(T^*_{\tau_{12}}\vartheta_{k_1})v}}]. 
\end{align}

For $a=m$, we have 
\begin{align}\label{eq:infty identification}
\lim_{q\to 0}[\theta_{k_1,k_2}(q,u,q^{-2m-1}v)]_{(k_1,k_2)\in H}&
=[\underset{k_2=0}{\underbrace{T^*_{m\tau_{12}}\vartheta_{k_1}}},
\underset{k_2=1}{\underbrace{(T^*_{(m+1)\tau_{12}}\vartheta_{k_1})v}}] 
\end{align}

Thus the limit of the abelian surface $(G_{\eta},\cL_{\eta})$ as $q\to 0$ 
is the union of $m$ copies of one and the same  
$\bP^1$-bundle over an elliptic curve. By (\ref{eq:Gm bundle}), 
any of the  $\bP^1$-bundle is the 
same compactification of  
the same $\bG_m$-bundle whose extension class is given 
by $\tau_{12}$ through the isomorphism 
$$\Ext(E,\bG_m)\simeq E^{\vee}\simeq E\ni \tau_{12}. 
$$

By (\ref{eq:Gm bundle}) and (\ref{eq:infty identification}), 
the zero section of the first $\bP^1$-bundle 
is identified with the $\infty$-section of the $m$-th $\bP^1$-bundle 
by shifting by $\tau_{12}$. 
\end{subcase}

\section{PSQASes in the general case}
\label{sec:PSQAS general case}

\subsection{The degeneration data of Faltings-Chai}
\label{subsec:deg data of Faltings-Chai}
 Now we consider the general case. 
Let $R$ be a complete discrete valuation ring (CDVR), 
$k(\eta)$ the fraction field of $R$, $I$ the maximal ideal of $R$, 
$q$ a generator (uniformizer) of $I$ and $S=\Spec R$.  Then 
we can construct similar degenerations of abelian varieties 
if we are given a lattice $X$, a sublattice $Y$ of $X$ of finite index and 
$$a(x)\in k(\eta)^{\times},\quad (x\in X)
$$
such that the following conditions are satisfied
\begin{enumerate}
\item[(i)] $a(0)=1$,
\item[(ii)] $b(x,y):=a(x+y)a(x)^{-1}a(y)^{-1}$ is a symmetric 
bilinear form on $X\times X$,
\item[(iii)] $B(x,y):=\val_qb(x,y)$ is positive definite, 
\item[(iv)$^*$] $B$ is even and $\val_qa(x)=B(x,x)/2$. 
\end{enumerate}

We assume here a stronger condition (iv)$^*$ for simplicity. 

These data do exist for any abelian scheme $G_{\eta}$ if 
$G_0$ is a split torus.
This is proved  by Faltings-Chai \cite{FC90}. \par
Suppose that we are given an abelian scheme 
$(G_{\eta},\cL_{\eta})$ and a polarization morphism 
\begin{equation}\label{eq:pol morphism of Geta}
\lambda(\cL_{\eta}):G_{\eta}\to G^t_{\eta}:=\Pic^0(G_{\eta}).
\end{equation}
Then there exists the connected N\'eron model 
of $G_{\eta}$ (resp. $G^t_{\eta}$), which we denote by $G$ (resp. $G^t$).
Then by finite base change if necessary we may assume $G$ 
is semi-abelian, that is, an extension of an abelian scheme 
by an algebraic torus.\par
 For simplicity, we assume 
\begin{equation}\label{eq:assumption}
\text{
$G_0$ are $G^t_0$ are split tori over $k(0):=R/qR$.} 
\end{equation}  

Let 
\begin{equation}\label{eq:lattices X Y}
X=\Hom_{\text{gp.sch.}}(G_0,\bG_m),\quad 
Y=\Hom_{\text{gp.sch.}}(G^t_0,\bG_m).
\end{equation} Then both $X$ and $Y$ 
are lattices of rank $g$, and $Y$ is 
a sublattice of $X$ of finite index because 
$G_0\to G^t_0$ is surjective.  
This case is called a totally degenerate case, 
that is, the case when $\rank_{\bZ} X=\dim G_{\eta}$, which  
is what we mainly discuss here. 
\par
If $G_0$ is neither a torus nor an abelian variety, then the case 
 is called a partially degenerate case. 
Also in the partially degenerate case we have  
degeneration data similar to the above $a(x)$ and $b(x,y)$, 
though a bit more complicated.  This 
enables us to similarly construct a degenerating family of 
abelian varieties.
 
\par

In what follows we consider the case where $G_0$ is a (split) 
torus $\bG_{m,k(0)}^g$ over $k(0)$. 
\begin{sublemma}Let $R$ be  a CDVR, 
$G$ a flat $S$-group scheme, 
and $G_0$ 
the closed fiber of $G$. 
Suppose that $G_0$ is a (split) 
torus $\bG_{m,k(0)}^g$ over $k(0)$ for some $g$. Then 
the formal completion $G_{\formal}$ of $G$ along $G_0$ is isomorphic to
 a formal $R$-torus:
\begin{equation}\label{eq:isom of Gfor}
G_{\formal}\simeq \bG_{m,R,\formal}^g=\Spf R[[w^x;x\in X]]^{\text{$I$-adic}}
\end{equation}where $X$ is a lattice of rank $g$. 
\end{sublemma}
\begin{proof}Let $k=k(0)$.
Let $n$ be any nonnegative integer, $R_n=R/I^{n+1}$, $S_n=\Spec R_n$ 
and $G_n:=G\times_SS_n$. By the assumption, $G_0=\bG^g_{m,k}$ for some $g$. 
Let $H:=\bG_{m,R,\formal}^g$ 
(the formal torus over $R$) and 
$H_n=H\times_SS_n$. Hence $G_0=H_0=\bG_{m,k}^g$.
Let $f_0:H_0\to G_0$ be the identity $\id_{\bG^g_{m,k}}$ of $\bG^g_{m,k}$. 
Since $H_0=\bG^g_{m,R_0}$ is affine, 
the cohomology group $H^2(H_0,f_0^*{\cal Lie}(G_0/k))$ vanishes, where 
${\cal Lie}(G_0/k)$ is the tangent sheaf of $G_0$, hence isomorphic to 
$O_{G_0}^g$, hence $f_0^*{\cal Lie}(G_0/k))\simeq O_{H_0}^g$. 
By applying \cite[I, Expos\'e III, Corollaire~2.8, p.~118]{SGA3} 
to $H_1$, $G_1$ and $f_0$, we see that $f_0$ can be uniquely 
lifted to an $S_1$-(homo)morphism $f_1:H_1\to G_1$ 
as $S_1$-group schemes. 
This lifting $f_1$ is an isomorphism  
because $f_0$ is an isomorphism. 
Similarly any isomorphism $f_n:H_n\to G_n$ as $S_{n}$-group schemes  
can be lifted again by \cite[I, Expos\'e III, Corollaire~2.8, p.~118]{SGA3} 
to an $S_{n+1}$-isomorphism $f_{n+1}:H_{n+1}\to G_{n+1}$ as 
$S_{n+1}$-group schemes because  $H_n$ is affine,  and the cohomology group 
$H^2(H_n,f_n^*{\cal Lie}(G_n/k))$ vanishes 
by the same argument as the $n=0$ case. Hence 
$H_{\formal}\simeq G_{\formal}$ as $S$-group schemes. 
\end{proof}

\begin{sublemma}
\label{sublemma:theta Fourier expansion}
We have
\begin{enumerate}
\item 
any line bundle on $\bG^g_{m,R,\formal}$ is trivial. 
\item any global section $\theta\in\Gamma(G,\cL^n)$ is 
a formal power series of $w^x$, and we can write $\theta$ as 
\begin{equation}\label{eq:theta expansion}
\theta=\sum_{x\in X}\sigma_x(\theta)w^x
\end{equation} for some $\sigma_x(\theta)\in R$. 
\end{enumerate}
\end{sublemma}
\begin{proof}
Let $R_n=R/I^{n+1}$, $A_n:=R_n[w_i^{\pm 1}; i=1,\cdots,g]$ and 
$$G_n:=\bG^g_m\otimes R_n=\Spec A_n.$$ 
 To prove the first assertion, it suffices to prove 
\begin{enumerate}
\item[(i)] any line bundle $L_0$ on $G_0$ is trivial,
\item[(ii)] if a line bundle $L$ on $G_{n}$ is trivial on $G_{n-1}$, it is trivial on $G_{n}$.
\end{enumerate} 
 
Any line bundle $L_0$ on $G_0$ is linearly equivalent to $D-D'$ for some effective divisors $D$ and $D'$ on $G_0$. For proving (i) it suffices to prove that 
the line bundle $L'=[D]$ associated to any irreducible divisor $D$ on $G_0$ 
is trivial. Since $G_0$ is affine, $D$ is defined by a prime ideal $\frak p$ of $A_0$ of height one. Since $A_0$ is a UFD,  $\frak p$  is generated by a single generator \cite[Theorem~47, p.~141]{Matsumura70}, hence 
it defines a trivial line bundle globally on $G_0$. This proves (i). \par
Next we prove (ii). Since $G_n$ is an $R_n$-scheme, we 
can find an affine covering 
$U_j=\Spec B_j$ of $G_n$ for some $R_n$-algebras $B_j$, and one cocycle 
$f_{jk}\in\Gamma(O_{U_{jk}})^{\times}$ (the units of $\Gamma(O_{U_{jk}})$) 
associated to the line bundle $L$ on $G_n$ such that
\begin{equation}\label{eq:one cocycle of L}f_{ij}f_{jk}=f_{ik}. 
\end{equation}By the assumption that $L$ is trivial on $G_{n-1}$ 
there exist $g_j\in B_j^{\times}$ such that 
$f_{ij}=g_i^{-1}g_j\mod I^n. $  Let $g_{ij}=g_ig_j^{-1}f_{ij}$. Then $g_{ij}$ is the one cocycle defining $L$ on $G_n$ such that 
$g_{ij}=1+a_{ij}q^n$ 
for some $a_{ij}\in B_{ij}$.  By (\ref{eq:one cocycle of L}), 
we have $g_{ij}g_{jk}=g_{ik}$, hence 
\begin{equation*}a_{ij}+a_{jk}=a_{ik}\quad \text{in $B_{ijk}\otimes R/I$,}
\end{equation*}
where $B_{ijk}=\Gamma(O_{U_i\cap U_j\cap U_k})$. Since $H^1(O_{G_0})=0$, 
we have $b_j\in B_{i}\otimes R_0$ such that $a_{ij}=-b_i+b_j$. Hence 
\begin{equation*}g_{ij}=(1+b_iq^n)^{-1}(1+b_jq^n),
\end{equation*}which defines the trivial line bundle on $G_n$. 
This proves (ii). Hence this completes the proof of the first assertion 
of Lemma~\ref{sublemma:theta Fourier expansion}. The second assertion 
of Lemma~\ref{sublemma:theta Fourier expansion} follows easily from it. 
\end{proof}

\begin{subthm}
\label{subthm:degeneration data}
If $G$ is totally degenerate, then by a suitable finite base change, 
there exist data 
$\{a(x);x\in X\}$ satisfying (i)-$(iv)^*$. In terms of these data, 
we have using the expression (\ref{eq:theta expansion})
\begin{enumerate}
\item[(v)] for any $n\geq 1$,
$\Gamma(G_{\eta},\cL_{\eta}^n)$ is 
the $k(\eta)$ vector space of  
$\theta$ such that $$\sigma_{x+y}(\theta)
=a(y)^nb(y,x)\sigma_x(\theta)$$ and $\sigma_x(\theta)\in k(\eta)$ 
$\text{for any $x\in X, y\in Y$}$. 
\end{enumerate}
\end{subthm}

The condition (v) enables us to prove 
the part (1) of Theorem~\ref{thm:refined stable reduction}.

\subsection{Construction}
\label{subsec:construction}
So we may assume we are given the data $a(x)$ as above. 
Then we define 
$\cX$, $U_n$ $(n\in X)$,   by 
\begin{align*}
\cX&=\Proj \wR,\quad \wR:=R[a(x)w^x\vartheta; x\in X],\\
U_n&=\Spec R[a(x)w^x/a(n)w^n; x\in X]\\
&=\Spec R[(a(x)/a(n))w^{x-n}],
\end{align*}where $\wR$ is  a subring of 
$k(\eta)[w^x; x\in X][\vartheta]$ 
as in Subsec.~\ref{subsec:scheme limit}, 
and $\cX$ is a scheme locally of finite type, covered with 
open affine schemes $U_n$ $(n\in X)$. 
Let $\cX_{\formal}$ be the formal completion of $\cX$ 
along the special fiber.\par  
We define 
$\cL_{\cX}$ to be the line bundle of $\cX$ given by 
the homogeneous ideal of $\wR$  
generated by the degree one generator $\vartheta$. 
We identify $X\times_{\bZ}\bG_{m,R}\ (\simeq \bG_{m,R}^g)$ 
with $\Hom_{\bZ}(X,\bG_{m,R})$.
Then we have the actions $S_z$  and $T_{\beta}$ on $\cX$ as follows: 
\begin{align*}
S_z^*(a(x)w^x\vartheta)&=a(x+z)w^{x+z}\vartheta,\\
T_{\beta}^*(a(x)w^x\vartheta)&=\beta(x)a(x)w^x\vartheta,\ \text{hence}\\
T_{\beta}^*S_z^*(a(x)w^x\vartheta)&=\beta(x+z)a(x+z)w^{x+z}\vartheta,\\
S_z^*T_{\beta}^*(a(x)w^x\vartheta)&=\beta(x)a(x+z)w^{x+z}\vartheta,
\end{align*}where $z\in X$ and $\beta\in\Hom(X,\bG_{m,R})
\ (\simeq \bG_{m,R}^g)$. 
It follows that on $\cL_{\cX}$
\begin{equation}\label{eq:action of cGH}
S_zT_{\beta}=\beta(z)T_{\beta}S_z, \quad \text{or}\quad 
[S_z, T_{\beta}]=\beta(z)\id_{\cL_{\cX}}.
\end{equation} 

Let $Q_{\formal}:=\cX_{\formal}/Y:=\cX_{\formal}/\{S_y; y\in Y\}$ :
\begin{align*}
\cX_{\formal}/Y&=(\Proj R[a(x)w^x\vartheta, x\in X])_{\formal}/Y.
\end{align*}
Then $\cL_{\cX}$ descends to the formal quotient $Q_{\formal}$ 
as an ample sheaf. Hence by Grothendieck's algebraization theorem 
\cite[III, {\bf 11}, 5.4.5]{EGA} there exists a scheme $(Q,\cL)$ 
such that the formal completion of $(Q,\cL)_{\formal}$ is 
isomorphic to $(Q_{\formal},\cL_{\formal})$.
This is $(Q,\cL_Q)$ in Theorem~\ref{thm:refined stable reduction}.

\begin{subrem}
For any connected $R$-scheme $T$, and for any $T$-valued points  
$x\in X(T)=X$ and $\beta\in\Hom(X,\bG_{m,R})(T)$, 
we have 
$\beta(x)\in\bG_{m,R}(T)=\Gamma(O_T)^{\times}$. 
Any $\beta\in\Hom(X,\bG_{m,R})$ 
acts on $\cX$\ by $T_{\beta}$.
It follows that the $R$-split torus 
$\Hom(X,\bG_{m,R})$ 
acts on $\cX$ by $T_{\beta}$. 
\end{subrem}

\begin{subdefn}
\label{subdefn:Sz Tbeta}
Let $H=X/Y$, $H^{\vee}:=\Hom(H,\bG_m)$. 
We define $\cG(Q,\cL)=\cG(P,\cL)$ to be the group generated by 
$S_z$ and $T_{\beta}$ $(z\in H=X/Y, \beta\in H^{\vee})$. 
Since $H^{\vee}$ is a subgroup of $\Hom(X,\bG_m)$, we infer 
from (\ref{eq:action of cGH}) that 
\begin{equation}\label{eq:action of cGH on (Q,L)}
S_zT_{\beta}=\beta(z)T_{\beta}S_z.
\end{equation}

This is isomorphic to $\cG_H$ in 
Definition~\ref{defn:Heisenberg group} by mapping 
$S_z$ (resp. $T_{\beta}$) to $(1,z,0)$ (resp. $(1,0,\beta)$). \theoremtextend
\end{subdefn}

In what follows, we 
wish to prove Theorem~\ref{thm:refined stable reduction}~(1) 
\begin{equation}
\label{eq:isom}
(P_{\eta},\cL_{\eta})\simeq (Q_{\eta},\cL_{\eta})\simeq (G_{\eta},\cL_{\eta}).
\end{equation} 

For doing so, we essentially need only the following. 
\begin{lemma}\label{lemma:cohomology}
With the notation in 
Subsec.~\ref{subsec:limit obj} and 
Theorem~\ref{thm:refined stable reduction}, 
suppose 
$(K_{\eta},e_{\Weil})\simeq (K,e_K)$ as symplectic groups. 
Let $Z=P$ or $Q$. 
Then 
there exists $n_0$ such that for any $n\geq n_0$ we have 
\begin{enumerate}
\item 
$H^q(Z_0,\cL_0^n)=H^q(Z,\cL^n)=0$ for $q\geq 1$, 
\item $\Gamma(Z_0,\cL_0^n)=\Gamma(Z,\cL^n)\otimes k(0)$ is 
a $k(0)$-vector space rank $n^g\sqrt{|K|}$, 
\item
 $\Gamma(P_{\eta},\cL_{\eta}^n)
=\Gamma(P,\cL^n)\otimes k(\eta)$,
\item\label{item:level one theta} $\Gamma(P,\cL)=\Gamma(Q,\cL)$, 
which is a free $R$-module of rank $\sqrt{|K|}$, 
\item if $e_{\min}(K)\geq 3$, then $\Gamma(Q,\cL)$ is very ample on $Q$.
\end{enumerate}
\end{lemma}
\begin{proof}This is a corollary to Serre's vanishing theorem except 
(\ref{item:level one theta}). 
See \cite[Lemma~5.12]{Nakamura99} for (4). 
See \cite[Lemma~6.3]{Nakamura99} for (5).
\end{proof}

\subsection{
Proof of $(P_{\eta},\cL_{\eta})\simeq 
(Q_{\eta},\cL_{\eta})\simeq (G_{\eta},\cL_{\eta})$}
\label{subsec:proof of isom}

By \cite[Remark~3.10, p.~673]{Nakamura99} 
(see also \cite[Remark~4.11, p.~679]{Nakamura99}), 
$\Gamma(P_{\eta},\cL_{\eta}^n)$ is  
a $k(\eta)$-submodule of 
$\Gamma(G_{\formal},\cL_{\formal}^n)\otimes k(\eta)$ 
given by
\begin{equation*}
\left\{
\theta=\sum_{x\in X}\ c(x)w^x;\begin{matrix} 
c(x+ny)=b(y,x)a(y)^nc(x)\\
c(x)\in k(\eta),\text{any $x\in X, y\in Y$}
\end{matrix}
\right\}
\end{equation*}
where the $I$-adic convergence of $\theta$ is automatic by the condition 
$$c(x+ny)=b(y,x)a(y)^nc(x).$$ 
This is the same as $\Gamma(G_{\eta},\cL_{\eta}^n)$ by 
Theorem~\ref{subthm:degeneration data}. 
A $k(\eta)$-basis of $\Gamma(G_{\eta},\cL_{\eta}^n)$ is given for instance as 
$\theta^{[n]}_{\bar x}$ $(x\in X/nY)$ 
$$\theta^{[n]}_{\bar x}:=\sum_{y\in Y}b(y,x)a(y)^na(x)w^{x+ny}=
\sum_{y\in Y}a(y)^{n-1}a(x+y)w^{x+ny}.
$$

We choose $n\geq 4$ large enough so that 
$\cL^n_{\eta}$ is very ample. 
Then the abelian variety $G_{\eta}$ embedded by the linear system 
$\Gamma(G_{\eta},\cL_{\eta}^n)$ is given as the intersection of 
certain quadrics of $\theta^{[n]}_{\bar x}$ by 
\cite[Theorem~10, p.80]{Mumford70} 
(see also \cite[Theorem~2.1, p.~717]{Sekiguchi78}).  The 
coefficients of the defining equations are given by 
the Fourier coefficients of $\theta^{[n]}_{\bar x}$. 
This proves 
$$(Q_{\eta},\cL_{\eta})\simeq (P_{\eta},\cL_{\eta})\simeq (G_{\eta},\cL_{\eta}).$$where 
$(Q_{\eta},\cL_{\eta})\simeq 
(P_{\eta},\cL_{\eta})$ is clear. 
 \theoremtextend

\subsection{The Delaunay decompositions}
Let $X$ be a lattice of rank $g$ and $B$ a positive symmetric 
integral bilinear form on $X$ associated with the degeneration data 
for $(\cZ,\cL)$. 

\begin{subdefn}
\label{subdefn:Delaunay decomp}
For a fixed $\lambda\in X\otimes_{\bZ}\bR$ fixed, we define 
a Delaunay cell $\sigma$ to be 
the convex closure of all the integral 
vectors (which we call Delaunay vectors) attaining the minimum  
of the function
\begin{equation*}B(x,x)-2B(\lambda,x)\ \ (x\in X).
\end{equation*}
\end{subdefn}

%
\vspace*{1.7cm}
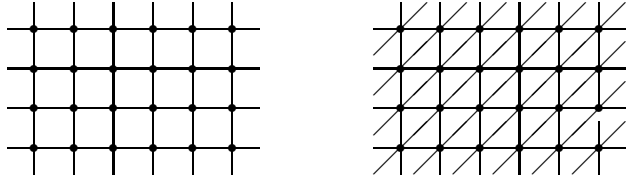
\begin{figure}[ht]
\hspace*{-13.5cm}
\begin{picture}(30,40)(-100,-20)
   \put(15,0){\line(-1,0){10}}
   \put(15,0){\circle*{3}}
   \put(15,0){\line(1,0){15}}
   \put(30,0){\circle*{3}}
   \put(30,0){\line(1,0){15}}
   \put(45,0){\circle*{3}}
   \put(45,0){\line(1,0){15}}
   \put(60,0){\circle*{3}}
   \put(60,0){\line(1,0){15}}
   \put(75,0){\circle*{3}}
   \put(75,0){\line(1,0){15}}
   \put(90,0){\circle*{3}}
   \put(90,0){\line(1,0){10}}
   \put(15,15){\line(-1,0){10}}
   \put(15,15){\circle*{3}}
   \put(15,15){\line(1,0){15}}
   \put(30,15){\circle*{3}}
   \put(30,15){\line(1,0){15}}
   \put(45,15){\circle*{3}}
   \put(45,15){\line(1,0){15}}
   \put(60,15){\circle*{3}}
   \put(60,15){\line(1,0){15}}
   \put(75,15){\circle*{3}}
   \put(75,15){\line(1,0){15}}
   \put(90,15){\circle*{3}}
   \put(90,15){\line(1,0){10}}
   \put(15,30){\line(-1,0){10}}
   \put(15,30){\circle*{3}}
   \put(15,30){\line(1,0){15}}
   \put(30,30){\circle*{3}}
   \put(30,30){\line(1,0){15}}
   \put(45,30){\circle*{3}}
   \put(45,30){\line(1,0){15}}
   \put(60,30){\circle*{3}}
   \put(60,30){\line(1,0){15}}
   \put(75,30){\circle*{3}}
   \put(75,30){\line(1,0){15}}
   \put(90,30){\circle*{3}}
   \put(90,30){\line(1,0){10}}
   \put(15,45){\line(-1,0){10}}
   \put(15,45){\circle*{3}}
   \put(15,45){\line(1,0){15}}
   \put(30,45){\circle*{3}}
   \put(30,45){\line(1,0){15}}
   \put(45,45){\circle*{3}}
   \put(45,45){\line(1,0){15}}
   \put(60,45){\circle*{3}}
   \put(60,45){\line(1,0){15}}
   \put(75,45){\circle*{3}}
   \put(75,45){\line(1,0){15}}
   \put(90,45){\circle*{3}}
   \put(90,45){\line(1,0){10}}
   \put(30,0){\line(0,-1){10}}
   \put(45,0){\line(0,-1){10}}
   \put(60,0){\line(0,-1){10}}
   \put(75,0){\line(0,-1){10}}
   \put(90,0){\line(0,-1){10}}
   \put(15,0){\line(0,-1){10}}
   \put(15,0){\line(0,1){15}}
   \put(30,0){\line(0,1){15}}
   \put(45,0){\line(0,1){15}}
   \put(60,0){\line(0,1){15}}
   \put(75,0){\line(0,1){15}}
   \put(90,0){\line(0,1){15}}
   \put(15,15){\line(0,-1){10}}
   \put(15,15){\line(0,1){15}}
   \put(30,15){\line(0,1){15}}
   \put(45,15){\line(0,1){15}}
   \put(60,15){\line(0,1){15}}
   \put(75,15){\line(0,1){15}}
   \put(90,15){\line(0,1){15}}
   \put(15,30){\line(0,-1){10}}
   \put(15,30){\line(0,1){15}}
   \put(30,30){\line(0,1){15}}
   \put(45,30){\line(0,1){15}}
   \put(60,30){\line(0,1){15}}
   \put(75,30){\line(0,1){15}}
   \put(90,30){\line(0,1){15}}
   \put(15,45){\line(0,-1){10}}
   \put(15,45){\line(0,1){10}}
   \put(30,45){\line(0,1){10}}
   \put(45,45){\line(0,1){10}}
   \put(60,45){\line(0,1){10}}
   \put(75,45){\line(0,1){10}}
   \put(90,45){\line(0,1){10}}
   \end{picture}
\begin{picture}(30,40)(-205,-20)
   \put(15,0){\line(-1,0){10}}
   \put(15,0){\circle*{3}}
   \put(15,0){\line(1,0){15}}
   \put(30,0){\circle*{3}}
   \put(30,0){\line(1,0){15}}
   \put(45,0){\circle*{3}}
   \put(45,0){\line(1,0){15}}
   \put(60,0){\circle*{3}}
   \put(60,0){\line(1,0){15}}
   \put(75,0){\circle*{3}}
   \put(75,0){\line(1,0){15}}
   \put(90,0){\circle*{3}}
   \put(90,0){\line(1,0){10}}
   \put(15,15){\line(-1,0){10}}
   \put(15,15){\circle*{3}}
   \put(15,15){\line(1,0){15}}
   \put(30,15){\circle*{3}}
   \put(30,15){\line(1,0){15}}
   \put(45,15){\circle*{3}}
   \put(45,15){\line(1,0){15}}
   \put(60,15){\circle*{3}}
   \put(60,15){\line(1,0){15}}
   \put(75,15){\circle*{3}}
   \put(75,15){\line(1,0){15}}
   \put(90,15){\circle*{3}}
   \put(90,15){\line(1,0){10}}
   \put(15,30){\line(-1,0){10}}
   \put(15,30){\circle*{3}}
   \put(15,30){\line(1,0){15}}
   \put(30,30){\circle*{3}}
   \put(30,30){\line(1,0){15}}
   \put(45,30){\circle*{3}}
   \put(45,30){\line(1,0){15}}
   \put(60,30){\circle*{3}}
   \put(60,30){\line(1,0){15}}
   \put(75,30){\circle*{3}}
   \put(75,30){\line(1,0){15}}
   \put(90,30){\circle*{3}}
   \put(90,30){\line(1,0){10}}
   \put(15,45){\line(-1,0){10}}
   \put(15,45){\circle*{3}}
   \put(15,45){\line(1,0){15}}
   \put(30,45){\circle*{3}}
   \put(30,45){\line(1,0){15}}
   \put(45,45){\circle*{3}}
   \put(45,45){\line(1,0){15}}
   \put(60,45){\circle*{3}}
   \put(60,45){\line(1,0){15}}
   \put(75,45){\circle*{3}}
   \put(75,45){\line(1,0){15}}
   \put(90,45){\circle*{3}}
   \put(90,45){\line(1,0){10}}
   \put(30,0){\line(0,-1){10}}
   \put(45,0){\line(0,-1){10}}
   \put(60,0){\line(0,-1){10}}
   \put(75,0){\line(0,-1){10}}
   \put(90,0){\line(0,-1){10}}
   \put(15,0){\line(0,-1){10}}
   \put(15,0){\line(0,1){15}}
   \put(30,0){\line(0,1){15}}
   \put(45,0){\line(0,1){15}}
   \put(60,0){\line(0,1){15}}
   \put(75,0){\line(0,1){15}}
   \put(90,0){\line(0,1){10}}
   \put(15,15){\line(0,-1){10}}
   \put(15,15){\line(0,1){15}}
   \put(30,15){\line(0,1){15}}
   \put(45,15){\line(0,1){15}}
   \put(60,15){\line(0,1){15}}
   \put(75,15){\line(0,1){15}}
   \put(90,15){\line(0,1){15}}
   \put(15,30){\line(0,-1){10}}
   \put(15,30){\line(0,1){15}}
   \put(30,30){\line(0,1){15}}
   \put(45,30){\line(0,1){15}}
   \put(60,30){\line(0,1){15}}
   \put(75,30){\line(0,1){15}}
   \put(90,30){\line(0,1){15}}
   \put(15,45){\line(0,-1){10}}
   \put(15,45){\line(0,1){10}}
   \put(30,45){\line(0,1){10}}
   \put(45,45){\line(0,1){10}}
   \put(60,45){\line(0,1){10}}
   \put(75,45){\line(0,1){10}}
   \put(90,45){\line(0,1){10}}
   \put(15,0){\line(-1,-1){10}}
   \put(15,0){\line(1,1){15}}
   \put(30,0){\line(1,1){15}}
   \put(30,0){\line(-1,-1){10}}
   \put(45,0){\line(1,1){15}}
   \put(45,0){\line(-1,-1){10}}
   \put(60,0){\line(1,1){15}}
   \put(60,0){\line(-1,-1){10}}
   \put(75,0){\line(1,1){15}}
   \put(75,0){\line(-1,-1){10}}
   \put(90,0){\line(1,1){10}}
   \put(90,0){\line(-1,-1){10}}
   \put(15,15){\line(-1,-1){10}}
   \put(15,15){\line(1,1){15}}
   \put(30,15){\line(1,1){15}}
   \put(45,15){\line(1,1){15}}
   \put(60,15){\line(1,1){15}}
   \put(75,15){\line(1,1){15}}
   \put(90,15){\line(1,1){10}}
   \put(15,30){\line(-1,-1){10}}
   \put(15,30){\line(1,1){15}}
   \put(30,30){\line(1,1){15}}
   \put(45,30){\line(1,1){15}}
   \put(60,30){\line(1,1){15}}
   \put(75,30){\line(1,1){15}}
   \put(90,30){\line(1,1){10}}
   \put(15,45){\line(-1,-1){10}}
   \put(15,45){\line(1,1){10}}
   \put(30,45){\line(1,1){10}}
   \put(45,45){\line(1,1){10}}
   \put(60,45){\line(1,1){10}}
   \put(75,45){\line(1,1){10}}
   \put(90,45){\line(1,1){10}}
   \end{picture}
\caption{Delaunay decompositions}
\label{fig:2dim Delaunay}
\end{figure}

When $\lambda$ ranges in $X\otimes_{\bZ}\bR$, 
we will have various Delaunay cells. Together, they constitute 
a locally finite polyhedral decomposition of $X\otimes_{\bZ}\bR$, 
invariant under the translation by $X$. 
We call this the Delaunay decomposition of $X\otimes_{\bZ}\bR$, which we
denote by $\Del_B$. \par
There are two types of Delaunay decomposition of $\bZ^2\otimes\bR=\bR^2$ 
inequivalent under the action of $\SL(2,\bZ)$.
See Figure~\ref{fig:2dim Delaunay}.

The Delaunay decomposition
describes a PSQAS as follows.
\begin{thm}\label{thm:stratification of Z} 
Let $(Z,L):=(Q_0,\cL_0)$ be a totally degenerate PSQAS, $X$ the integral 
lattice, $Y$ the sublattice of $X$ of finite index and $B$ 
the positive integral bilinear form on $X$ 
all of which were defined in Subsec.~\ref{subsec:deg data of Faltings-Chai}. 
Let $\sigma$,$\tau$ be Delaunay cells in $\Del_B$. Then
\begin{enumerate} 
\item for each $\sigma$ there exists 
a subscheme $O(\sigma)$ of $Z_{\red}$,  
which is a torus of dimension $\dim\sigma$ 
invariant under the action of the torus $G_0$,  
\item  $\sigma\subset\tau$ iff $O(\sigma)
\subset\overline{O(\tau)}$, where $\overline{O(\tau)}$ is 
the closure of $O(\tau)$ in $Z$,
\item $\overline{O(\tau)}$ is 
the disjoint union of $O(\sigma)$ for all $\sigma\subset\tau$,
\item  $Z_{\red}=\bigcup_{\sigma\in\Del_B\op{mod}\ Y}\ O(\sigma)$, 
\item  the local scheme structure of $Z$ is completely described by $B$,
\item  $L$ is ample, and 
it is very ample if $e_{\min}(X/Y)\geq 3$. 
\end{enumerate}
\end{thm}


We have similar descriptions of the partially degenerate PSQASes  
and of TSQASes $(P_0,\cL_0)$
(see \cite[p.~410]{AN99} and \cite[p.~678]{Nakamura99}).

\subsection{The totally degenerate case in dimension two}
\label{subsec:totally degenerate dim two}
We note that we learned  more or less the same computation 
as this subsection
in a letter of K. Ueno to Namikwa in 1972. 
We shall explain here what 
 Figure~\ref{fig:2dim Delaunay} means geometrically. \par
We follow 
the construction in Subsec.~\ref{subsec:construction}.
Let $R$ be a CDVR with uniformizer $q$, $k(0)=R/qR$ and 
$X=\bZ f_1\oplus\bZ f_2$ a lattice of rank two. 
Let $\ell$ and $m$ be any positive integers, and 
set $Y=\bZ\ell f_1\oplus \bZ mf_2$.

\begin{subcase}\label{subcase:first}\quad
Let $B(x)=x_1^2+x_2^2$,  
\begin{gather*}
a(x)=q^{x_1^2+x_2^2}a^{2x_1x_2},\
b(x,y)=q^{2x_1y_1+ 2x_2y_2}a^{2x_1y_2+2y_1x_2}
\end{gather*}where $a\in R^{\times}$, $x=x_1f_1+x_2f_2$, $y=y_1f_1+y_2f_2$.
Then we define 
\begin{align*}
\cX&=\Proj R[a(x)w^x\vartheta, x\in X],\\
U_n&=\Spec R[a(x)w^x/a(n)w^n, x\in X]\quad (n\in X)\\
&=\Spec R[(a(x)/a(n))w^{x-n}],\\
\cX_{\formal}/Y&=(\Proj R[a(x)w^x\vartheta, x\in X])_{\formal}/Y.
\end{align*}

Let $Q'_{\formal}:=\cX_{\formal}/Y$. Let $n=0$ for simplicity. Then we have 
\begin{align*}
U_0&=\Spec R[a(f_1)w_1, a(f_2)w_2, a(-f_1)w_1^{-1}, a(-f_2)w_2^{-1}],\\
(U_0)_0&=\Spec R[qw_1, qw_2, qw_1^{-1}, qw_2^{-1}]\otimes k(0)\\
&\simeq\Spec k(0)[u_1, u_2, v_1, v_2]/(u_1v_1, u_2v_2),
\end{align*}where $(U_0)_0=U_0\otimes k(0)$. 
Hence $U_n\simeq U_0$ and 
\begin{align*}
(U_n)_0:&=\Spec k(0)[u^{(n)}_1, u^{(n)}_2, v^{(n)}_1, v^{(n)}_2]/(u^{(n)}_1v^{(n)}_1, u^{(n)}_2v^{(n)}_2)
\end{align*}where $n=n_1f_1+n_2f_2$, and 
\begin{gather*}
u^{(n)}_1=q^{2n_1+1}w_1, 
u^{(n)}_2=q^{(2n_2+1)}w_2,\\ 
v^{(n)}_1=q^{(-2n_1+1)}w_1^{-1}, v^{(n)}_2=q^{(-2n_2+1)}w_2^{-1}.
\end{gather*}

 These charts will be patched together to 
yield $(Q'_{\formal})_0$. \par
 This PSQAS $(Q'_{\formal})_0$ 
is a union of 
$\ell m$ copies of $\bP^1\times \bP^1$, whose configuration is just the same as the Delaunay decomposition on the left hand side in Fig.~4. The first 
horizontal chain of $\ell$ rational curves 
is identified with the $m$-th horizontal chain of $\ell$ rational curves by 
shifting by multiplication by $a^{2m}$ on each rational curve, while 
the first 
vertical chain of $m$ rational curves 
is identified with the $\ell$-th vertical chain of $m$ rational curves by 
shifting by multiplication by $a^{2\ell}$ on each rational curve
because 
\begin{align*}
S_{mf_2}^*(w_1)&=b(f_1,mf_2)w_1=a^{2m}w_1,\\
S_{\ell f_1}^*(w_2)&=b(\ell f_1,f_2)w_2=a^{2\ell}w_2.
\end{align*}

The PSQAS $(Q'_{\formal})_0$ is a level-$\cG_H$ PSQAS.
where $H=(\bZ/\ell\bZ)\oplus(\bZ/m\bZ)\simeq 
(\bZ/e_1\bZ)\oplus(\bZ/e_2\bZ)$, 
with $e_1=\text{GCD}(\ell,m)$ and $e_2=\ell m/e_1$.  
\end{subcase}
\begin{subcase}\label{subcase:second}\quad
Let $B(x)=x_1^2-x_1x_2+x_2^2$, 
\begin{gather*}
a(x)=q^{x_1^2-x_1x_2+ x_2^2},\ 
b(x,y)=q^{2x_1y_1-x_1y_2-x_2y_1+2x_2y_2}
\end{gather*}where $x=x_1f_1+x_2e_2$, $y=y_1f_1+y_2e_2$.
Then we define 
\begin{align*}
\cX&=\Proj R[a(x)w^x\vartheta, x\in X],\\
U_n&=\Spec R[a(x)w^x/a(n)w^n, x\in X]\quad (n\in X)\\
&=\Spec R[(a(x)/a(n))w^{x-n}],\\
\cX_{\formal}/Y&=(\Proj R[a(x)w^x\vartheta, x\in X])_{\formal}/Y.
\end{align*}

Let $Q''_{\formal}:=\cX_{\formal}/Y$. Let $n=0$ for simplicity. Then we have 
\begin{align*}
U_0
&=\Spec R[qw_1, qw_1w_2, qw_2, qw_1^{-1}, qw_1^{-1}w_2^{-1}, 
qw_2^{-1}]\\
&\simeq\Spec k(0)[u_i; 0\leq i\leq 5]/(u_{i-1}u_{i+1}-qu_i,u_iu_{i+3}-q^2)\\
(U_0)_0&\simeq\Spec k(0)[u_i; 0\leq i\leq 5]/(u_iu_j; |i-j\,\text{(mod}\ 6)|\geq 2),
\end{align*}where $(U_0)_0=U_0\otimes k(0)$. 

We have a PSQAS $(Q''_{\formal})_0$. This PSQAS $(Q''_{\formal})_0$ 
is a union of 
$\ell m$ copies of $\bP^2$, whose configuration 
is just the same as the Delaunay decomposition 
on the right hand side in Fig.~4. 
The first 
horizontal chain of $\ell$ rational curves 
is identified with the $m$-th horizontal chain of $\ell$ rational curves 
without shifting on each rational curve, while 
the first 
vertical chain of $m$ rational curves 
is identified with the $\ell$-th vertical chain of $m$ rational curves without shifting on each rational curve. The PSQAS $(Q''_{\formal})_0$ is 
a level-$\cG_H$ PSQAS for $H=(\bZ/\ell\bZ)\oplus(\bZ/m\bZ)$.
\end{subcase}

\begin{subrem}
Gunji \cite{Gunji06} studied the defining equations of 
the universal abelian surface 
with level three structure. 
His universal abelian surface is the same as our universal PSQAS
over the moduli space $SQ_{2,K}$ 
when $K=H\oplus H^{\vee}$, $H=(\bZ/3\bZ)^{\oplus 2}$ 
and the base field is $\bC$.
 He proved that the level three 
universal abelian surface is 
the intersection of 9 quadrics and 4 cubics
of $\bP^8\times_{\cO_3}{SQ_{2,K}}\times_{\cO_3}{\bC}$ 
\cite[Theorem~8.3]{Gunji06}. 
In his article Gunji determines 
the fibers only partially \cite[pp.~95-96]{Gunji06}.
\par
By our study \cite[Theorem~11.4]{Nakamura99} (Theorem~\ref{thm:SQgK}),
any fiber of the universal PSQAS over $SQ_{2,K}$ 
is a smooth abelian surface, or
a cycle of 3 rational elliptic surfaces 
in Subsec.~\ref{subsec:PSQAS in dim two partially deg},  
with $\ell=m=3$, or else 
one of the singular surfaces in Cases~\ref{subcase:first} 
or \ref{subcase:second} with $\ell=m=3$. 
\end{subrem}

\begin{subrem}Here we explain only a little 
about the local structure of $SQ_{g,K}$ 
for $g=2$. It turns out that the local structure of 
$SQ_{g,K}$ is the same as that of a toroidal compactification,   
the second Voronoi compactification. \par
Let $X$ be a lattice of rank two, 
$B(x)$ the bilinear form on $X$ given in Case~\ref{subcase:second}
$$B(x)=x_1^2-x_1x_2+x_2^2.
$$ 

The Voronoi cone $V_B$ with center $B$ is defined to be
\begin{align*}
V_B&:=\{\beta : \text{positive definite bilinear form on $X$ with 
$\Del_{\beta}=\Del_B$}\}\\
&=\left\{\beta(x):=(\beta_{12}+\beta_{13})x_1^2-2\beta_{12}x_1x_2
+(\beta_{12}+\beta_{23})x_2^2; \beta_{ij}>0\right\}.
\end{align*}

We define a chart $T$ and 
a semi-universal covering $\cX$ over $T$ to be 
\begin{gather*}T:=T(V_B):=\Spf W(k)[[q_{ij};i<j]],\\
\cX=\Proj W(k)[[q_{ij};i<j]][a(x)w^x\vartheta; x\in X]
\end{gather*}
where $W(k)$ is the Witt ring of $k$ 
(see Subsec.~\ref{subsec:pro-representability}),
$q_{ij}=q^{\beta_{ij}}$ and
$$a(x):=q^{\beta^*(x)}:
=q_{13}^{(x_1^2+x_1)/2}q_{23}^{(x_2^2+x_2)/2}
q_{12}^{(x_1^2-2x_1x_2+x_2^2+x_1-x_2)/2},
$$where 
$\beta^*(x)=\frac{1}{2}\beta(x)+\frac{1}{2}r(x)$ and 
$r(x)=(\beta_{13}+\beta_{12})x_1+(-\beta_{12}+\beta_{23})x_2$.

Let $\cL_{\cX}$ be the invertible sheaf $O_{\cX}(1)$ 
on $\cX$. We define the action 
of the lattice $X$ on $\cX$ by 
$$S_z^*(a(x)w^x\vartheta)=a(x+z)w^{x+z}\vartheta.$$  

Let $(\cX_{\formal},\cL_{\formal})$ be the formal 
completion of $(\cX,\cL_{\cX})$ 
along the closed subscheme $\cX_0$ of $\cX$ given by $q_{ij}=0$. 
Let $Y$ be a sublattice of $X$ of finite index.  
We take the formal quotient of $\cX_{\formal}$ by $Y$ 
$$(Q_{\formal},\cL_{\formal}):=(\cX_{\formal},\cL_{\formal})/Y,$$ 
 where
$Q_{\formal}\otimes k(0)\simeq Q_0''$ 
if $Y$ is the same as in Case~\ref{subcase:second}. 
Moreover $(Q_{\formal},\cL_{\formal})$ is a semi-universal PSQAS over $T$. 
In other words, the deformation functor of 
$(Q_{\formal},\cL_{\formal})\otimes k(0)$ is 
pro-represented by $W(k)[[q_{ij};i<j]]$.
Compare \cite{Nakamura75} and Subsec~\ref{subsec:deform of AV}.

Let $\tau=(\tau_{ij})$ be a $2\times 2$ complex symmetric 
matrix with positive imaginary part, and set 
$$q_{12}=e^{-2\pi i\tau_{12}}, q_{13}=e^{2\pi i(\tau_{11}+\tau_{12})}, 
q_{23}=e^{2\pi i(\tau_{12}+\tau_{22})}.$$

These are regular parameters 
of $SQ_{2,K}$  at $(Q_{\formal},\cL_{\formal})$ 
for any $K$ with $e_{\min}(K)\geq 3$. 
This is also an infinitesimally local chart 
of the Mumford toroidal compactification, which is in this case 
the so-called Voronoi compactification, or to be a little more precise, 
the Mumford toroidal compactification associated to the second
Voronoi decomposition and some 
arithmetic subgroup of $\Spl(4,\bZ)$.  See \cite{Namikawa80}.
\end{subrem} 
\subsection{Nonseparatedness of a naive moduli}
\label{subsec:nonseparatedness}
We shall explain here 
how a naive generalization of classical level-$n$ structure 
results in a nonseparated compactification of 
the moduli of abelian varieties. See \cite{Nakamura75}.
\par
In three dimensional case, let $X$ be a lattice of rank 3. 
We choose 
$$B=\begin{pmatrix}
2&-1&0\\
-1&2&-1\\
0&-1&2
\end{pmatrix}.
$$ 

The level-1 PSQASes $(P_0,\cL_0)$ associated to $B$ are parameterized by  
3 nontrivial parameters \cite[p.~197]{Nakamura75}. 

Let $\Del_B/X$ be the quotient of 
the Delaunay decomposition $\Del_B$ by the translation action of $X$.
Then $\Del_B/X$ consists of 
three three-dimensional cells (two tetrahedra and an octahedron), 
eight two-dimen-\linebreak sional cells 
and six one-dimensional cells and a 0-dimensional cell 
\cite[pp.~195-196]{Nakamura75}.  
Each level-1 PSQAS $(P_0,\cL_0)$ has three irreducible components,  
two (say, $T_1, T_2$) of which are $\bP^3$ (modulo $X$ action) and 
the third (say, $O$) of which
is a rational variety distinct from $\bP^3$. Each of the three irreducible   
components is a compactification of $\bG_m^3$. \par
It follows that there are two different 
types (modulo $\Aut(P_0)$) of embedding  
of $\bG_m^3$ into $(P_0,\cL_0)$, that is, 
$\bG_m^3\subset T_k$ and $\bG_m^3\subset O$. 
Therefore there is a pair of 
$R$-PSQASes $(P',\cL')$ and $(P'',\cL'')$ 
such that 
$$(P'_{\eta},\cL'_{\eta})\simeq (P''_{\eta},\cL''_{\eta}),\ 
(\bG_m^3\subset P'_0)\not\simeq (\bG_m^3\subset P''_0).
$$
  
This also implies that there are two inequivalent classes of classical 
level-$n$ structures on the \'etale $(\bZ/n\bZ)^3$-covering 
$(P'_0,\cL'_0)$ of $(P_0,\cL_0)$ as the limits of the same (isomorphic) 
generic fiber. 
This shows that a naive generalization of classical level-$n$ structure 
will lead us to a nonseparated moduli.

\section{The $G$-action and the $G$-linearization}
\label{sec:G action G linearization}

Let $G$ be a group (scheme). 
The purpose of this section is to prove compatibility of various definitions 
about $G$-linearization.  

\subsection{The $G$-linearization}
\label{subsec:G-linearization}
\begin{subdefn}\label{subdefn:action and linearization}
{\em A $G$-linearization on $(Z,L)$}  is by definition the data
$\{(T_g,\phi_g);g\in G\}$ satisfying the conditions 
\begin{enumerate}
\item[(i)] $T_g$ is an automorphism of $Z$, 
such that $T_{gh}=T_gT_h$, $T_1=\id_Z$, 
\item[(ii)] $\phi_g:\cL\to T_{g}^*(\cL)$ 
is a bundle isomorphism with $\phi_1=\id_L$, 
\item[(iii)] $\phi_{gh}=(T_{h}^*\phi_g)\phi_h$ for any 
$g, h\in G((T)$.
\end{enumerate}

We say that $(Z,L)$ is {\em $G$-linearized} 
if the above conditions are true. 
\end{subdefn}

\begin{subrem}\label{subrem:linearization of tensor}
If $L$ and $L'$ are $G$-linearized, then $L\otimes L'$ is also 
$G$-linearized. 
\end{subrem}
\begin{subdefn}
If $(Z,L)$ is $G$-linearized, then we define a $G$-action $\tau$ 
on the pair $(Z,L)$.  
Via the isomorphism $L \overset{\phi_h}{\longrightarrow} T^*_{h}(L)$, 
for $x\in Z$, $\zeta\in L_x$, we define 
\begin{equation}\label{eq:action tau}
\tau(h)(z, \zeta):=(T_{h}(z), \phi_h(z)\zeta).
\end{equation}
\end{subdefn}

\begin{claim}
$\tau$ is an action of $G$ on $(Z,L)$.
\end{claim}
\begin{proof}
Via the isomorphisms
\begin{equation*}
L \overset{\phi_h}{\longrightarrow} T^*_{h}(L)
\overset{T^*_{h}\phi_g}{\longrightarrow}
T^*_{h}(T^*_{g}(L))=T^*_{gh}(L), 
\end{equation*} we see 
\begin{align*}
\tau(g)\left(\tau(h)(z, \zeta)\right)&
= \tau(g)\cdot (T_{h}(z), \phi_h(z)\zeta)\\
&= (T_{g}(T_{h}(z)), \phi_g(T_{h}(z))\phi_h(z)\zeta)\\
&=(T_{gh}(z), (T_{h}^*\phi_g\cdot\phi_h)(z)\zeta)\\
&=(T_{gh}(z), \phi_{gh}(z)\cdot\zeta)=\tau(gh)(z, \zeta).
\end{align*}

Hence $\tau$ is an action of $G$.
\end{proof}

Finally we note that if we are given an action 
$\tau$ of $G$ on the pair $(Z,L)$ of 
a scheme $Z$ and a line bundle $L$ on $Z$,
then we have a $G$-linearization of $L$. 
In fact, $\tau$ is an action of $G$ iff $T_{gh}=T_{g}T_{h}$ and 
$\phi_{gh}=T_{h}^*\phi_g\cdot\phi_h$.

\begin{claim}
\label{claim:defn of rhoL}\text{$($\cite[p.~295]{Mumford66}$)$}
Associated to a given $G$-action $\tau$ on $(Z,L)$, 
 we define a map $\rho_{\tau,L}(g)$ of $H^0(Z,L)$ to be
\begin{equation}\label{eq:definition of rho_L}
\rho_{\tau,L}(g)(\theta):=T_{g^{-1}}^*(\phi_g(\theta))\ 
\text{for any $g\in G$ and any $\theta\in H^0(Z,L)$.}
\end{equation}
Then $\rho_{\tau,L}$ is a homomorphism.
\end{claim}
\begin{proof}We see
\begin{align*}
\rho_{\tau,L}(gh)(\theta)&=T_{h^{-1}g^{-1}}^*(\phi_{gh}\theta)
=T_{g^{-1}}^*\{T_{h^{-1}}^*(T_h^*\phi_g\cdot\phi_h\theta)\}\\
&=T_{g^{-1}}^*\{T_{h^{-1}}^*(T_h^*\phi_g)\cdot (T_{h^{-1}}^*\phi_h\theta)\}\\
&=T_{g^{-1}}^*\{\phi_g\cdot (T_{h^{-1}}^*\phi_h\theta)\}
=\rho_{\tau,L}(g)\rho_{\tau,L}(h)(\theta).
\end{align*}
\end{proof}

\subsection{The $G$-linearization of $O_{\bP(V)}(1)$}
\label{subsec:Linearization of O(1)}
Let $R$ be any ring. 
Suppose we are given an action of a group $G$ on an $R$-free module 
$V$ of finite rank, in other words, 
a homomorphism $\rho:G\to \End(V)$. 
Let $V^{\vee}:=\Hom(V,R)$ be the dual of $V$, 
$\bP(V)$ the projective space with 
$V=H^0(\bP(V),O_{\bP(V)}(1))$,
$\bH=O_{\bP(V)}(1)$ the hyperplane bundle of $\bP(V)$. Then $V^{\vee}$ 
admits a natural affine $R$-scheme structure $\bV^{\vee}$ defined by 
$$\bV^{\vee}=\Spec \Sym V:=\Spec \bigoplus_{n=0}^\infty S^nV.
$$

The action $\rho$ of $G$ on $V$ induces an action of $G$ 
on $S^nV$, hence on $\Sym V$, hence on $\bV^{\vee}$, hence 
on the pair $(\bP(V),\bV^{\vee}-\{0\})$ of schemes. 
We note that $\bV^{\vee}-\{0\}$ is 
a $\bG_m$-bundle over $\bP(V)$ associated with the dual of 
the hyperplane bundle $\bH$ of $\bP(V)$. Hence 
the action $\rho$ of $G$ on $V$ induces 
the action on the pair $(\bP(V),\bH)$ of schemes. 
\par
Let $S$ be any $R$-scheme 
and $P\in \bP(V)(S)$ any $S$-valued point. 
By choosing affine coverings $U_i:=\Spec A_i$ 
of $S$ if necessary, 
$P$ is a collection of $P_i\in \bP(V)(U_i)$ of 
(the equivalence class of) 
the points given by 
$$\gamma_{P_i}\in\Hom(V,A_i)$$
such that the ideal of $A_i$ generated by $\gamma_P(V)$ is $A_i$, 
where $\gamma_{P_i}\sim\gamma_{Q_i}$ iff 
$\gamma_{Q_i}=c\gamma_{P_i}$ for some $c\in A_i^{\times}$. 
Hence there are $c_{ij}\in A^{\times}_{ij}
:=\Gamma(O_{U_i\cap U_j})^{\times}$ such that 
$\gamma_{P_i}=c_{ij}\gamma_{P_j}$. 
In what follows, we suppose $S=U_i$ 
for simplicity and we identify $P$ with $\gamma_P$. \par
We define an action of $G$ on $(\bP(V),\bV^{\vee}\setminus\{0\})$ by
\begin{equation}\label{defn:Svee}
S^{\vee}(g)([\gamma_P],\gamma_P):=
([\gamma_P\circ\rho(g^{-1})], \gamma_P\circ\rho(g^{-1})).
\end{equation}Then we see, 
\begin{align*}
S^{\vee}(gh)(\gamma_P)&=\gamma_P\circ \rho((gh)^{-1})
=\gamma_P\circ \rho(h^{-1})\rho(g^{-1})\\
&=S^{\vee}(h)(\gamma_P)\rho(g^{-1})=
S^{\vee}(g)S^{\vee}(h)(\gamma_P).
\end{align*}
Thus we have an action of $G$
on the pair $(\bP(V),\bV^{\vee}\setminus\{0\})$ 
by $\bG_m$-bundle automorphisms. 

\begin{subdefn}\label{subdefn:G action on H}
The action $S^{\vee}(g)$ 
of $g\in G$ on $(\bP(V),\bV^{\vee}\setminus\{0\})$ 
induces an action on $(\bP(V), \bH)$, which we denote by $S(g)$. 
\end{subdefn}

\begin{rem}\label{rem:rho=rho}
Let $R$ be any ring, $V$ an $R$-free module of finite rank, and 
$\rho:G\to\End(V)$ an action of $G$ on $V$. Let $V^{\vee}:=\Hom(V,R)$ and 
$\langle \phantom{\gamma},\phantom{f}\rangle:V^{\vee}\times V\to R$ 
the dual pairing. Using this pairing we have a dual action ${}^t\rho$ 
of $G$ on $V^{\vee}$ 
such that 
$$\langle {}^t\rho(g)\gamma, F\rangle:=\langle\gamma,\rho(g)F\rangle,$$
where $\gamma\in V^{\vee}$, and $F\in V$. 
Then ${}^t\rho(gh)={}^t\rho(h){}^t\rho(g)$. Thus 
this is made into a left action of $G$ on $\bP(V)$ 
by taking $T_g(\gamma):={}^t\rho(g^{-1})(\gamma)$. This $T_g$ is the same as $S^{\vee}(g)$ in Subsec.~\ref{subsec:Linearization of O(1)} because 
\begin{align*}
T_g(\gamma)(F)&=\langle{}^t\rho(g^{-1})(\gamma),F\rangle
=\langle\gamma,\rho(g)^{-1}F\rangle\\
&=\gamma(\rho(g^{-1})F)=S^{\vee}(g)(\gamma)(F).\end{align*}

Since we have the action $T_g$ on $\bP(V)$, 
Claim~\ref{claim:defn of rhoL} 
defines a homomorphism $\rho_{T,\bH}$ (well known as 
the contragredient representation of $T_g$). Then we have 
\begin{align*}
(\rho_{T,\bH}(g)F)(\gamma):&=F(T_{g^{-1}}\gamma)=F({}^t\rho(g)\gamma)\\
&=\langle{}^t\rho(g)\gamma,F\rangle=\langle\gamma,\rho(g)F\rangle
=(\rho(g)F)(\gamma),\end{align*}
where $x\in V^{\vee}\setminus\{0\}$, $F\in V$. 
Hence $\rho_{T,\bH}=\rho$. \par
This justifies our notation $(C,i,U_H)$ (resp. $(Z,i,U_H)$) 
in Lemma~\ref{lemma:G(3) uniqueness} (resp.   
in Theorem~\ref{thm:SQgK}) where we 
indicate the action on $(C,L)$ or $(Z,L)$ 
induced from $U_H$ simply by $U_H$. 
\end{rem}

\subsection{$G$-invariant closed subschemes}
\label{subsec:G equiv closed immersion}
Let $R$ be any ring, $V$ an $R$-free module of finite rank, and 
$G$ any subgroup of $\PGL(V)$. 
If $Z$ be a $G$-invariant closed subscheme of $\bP(V)$ with $L=O_Z(1)$, 
then the $G$-action of $(\bP(V), \bH)$ keeps $(Z,L)$ stable, hence 
we have an action of $G$ on the pair $(Z,L)$. This gives rise to a 
$G$-linearization of $(Z,L)$. \par
Conversely 
\begin{claim}\label{claim:G equiv embedding}
Let $(Z,L)$ be an $R$-scheme
with $L$ a $G$-linearized line bundle on $Z$, and 
$V$ a $G$-submodule of $H^0(Z,L)$. 
{\em Suppose that 
$V$ is $R$-free of finite rank and very ample.} Then  
the natural morphism $(\psi,\Psi):(Z,L)\to(\bP(V),\bH)$ is 
a $G$-equivariant closed immersion. 
\end{claim}

This is a corollary to the following
\begin{claim}\label{claim:G equiv morphism}
Let $(Z,L)$ be an $R$-scheme
with $L$ a $G$-linearized line bundle on $Z$, and 
$V$ a $G$-submodule of $H^0(Z,L)$. 
{\em Suppose that 
$V$ is $R$-free of finite rank and base point free.} Then  
\begin{enumerate}
\item there is a $G$-action $S$ on $(\bP(V),\bH)$ 
in Subsec.~\ref{subsec:Linearization of O(1)},
\item
the natural morphism $(\psi,\Psi):(Z,L)\to(\bP(V),\bH)$ is 
$G$-equivariant. 
\end{enumerate}
\end{claim}
\begin{proof}
By Claim~\ref{claim:defn of rhoL}, $H^0(X,L)$ is a $G$-module. 
By the assumption $V$ is a $G$-submodule of $H^0(X,L)$.  Then 
by Subsec.~\ref{subsec:Linearization of O(1)} we have a $G$-action 
$S$ on $(\bP(V),\bH)$.  With the notation 
in Subsec.~\ref{subsec:Linearization of O(1)}, 
we define the map $\psi$ by 
$\gamma_{\psi(z)}(\theta)=\theta(z)$ 
for $\theta\in V=H^0(Z,L)$. This defines a natural map 
$(\psi,\Psi):(Z,L)\to (\bP(V),\bH)$ because $L=\psi^*\bH$. 
We prove that with respect to the $G$-actions 
$\tau$ on $(Z,L)$ and $S$ on $(\bP(V),\bH)$,  
$(\psi,\Psi)$ is $G$-equivariant. 
Let $(z,\zeta)\in (Z,L)$ and $P=\psi(z)$. 
Then we have 
\begin{equation}\label{eq:action tau and S}
\tau(g)(z,\zeta)=(T_g(z),\phi_g(z)\zeta),\ 
(\psi,\Psi)(z,\zeta)=(\psi(z),\zeta).
\end{equation}
Since $(T_g^*\phi_{g^{-1}})\phi_g=\phi_1=\id_L$ 
by Definition~\ref{subdefn:action and linearization}~(iii), we see
\begin{align*}
\gamma_{\psi(z)}\circ \rho_L(g^{-1})(\theta)
&=\gamma_{\psi(z)}(T_g^*(\phi_{g^{-1}}\theta))\\
&=(T_g^*\phi_{g^{-1}}(z)T_g^*(\theta)(z)
=\phi_{g}^{-1}(z)T_g^*(\theta)(z)\\
&=\phi_{g}(z)^{-1}\theta(T_gz)
=\phi_{g}(z)^{-1}\gamma_{\psi(T_gz)}(\theta),
\end{align*}whence $[\gamma_{\psi(z)}\circ \rho_L(g^{-1})]
=[\gamma_{\psi(T_gz)}]=\psi(T_gz)$. 
By (\ref{defn:Svee}), regarding  $\zeta^{-1}$  
as  the (rational) fiber coordinate of $L^{\vee}$, we have 
\begin{align*}
S^{\vee}(g)(\psi,\Psi)(z,\zeta^{-1})&=([\gamma_{\psi(z)}\circ \rho_L(g^{-1})], 
\gamma_{\psi(z)}\circ \rho_L(g^{-1})\zeta^{-1})\\
&=([\gamma_{\psi(T_gz)}], \gamma_{\psi(T_gz)}\phi_g(z)^{-1}\zeta^{-1}), 
\end{align*}whence the fiber coordinate $\zeta^{-1}$ is transformed into 
$\phi_g(z)^{-1}\zeta^{-1}$ because $\psi(z)$ (resp. $\psi(T_gz)$) 
is a generator of the fiber of $\bH$. Hence $S^{\vee}(g)$ induces 
the transformation $\zeta\mapsto \phi_g(z)\zeta$ on $L$. Thus 
with the notation of (\ref{eq:action tau and S})
\begin{align*}
S(g)(\psi,\Psi)(z,\zeta)&
=([\gamma_{\psi(z)}\circ \rho_L(g^{-1})], \phi_g(z)\zeta)
=(\psi(T_g(z)),\phi_g(z)\zeta)\\
&=(\psi,\Psi)(T_g(z),\phi_g(z)\zeta)=(\psi,\Psi)\tau(g)(z,\zeta).
\end{align*} 

This proves that $(\psi,\Psi)$ is $G$-equivariant. 
\end{proof}


\subsection{The $G$-linearization in down-to-earth terms}
We quote this part from \cite[p.94]{Nakamura10}. 
The following enables us to understand $G_H$-linearization in 
down-to-earth terms.
\begin{claim}\label{claim:G-inv affine trivializing L}
Let $T=\Spec R$, and 
$G$ a finite group. 
Let $Z$ be a positive-dimensional $R$-flat projective scheme. 
$L$ an ample $G$-linearized line bundle on $Z$. 
Then for any point $z\in Z$, 
there exists a $G$-invariant open affine $R$-subscheme\/ $U$ 
of $Z$ such that $z\in U$ and 
$L$ is trivial on $U$.
\end{claim}
\begin{proof}
See \cite[Lemma~4.9]{Nakamura10}.
\end{proof}

Let $T=\Spec R$ be any affine scheme, and 
$G$ a finite group. 
Let $Z$ be a positive-dimensional $T$-flat projective scheme. 
Let $m:G\times_{R} G\to G$ be the multiplication of $G$, and
$\sigma:G\times_{R} Z\to Z$ an action of $G$ on $Z$. Let 
$L$ be an ample $G$-linearized line bundle on $Z$. 
The action $\sigma$ satisfies the condition:
\begin{equation}
\label{eq:commutativity of Glinearization}
\sigma(m\times\id_Z)=\sigma(\id_G\times\sigma).
\end{equation}

Now we shall give a concrete description 
of the $G$-linearization of $(Z,L)$ 
by using a nice open affine covering of $Z$. 
By Claim~\ref{claim:G-inv affine trivializing L}, we can choose 
an affine open covering $U_j:=\Spec(R_j)$\ $(j\in J)$ of $Z$ 
such that each $U_j$ is $G$-invariant and 
the restriction of $L$ is trivial on each $U_j$.\par
The induced bundles $\sigma^*L$, (resp. 
$(\id_G\times\sigma)^*\sigma^*(L)$, 
$(m\times\id_Z)^*\sigma^*(L)$) are
all trivial on $G\times_{R} U_j$ 
(resp. $G\times_{R} G\times_{R} U_j$ 
or  $G\times_{R} G\times U_j$) with 
the same fiber-coordinate as $L_{U_j}$. 
Let $\zeta_j$ be a fiber-coordinate of $L_{U_j}$.\par
Now we assume that $G$ is a constant finite group (scheme over $T$). 
Since $G$ is affine, let $A_G:=\Gamma(G,O_G)$ be the Hopf algebra of $G$. 
See \cite{W79}.
Then the isomorphism $\Psi:p_2^*L\to \sigma^*(L)$ over $U_j$ is 
multiplication by a unit 
$\psi_j(g,x)\in (A_G\otimes_{R} R_j)^{\times}$ 
at $(g,x)\in G\times_{R} U_j$. 
Let $A_{jk}(x)$ be the one-cocycle defining $L$. Then 
$\sigma^*(L)$ is defined by the one-cocycle $\sigma^*A_{jk}(x)$. 
Hence $\Psi:p_2^*L\to \sigma^*(L)$ over $U_j$ and $U_k$ are related by
$$\psi_j(g,x)=\frac{A_{jk}(gx)}{A_{jk}(x)}\psi_k(g,x).
$$

This is the condition (ii) of 
Definition~\ref{subdefn:action and linearization}. 
The condition (iii) of Definition~\ref{subdefn:action and linearization} 
is expressed as 
$$\psi_j(gh,x)=\psi_j(g,hx)\psi_j(h,x).
$$

\section{The moduli schemes $A_{g,K}$ and $SQ_{g,K}$}
\label{sec:moduli AgK SQgK}
Let $H=\bigoplus_{i=1}^g(\bZ/e_i\bZ)$ be a finite Abelian group 
with $e_i|e_{i+1}$, $e_{\min}(H):=e_1$, $K=H\oplus H^{\vee}$, 
$N=|H|=\prod_{i=1}^ge_i$ and $\cO_N=\bZ[\zeta_N,1/N]$.  
The purpose of this section is to construct two schemes, 
projective (resp. quasi-projective)
$SQ_{g,K}$ (resp. $A_{g,K}$). We will see later that 
$A_{g,K}$ is the fine moduli scheme of abelian varieties, which is a 
Zariski open subset of the projective scheme $SQ_{g,K}$.
As a (geometric) point set, $SQ_{g,K}$ is the set of all 
GIT-stable degenerate abelian schemes (Theorem~\ref{thm:stability of PSQAS}).

\begin{thm}
\label{thm:embedding of (Q0,L0)}
Let\/ $V_H:=\bigoplus_{\mu\in H^{\vee}}
\cO_Nv(\mu)$. Let $(Z,L)$ be a  PSQAS over $k(0)$, 
$(Q,\cL)$ a PSQAS over a CDVR $R$ with $\ker\lambda(\cL)\simeq K$ 
such that 
$(Z,L)\simeq (Q,\cL)\otimes k(0)$ and 
the generic fiber $(Q_{\eta},\cL_{\eta})$ is an abelian variety.
Let $\cV_0:=\Gamma(Q,\cL)\otimes k(0)$. Then 
\begin{enumerate}
\item $\dim_{k(0)}\cV_0=|H|$, and 
$\cV_0\simeq V_H\otimes k(0)$ as $\cG_H$-modules, 
\item $\cV_0$ is uniquely determined by $(Z,L)$, and 
independent of the choice of $(Q,\cL)$,
\item if $e_{\min}(H)\geq 3$, then both $\Gamma(Q,\cL)$ and $\cV_0$ 
are very ample,  
\item if $e_{\min}(H)\geq 3$, then $(Z,L)$ is 
embedded $\cG_H$-equivariantly into  $(\bP(V_H),\bH)$ by 
the linear subspace $\cV_0$ via the isomorphism  
$\cV_0\simeq V_H\otimes k(0)$ as $\cG_H$-modules.
\end{enumerate}
\end{thm}
\begin{proof}
By Theorem~\ref{thm:refined stable reduction}, 
there exists a CDVR $R$ and a 
projective flat morphism $\pi:(Q,\cL)\to \Spec R$ 
(resp. $\pi:(P,\cL)\to \Spec R$) such that 
$(Q_0,\cL_0)\simeq (Q,\cL)\otimes k(0)$, and 
$P$ is the normalization of $Q$ with $P_0$ reduced. 
Then by \cite[Theorems~3.9 and 4.10]{Nakamura99}, for instance,  
here in the totally degenerate case, we have 
\begin{align*}
\Gamma(P_0,\cL_0)&=
\left\{\sum_{\barx\in X/Y}c(\barx)\sum_{y\in Y}\ a(x+y)w^{x+y}\otimes k(0);
c(\barx)\in k(0)
\right\},\\
\Gamma(P,\cL)&=\left
\{\sum_{\barx\in X/Y}c(\barx)\sum_{y\in Y}\ a(x+y)w^{x+y};
c(\barx)\in R
\right\}, 
\end{align*}where $\barx$ is the class of $x \mod Y$.
 Hence $\Gamma(Q,\cL)=\Gamma(P,\cL)$ because 
$\Gamma(Q,\cL)$ is an $R$-submodule of $\Gamma(P,\cL)$, 
and any of the generators of $\Gamma(P,\cL)$ 
belongs to $\Gamma(Q,\cL)$ by the construction 
in Subsec.~\ref{subsec:construction}. 
Hence 
$$\cV_0:=\Gamma(Q,\cL)\otimes k(0)=\Gamma(P,\cL)\otimes k(0)
=\Gamma(P_0,\cL_0),$$ 
By \cite[Corollary~3.9]{Nakamura10} $(P_0,\cL_0)$ 
is uniquely determined by $(Q_0,\cL_0)$, 
whence $\cV_0$ is independent of the choice of $(Q,\cL)$. 
This proves (2).
\par

This $\cV_0$ 
is very ample  and of rank $|H|$ by Lemma~\ref{lemma:cohomology}~(5)
if $e_{\min}(H)\geq 3$. Hence so is  $\Gamma(Q,\cL)$.
Since $(Z,L)$, hence $(Q_0,\cL_0)$, hence $(P_0,\cL_0)$ admit 
a $\cG_H$-action, $\cV_0=\Gamma(P_0,\cL_0)$ is a $\cG_H$-module.
Hence by Claim~\ref{claim:G equiv embedding},
$(Z,L)$ is {\it 
embedded $\cG_H$-equivariantly into $(\bP(V_H),\bH)$}.
\end{proof}
\begin{subdefn}\label{subdefn:chara subspace} 
Let $(Z,L)=(Q_0,\cL_0)$ be a $k(0)$-PSQAS. 
We call $\cV_0$ {\em a characteristic subspace} 
of $\Gamma(Z,L)$, and denote $\cV_0$ by $V(Z,L)$. 
This $\cV_0$ is uniquely determined by $(Z,L)$ because 
$\cV_0=\Gamma(P_0,\cL_0)$ and $(P_0,\cL_0)$ is 
uniquely determined by $(Z,L)=(Q_0,\cL_0)$. 
\end{subdefn}
\begin{subrem}In connection with the GIT-stability of $(Z,L)$, 
it is more important to know whether $V(Z,L)$ is
very ample than to know whether $L$ (that is, $\Gamma(Z,L)$) is very ample.
See \cite[Theorem~11.6]{Nakamura99} 
and Theorem~\ref{thm:stability of PSQAS}.
However \cite[p.~697]{Nakamura99} 
conjectures $V(Z,L)=\Gamma(Z,L)$. 
\end{subrem}

\begin{subdefn}\label{subdefn:G(A,L)}
Let $k$ be an algebraically closed field
 with $k\ni 1/N$ and $H$ 
a finite Abelian group with $|H|=N$. 
Let $(A,L)$ be an abelian variety over $k$. 
Then we define $\cG(A,L)$ to be the bundle automorphism group which 
induces translations of $A$ by $\ker(\lambda(L))$. 
If $\ker(\lambda(L))\simeq K:=H\oplus H^{\vee}$, then 
$\cG(A,L)\simeq \cG_H$ by Lemma~\ref{lemma:chara GH action}. \par
Let $K(A,L):=\ker(\lambda(L))=\cG(A,L)/\bG_m$.
\end{subdefn}

\begin{subrem}\label{subrem:chara GH action}
Let $k$ be an algebraically closed field
 with $k\ni 1/N$ and  
$(Z,L)$ any PSQAS over $k$. 
Hence there exists 
a PSQAS $(Q,\cL)$ over a CDVR $R$ 
such that $(Z,L)\simeq (Q_0,\cL_0)$  
 and the generic fiber $(Q_{\eta},\cL_{\eta})$ 
is an abelian variety with 
$\ker\lambda(\cL_{\eta}))\simeq K=H\oplus H^{\vee}$.
Then the natural $\cG_H$-action ($=\cG(Q_{\eta},\cL_{\eta})$) 
on $(Q_{\eta},\cL_{\eta})$ 
extends to that on $(Q,\cL)$, whose restriction to $(Q_0,\cL_0)$
is the $\cG_H$-action on $(Z,L)$. 
We denote by $\cG(Z,L)$ the $\cG_H$-action on $(Z,L)$.  
This is determined by $(Z,L)$ 
uniquely up to an automorphism of $\cG_H$. 
Let $K(Z,L):=\cG(Z,L)/\bG_m$. 
\end{subrem}

\begin{subdefn}
\label{subdefn:chara GH action}
Let $(Z,L)$ be a PSQAS over $k$. 
We call the action 
$\tau:\cG_H\times (Z,L)\to (Z,L)$ of $\cG_H$  
{\em a characteristic $\cG_H$-action,} 
or simply {\em characteristic}, if $\tau$ 
induces the natural isomorphism in Remark~\ref{subrem:chara GH action}
$$\cG_H\overset{\cong}{\to} \cG(Z,L)\subset \Aut(L/Z),$$
where $\Aut(L/Z)$ is the bundle automorphism group of $L$ over $Z$. 
\end{subdefn}

\begin{subrem}Let $C$ be a planar cubic defined by 
$$x_0^3+\zeta_3x_1^3+\zeta_3^2x_2^3=0.$$
This cubic $C$ is $\cG(3)$-invariant, 
hence $\sigma$ and $\tau$ 
in Subsec.~\ref{subsec:noncommutative level} act on $C$. 
However $\tau$ is not a translation of $C$. See \cite[p.~712]{Nakamura99}. 
Therefore $\cG(3)$ on $C$ is not a characteristic $\cG(3)$-action of $C$.
\end{subrem}

\subsection{The level-$\cG_H$ structure}
\begin{subdefn}
\label{subdefn:level GH PSQAS}
Let $k$ be an algebraically closed 
field with $k\ni 1/N$. 
A 6-tuple 
$(Z,L,V(Z,L),\phi,\cG_H,\tau)$ or the triple $(Z,\phi,\tau)$ over $k$ is 
{\em a PSQAS with level-$\cG_H$ 
structure} or  {\em a level-$\cG_H$ PSQAS} if 
\begin{enumerate}
\item[(i)] 
$(Z,L)$ is a PSQAS $(Q_0,\cL_0)$ over $k$ with $L$ very ample,
\item[(ii)] $\tau:\cG_H\times (Z,L)\to (Z,L)$ 
is a characteristic $\cG_H$-action, 
\item[(iii)] $\phi:Z\to \bP(V_H)$ is a $\cG_H$-equivariant closed 
immersion (with respect to $\tau$) such that 
$V(Z,L)=\phi^*(V_H\otimes k)\subset\Gamma(Z,L)$. 
\end{enumerate}
\end{subdefn}

\begin{subdefn}
\label{subdefn:rho(phi,tau)}
For a level-$\cG_H$ PSQAS $(Z,\phi,\tau)$ over $k$, let 
\begin{equation}\label{eq:rho(phi,tau)}
\rho(\phi,\tau)(g)(v):=(\phi^*)^{-1}\rho_{\tau,L}(g)\phi^*(v)
\end{equation}for $v\in V_H$. 
\end{subdefn}

\begin{subrem}
\label{subrem:phi Phi}By Claim~\ref{claim:G equiv morphism}, 
the following condition (iv) is automatically 
satisfied by $(Z,L)$ in Definition~\ref{subdefn:level GH PSQAS} :  
\begin{enumerate}
\item[(iv)] $(\phi,\Phi):(Z,L)\to (\bP(V_H),\bH)$ 
 is a $\cG_H$-equivariant morphism (with respect to $\tau$)
where 
$\bH$ is the hyperplane bundle of $\bP(V_H)$ 
and $\Phi:L=\phi^*\bH\to \bH$ the natural bundle morphism.
That is, 
\begin{equation}\label{eq:equivariance}
(\phi,\Phi)\circ\tau(g)=
S(\rho(\phi,\tau)g)\circ(\phi,\Phi)\ \text{for any $g\in \cG_H$}
\end{equation}
 with the notation 
of Definition~\ref{subdefn:G action on H}. \theoremtextend
\end{enumerate}

We added (iv) here for notational convenience. We denote 
(iii) and (iv) together by $\phi\tau=S\phi$ or $\phi\tau(g)=S(g)\phi$ for any $g\in \cG_H$. 
\end{subrem}

\begin{subdefn}\label{subdefn:k-isom}
Two PSQASes $(Z,\phi,\tau)$ and $(Z',\phi',\tau')$ 
with level-$\cG_H$ structure 
are defined to be {\em isomorphic} iff 
there exists a $\cG_H$-isomorphism 
$f:(Z,L)\to (Z',L')$
such that $\phi'f = \phi$. \end{subdefn}

\begin{subrem}\label{subrem:k-isom}
In Definition~\ref{subdefn:k-isom}~(i), $V(Z,L)=f^*V(Z',L')$. 
Hence $f^*L'=L$ so that 
there always exists a $\cG_H$-isomorphism of bundles 
$(f,F(f)):(Z,L)\to (Z',L')$, 
that is, 
$$(f,F(f))\tau(g)=\tau'(g)(f,F(f))\quad\text{for any $g\in \cG_H$.}$$

The line bundle $L$ is a scheme over $Z$. The $\cG_H$-isomorphism 
$F(f):L\to L'$ is a $\cG_H$-isomorphism 
as a (line) bundle, which induces a 
$\cG_H$-isomorphism $f:Z\to Z'$. 
In what follows, we say this simply 
that $(f,F(f))$ {\em or 
$f:(Z,L)\to (Z',L')$ is a $\cG_H$-isomorphism of bundles}.
\end{subrem}

\begin{subdefn}\label{subdefn:generalization of Hesse cubic}
$(Z,\phi,\tau)$
is defined to be a rigid level-$\cG_H$ PSQAS, 
or a PSQAS with rigid level-$\cG_H$ structure if 
\begin{enumerate}
\item[(i)] $(Z,\phi,\tau)$ is a level-$\cG_H$ PSQAS, 
\item[(ii)] $\rho(\phi,\tau)=U_H$ : 
the Schr\"odinger representation of $\cG_H$.
\end{enumerate}
\end{subdefn}

\begin{subrem}
A rigid object in Definition~\ref{subdefn:generalization of Hesse cubic} 
is a natural generalization of a Hesse cubic. 
Lemma~\ref{sublemma:uniqueness of rigid PSQAS} 
shows that any PSQAS  $(Z,\phi,\tau)$ 
can be moved into a rigid one inside the same projective space. 
\end{subrem}

\begin{sublemma}
\label{sublemma:uniqueness of rigid PSQAS} Assume $e_{\min}(K)\geq 3$.
Then for a level-$\cG_H$ PSQAS  $(Z,\phi,\tau)$ over $k$, 
\begin{enumerate}
\item 
there exists a unique rigid level-$\cG_H$ PSQAS 
$(Z,\psi,\tau)$ isomorphic to $(Z,\phi,\tau)$,
\item there exists a 
unique $U_H$-invariant subscheme $(W,L)$ of $(\bP(V_H),\bH)$ 
such that $(W,i,U_H)\simeq (Z,\psi,\tau)$.
\end{enumerate}
\end{sublemma}
\begin{proof}
By Claim~\ref{claim:defn of rhoL}, we have
$$\rho(\phi,\tau)(gh)=\rho(\phi,\tau)(g)\rho(\phi,\tau)(h).$$
Hence $V_H$ is an irreducible $\cG_H$-module 
of weight one through $\rho(\phi,\tau)$.
By Schur's lemma, there exists $A\in\GL(V_H\otimes k)$ such that 
\begin{equation*}\label{eq:rho(phi,tau)}
U_H=A^{-1}\rho(\phi,\tau)A
=(\phi^*A)^{-1}\rho_{\tau,L}(g)(\theta)(\phi^*A).
\end{equation*} 
Hence it suffices to choose a closed immersion $\psi$ by 
$\psi^*=\phi^*A$. Then 
\begin{equation}\label{eq:rho(phi,tau)}
U_H=\rho(\psi,\tau)\ \text{and}\ (Z,\phi,\tau)\simeq (Z,\psi,\tau).
\end{equation} 

The uniqueness of $\psi$ follows from Schur's lemma
(Lemma~\ref{lemma:lemma of Schur}). In fact, 
suppose $U_H=\rho(\psi,\tau)=\rho(\phi,\tau)$. 
Let $\gamma:=(\phi^*)^{-1}(\psi^*)$. Then 
$$U_H=\rho(\phi,\tau)=\gamma\rho(\psi,\tau)\gamma^{-1}=\gamma U_H\gamma^{-1},$$
whence by Schur's lemma, $\gamma$ is a nonzero scalar. Hence $\psi=\phi$. \par
Finally we prove the second assertion. An example of $(W,i,U_H)$ is given 
by $(\psi(Z),i,U_H)$ by the first assertion. 
If we have another $U_H$-invariant PSQAS $(W',j,U_H)$ 
such that 
$(W,i,U_H)\simeq (W',j,U_H)$, there is 
an isomorphism 
$$f:(W,i,U_H)\to (W',j,U_H).$$ Hence $i=jf$. 
By the proof of the first assertion, $f^*$ is a nonzero scalar, hence 
$j=i$. Hence the closed subscheme $W$ is unique.  
 \end{proof}

\begin{sublemma}
\label{sublemma:autom trivial}Let $k$ 
be an algebraically closed field with $k\ni 1/N$. 
If $e_{\min}(H)\geq 3$, then 
any level-$\cG_H$ PSQAS  $(Z,\phi,\tau)$ has 
trivial automorphism group. 
\end{sublemma}
\begin{proof}Let $f$ be any 
isomorphism $f:(Z,\phi,\tau)\to (Z,\phi,\tau)$. Hence 
$f\tau(g)=\tau(g)f$ for any $g\in \cG_H$. Hence 
we have 
$$f^*\rho_{\tau,L}(g)=\rho_{\tau,L}(g)f^*\quad\text{on $V(Z,L)$ 
for any $g\in \cG_H$}.$$
Since $\rho_{\tau,L}$ is 
an irreducible representation of $\cG_H$ on $V(Z,L)$,
by Schur's lemma (Lemma~\ref{lemma:lemma of Schur}), 
$f^*$ is a scalar.  
Since $e_{\min}(H)\geq 3$, we have 
$\phi^{-1}:\phi(Z)\overset{\cong}{\to} Z$ is an isomorphism 
by Theorem~\ref{thm:embedding of (Q0,L0)}~(5).  
Since $f^*$ on $V(Z,L)$ is a nonzero scalar, 
$(\phi^*)^{-1}\circ f^*\circ(\phi^*)$ 
is a scalar isomorphism of $V_H\otimes k$, hence 
$\phi\circ f\circ \phi^{-1}$ 
is the identity of $\bP(V_H)$, hence it is 
the identity of $\phi(Z)$. 
Hence $f$ is the identity of $Z$.
\end{proof}

\begin{lemma}
\label{lemma:exists (A,L) with G(A,L)=G(K)}
Let $k$ be an algebraically closed field, 
let $H$ be a finite Abelian group, 
$H^{\vee}$ the Cartier dual of $H$, 
$K=H\oplus H^{\vee}$ the symplectic Abelian group and $N=|H|$.
If $k\ni 1/N$, then there exists  
a polarized abelian variety $(A,L)$ over $k$ 
such that the Heisenberg group $\cG(A,L)$ of $(A,L)$ 
is isomorphic to $\cG_H\otimes k$.
\end{lemma}
\begin{proof}
See \cite[Lemma~4.2]{Nakamura10}.
\end{proof}

\subsection{The Hilbert scheme $\hilb^{\chi(n)}$}
\label{subsec:Hilb}
Let $H$, $V_H$ and $\cG_H$ be the same as 
in Subsec.~\ref{defn:Heisenberg group}. 
Let $\hilb^{\chi(n)}$ be the Hilbert scheme parameterizing 
all the closed subscheme $(Z,L)$ of $\bP(V_H)$ with 
$\chi(Z,L^n)=n^g|H|=:\chi(n)$. 
Since $V_H$ is a $\cG_H$-module via $U_H$, 
$\cG_H$ acts on $(\bP(V_H),\bH)$, hence on $\hilb^{\chi(n)}$.
Let  
$$(\hilb^{\chi(n)})^{\cG_H\text{-inv}}$$ 
be the fixed point set of $\cG_H$ (the scheme-theoretic fixed points). 
This is a closed $\cO_N$-subscheme of $\hilb^{\chi(n)}$. 
Let $(Z_{\univ},L_{\univ})$ be 
the pull back to $(\hilb^{\chi(n)})^{\cG_H\text{-inv}}$
of the universal subscheme of 
$\bP(V_H)$ over $\hilb^{\chi(n)}$.  
Then there is an open $\cO_N$-subscheme $U_3$ 
of $(\hilb^{\chi(n)})^{\cG_H\text{-inv}}$ 
such that any geometric fiber 
of $(Z_{\univ},L_{\univ})$ is an abelian variety 
(with zero unspecified). 
It is clear that $\cG_H$ keeps $U_3$ stable. 
See \cite[Subsec.~11.1]{Nakamura99}.  
\par
Let $\Aut_{U_3}(Z_{\univ})$ be the relative automorphism group scheme of 
$(Z_{\univ})_{U_3}$ (see \cite[Subsec.~11.1]{Nakamura99}). 
We define a subset $U_4$  of $U_3$ to be  
\begin{align*}
U_4&=\left\{s\in U_3; 
\begin{matrix}
\text{the action of $\cG_H$ on $(Z_{\univ,s},L_{\univ,s})$ is\phantom{abc}}\\
\text{a translation of the abelian variety $Z_{\univ,s}$}
\end{matrix}
\right\}.
\end{align*}Since the subgroup of $\Aut_{U_3}(Z_{\univ})$ 
consisting of fiberwise translations 
is an (open and) closed 
subgroup $\bZ$-scheme of $\Aut_{U_3}(Z_{\univ})$, $U_4$ is 
a closed $\cO_N$-subscheme of $U_3$, which  is not empty
by Lemma~\ref{lemma:exists (A,L) with G(A,L)=G(K)}.

We denote $U_4$ by $A_{g,K}$ and we define $SQ_{g,K}$ to be 
the closure of $A_{g,K}$
(the minimal closed $\cO_N$-subscheme containing $A_{g,K}$)
\begin{equation}
\label{eq:SQgK}
SQ_{g,K}:=\overline{A_{g,K}}\subset (\hilb^{\chi(n)})^{G(K)\text{-inv}}.
\end{equation}

\begin{thm}
\label{thm:SQgK}
Let $H=\bigoplus_{i=1}^g(\bZ/e_i\bZ)$ with $e_i|e_{i+1}$ for any $i$ 
and $N=\prod_{i=1}^ge_i$. 
If $e_{\min}(H):=e_1\geq 3$, then 
for any algebraically 
closed field $k$ with $k\ni 1/N$, we have 
\begin{align*}
SQ_{g,K}(k)&=\left\{(Q_0,i,U_H);
\begin{matrix}
 Q_0:\text{a level-$\cG_H$ PSQAS\phantom{abcd}}\\
i:Q_0\subset\bP(V_H)\ \text{the inclusion}
\end{matrix}
\right\}
\end{align*}
\end{thm}
\begin{proof}
Let $x_0$ be any $k$-point of $SQ_{g,K}$. Then 
for a suitable CDVR $R$, 
there exists a morphism $j:\Spec R\to SQ_{g,K}$ 
such that 
\begin{enumerate}
\item[(i)] $j(0)=x_0\in SQ_{g,K}$, and
\item[(ii)] 
$j(\Spec k(\eta))\subset A_{g,K}\subset \hilb^{\chi(n)}$. 
\end{enumerate}

In other words, there exists a projective $R$-flat subscheme 
$(Z,\cL)$ of $(\bP(V_H),\bH)_R$ such that 
\begin{enumerate}
\item[(i${}^*$)] $x_0=(Z_0,\cL_0):=(Z,\cL)\otimes k(0)\in SQ_{g,K}$, 
\item[(ii${}^*$)] $(Z_{\eta},\cL_{\eta})$ 
is an $U_H$-invariant 
abelian variety  (to more precise, invariant under the action of $U_H\cG_H$ on $(\bP(V_H),\bH)$) such that
$\ker\lambda(\cL_{\eta})\simeq K:=H\oplus H^{\vee}$ and 
the actions of $\cG_H$ on $Z_{\eta}$ are translations of $Z_{\eta}$.
\end{enumerate}where 
$\eta$ is the generic point of $S$ and 
 $k(\eta)$ is the fraction field of $R$.\par

In this case, $(Z,\cL)$ 
is the pull back of $(Z_{\univ},L_{\univ})$ by $j$. 
Conversely, $j:\Spec R\to SQ_{g,K}$ is induced from the subscheme $(Z,\cL)$ of 
$(\bP(V_H),\bH)_R$ by the universality of $(Z_{\univ},L_{\univ})$.\par

Let $i:(Z,\cL)\to (\bP(V_H),\bH)_R$ be the natural inclusion, and 
$\cV_Z:=i^*\Gamma(\bP(V_H),\bH)=i^*V_H\otimes R$. 
Clearly $\cV_Z$ is very ample on $Z$.  
Since $j(\Spec k(\eta))\subset A_{g,K}$, the $\cG_H$-action 
on $(Z,\cL)$ induces a rigid level-$\cG_H$ structure on 
$(Z_{\eta},\cL_{\eta})$. That is,  
$(Z,\cV_Z,\cL,i,U_H)\otimes_Rk(\eta)$ 
is a rigid level-$\cG_H$ PSQAS over $k(\eta)$. In other words, 
$Z_{\eta}=i(Z_{\eta})$ is also a
$U_H$-invariant subscheme of $\bP(V_H)$. 
\par 
Meanwhile, by Theorem~\ref{thm:refined stable reduction}, 
by a finite base change 
if necessary, there exists 
a rigid level-$\cG_H$ PSQAS $(Q,\cL_Q,\phi,\tau)$ 
over $R$
 such that 
$$(Q_{\eta},\cL_{Q,\eta},\phi_{\eta},\tau_{\eta})
\simeq (Z_{\eta},\cL_{\eta},i_{\eta},U_H).$$
 By definition, 
$\rho(\phi,\tau)=U_H.$  Hence $\phi(Q)$ is a
$U_H$-invariant subscheme of $\bP(V_H)_R$.
Since $Z_{\eta}=i(Z_{\eta})$ is also a
$U_H$-invariant subscheme of $\bP(V_H)_{k(\eta)}$,  
by Lemma~\ref{sublemma:uniqueness of rigid PSQAS}~(2) (over $k(\eta)$)  
$$Z_{\eta}=i(Z_{\eta})=\phi(Q_{\eta}).
$$  

Hence their closures in $\bP(V_H)_{R}$ are the same. 
It follows $Z=\phi(Q)$, 
hence $(Z_0,\cL_0)=(\phi(Q_0),\cL_0)$ 
as a subscheme of $\bP(V_H)$.  Since 
$\Gamma(Q,\cL)=\phi^*\Gamma(\bP(V_H)_R,\bH_R)$ is very ample 
by Lemma~\ref{lemma:cohomology}
if $e_{\min}(H)\geq 3$, we have $Z_0=\phi(Q_0)\simeq Q_0$. It follows that 
$$x_0=(Z_0,\cL_0,i_0,U_H)\simeq (Q_0,\cL_0,\phi_0,\tau_0),$$ 
which is a rigid level-$\cG_H$ PSQAS. 
\end{proof} 

\begin{cor}Let $|H|=N$. Under the same assumption 
as in Theorem~\ref{thm:SQgK}, 
for any algebraically 
closed field $k$ with $k\ni 1/N$, we have 
\begin{align*}
A_{g,K}(k)&=\left\{(Q_0,i,U_H);
\begin{matrix}
 Q_0:\text{a level-$\cG_H$ abelian variety}\\
i:Q_0\subset\bP(V_H)\ \text{the inclusion}
\end{matrix}
\right\}.
\end{align*}
\end{cor}


\section{Moduli for PSQASes}
\label{sec:moduli psqas}

Let $\cO=\cO_N$. In this section we prove 
\begin{enumerate}
\item[(i)] $A_{g,K}$ is the fine moduli scheme for the functor of 
$T$-smooth PSQASes over $\cO$-schemes.
\item[(ii)] $SQ_{g,K}$ is the fine moduli scheme for the functor of 
$T$-flat PSQASes over reduced $\cO$-schemes.
\end{enumerate}

\subsection{$T$-smooth PSQASes}
\label{subsec:T-PSQASes}
Let $T$ be any $\cO$-scheme. 
In this subsection we define level-$\cG_H$ $T$-smooth PSQASes. 
Since any smooth PSQAS over a field is an abelian variety, 
any level-$\cG_H$ $T$-smooth PSQAS is 
a $T$-smooth scheme, any of whose geometric fiber is an abelian variety. 
It may have no global (zero) section over $T$.  
\begin{subdefn}
\label{subdefn:Tsmooth PSQAS}
A 6-tuple
$(Q,\cL,\cV,\phi,\cG,\tau)$ (or a triple $(Q,\phi,\tau)$ for brevity) is 
called {\em a $T$-smooth projectively stable quasi-abelian scheme} 
(abbr. {\em a $T$-smooth PSQAS})
of relative dimension $g$ 
with level-$\cG_H$ structure 
if the conditions (i)-(vi) are
true:
\begin{enumerate}
\item[(i)] 
$Q$ is a projective  
$T$-scheme 
with the projection $\pi:Q\to T$ surjective smooth,
\item[(ii)]$\cL$ is a relatively very ample line bundle of $Q$,
\item[(iii)] $\cG$ is a $T$-flat 
group scheme, $\tau:\cG\times (Q,\cL)\to (Q,\cL)$ 
is an action of $\cG$ as bundle automorphisms over $Q$, 
\item[(iv)]
$\phi:Q\to \bP(V_H)_T$ is 
a $\cG$-equivariant closed $T$-immersion of $Q$,
\item[(v)] there exists $M\in\Pic(T)$ with  
trivial $\cG$-action such that 
$\cL\simeq\phi^*\bH\otimes \pi^*M$ as $\cG$-modules, and 
$\cV=V_H\otimes_{\cO}M$ is a locally free 
$\cG$-invariant $O_T$-submodule \footnote{ 
$\cV=\pi_*\cL$ for $T$-smooth PSQASes. } of $\pi_*\cL$ of rank $|H|$ 
via the natural homomorphism, 
(see Remark~\ref{subrem:cV sub of piL}) 
\item[(vi)] for any geometric point $t$ of $T$, the fiber at $t$
$(Q_t,\cL_t,\cV_t,\phi_t,\cG_t,\tau_t)$ is {\it a level-$\cG_H$ smooth PSQAS} 
of dimension $g$ over $k(t)$.
\end{enumerate}

We call $(\phi,\tau)$ a level-$\cG_H$ structure on $Q$ 
if no confusion is possible. We also call $(Q,\phi,\tau)$ 
a level-$\cG_H$ 
$T$-smooth PSQAS. 
\end{subdefn}

\begin{subrem}
\label{subrem:aut0(Q) torsor}Let $Q$ be a $T$-smooth TSQAS. 
Then $\Aut^0_S(Q)$ is an abelian scheme over $S$
with zero section $\id_Q$, hence any $T$-smooth TSQAS $Q$ 
is an $\Aut^0_S(Q)$-torsor. See Theorem~\ref{thm:dAut_T(P)} 
and \cite{Nakamura14}.
\end{subrem}

\begin{subrem}
\label{subrem:cV sub of piL}
As in Definition~\ref{subdefn:level GH PSQAS} 
and Remark~\ref{subrem:phi Phi}, 
$\phi$ in (iv) is a $\cG$-morphism with respect to $\tau$ 
in the sense that 
$$\phi\tau(g)=S(\rho(\phi,\tau)(g))\phi,$$
under the notation $S(\rho(\phi,\tau)(g))$ 
in Subsec.~\ref{subsec:Linearization of O(1)}. \par
The natural homomorphism $\iota:\cV=V_H\otimes_{\cO}M\to\pi_*(\cL)$ is given as follows. Let $\pi_{\bP}:\bP(V_H)_T\to T$ be the natural projection. 
By the relation $\pi_{\bP}\phi=\pi$ and the projection formula, we see 
\begin{align*}
\pi_*(\cL)&=\pi_*(\phi^*(\bH\otimes \pi_{\bP}^*M))=
\pi_*(\phi^*(\bH)\otimes \pi^*M)=(\pi_{\bP})_*\phi_*\phi^*\bH\otimes M,
\end{align*}while $V_H\otimes M=(\pi_{\bP})_*(\bH)\otimes M$. 
Hence $\iota$ is induced from the natural homomorphism 
$\bH\to \phi_*\phi^*\bH$. In what follows we omit $\iota$.
\end{subrem}
\begin{subdefn}\label{subdefn:rigid level GH PSQAS}
Let $(Q,\phi,\tau)$  be a level-$\cG_H$ $T$-smooth PSQAS. 
Then $(\phi,\tau)$ 
is called {\em a rigid level-$\cG_H$ structure} 
if $\rho(\phi,\tau)=U_H$, where $\rho(\phi,\tau)$ is defined by
\begin{equation}\label{eq:rho(phi,tau)}
\rho(\phi,\tau)(g)(v):=(\phi^*)^{-1}\rho_{\tau,L}(g)\phi^*(v)
\end{equation}for $v\in\cV=\phi^*V_H\otimes_{\cO} M$. 
\end{subdefn}
\begin{subdefn}
\label{subdefn:morp of T-PSQAS}
Let $\sigma_i:=(Q_i,\cV_i,\cL_i,\phi_i,\cG_i,\tau_i)$ be 
a level-$\cG_H$ $T$-smooth PSQAS and 
$\pi_i:Q_i\to T$ the projection. 
Then $f:\sigma_1\to \sigma_2$ is called 
{\em a morphism of level-$\cG_H$ $T$-smooth PSQASes}
if there exists  $M\in\Pic(T)$,  a $T$-morphism 
$f:Q_1\to Q_2$ and a group scheme $T$-morphism $h:\cG_1\to \cG_2$
such that 
\begin{enumerate}\item[(i)]
$\phi_1=\phi_2\circ f$,
\item[(ii)] the following diagram is commutative:
\begin{equation*}
\CD
\cG_1\times (Q_1,\cL_1) @>{\tau_1}>> 
(Q_1,\cL_1) \\
@VV{h\times f}V @VV{f}V \\
\cG_2\times (Q_2,\cL_2\otimes_{O_T}\pi_2^*(M)) @>>{\tau_2}> 
(Q_2,\cL_2\otimes_{O_T}\pi_2^*(M)).
\endCD
\end{equation*}
\end{enumerate}

The morphism $f:\sigma_1\to\sigma_2$ is 
an isomorphism if and only if $f:Q_1\to Q_2$ is an isomorphism as schemes.   
\end{subdefn}
\begin{subrem}
\label{subrem:Remark to subdefn:morp of T-PSQAS}
From Definition~\ref{subdefn:morp of T-PSQAS}, 
we infer that there exists some $M\in\Pic(T)$ such that 
\begin{enumerate} 
\item[(i)] $\cL_1\simeq f^*(\cL_2)\otimes \pi_1^*(M)$ 
and $\cV_1=\cV_2\otimes M$, 
\item[(ii)] $(f,F(f)):(Q_1,\cL_1)\to (Q_2,\cL_2\otimes \pi_2^*(M))$ 
is a $\cG_1$-morphism of bundles: that is,
\begin{equation*}
(f,F(f))\circ\tau_1(g)=\tau_2(g)\circ (f,F(f)),\quad g\in \cG_1,
\end{equation*}
\item[(iii)] $\rho(\phi_1,\tau_1)=\rho(\phi_2,\tau_2)$. 
See \cite[Lemma~5.5]{Nakamura10}.
\end{enumerate}

In particular, for any $M\in\Pic(T)$ with trivial $\cG$-action, 
$$(Q,\cV,\cL,\phi,\cG,\tau)\simeq 
(Q,\cV\otimes M,\cL\otimes \pi^*M,\phi,\cG,\tau).$$
\end{subrem}

\begin{subrem}\label{subrem:isom by translation}
Since any $a\in Q(T)$ (a global section of $Q$) 
acts on $Q$ by translation, we have 
$$(Q,\cV,\cL,\phi,\cG,\tau)\simeq 
(Q,T_a^*\cV,T_a^*\cL,T_a^*\phi,\cG,T_a^*\tau),$$
where $T_a^*\tau=\{T_a^*\phi_g\}$ for $\tau=\{\phi_g\}$ 
as $\cG_H$-linearization. \theoremtextend
\end{subrem}

\begin{sublemma}
\label{sublemma:uniqueness of rigid T PSQAS} Assume $e_{\min}(H)\geq 3$. 
For a level-$\cG_H$ $T$-smooth (resp. $T$-flat) PSQAS  $(Z,\phi,\tau)$, 
there exists a unique rigid level-$\cG_H$  $T$-smooth (resp. $T$-flat) PSQAS 
$(Z,\psi,\tau)$ isomorphic to $(Z,\phi,\tau)$. 
\end{sublemma}
\begin{proof}One can prove this in parallel to 
Lemma~\ref{sublemma:uniqueness of rigid PSQAS}.  \par
By Definition~\ref{subdefn:Tsmooth PSQAS}, 
we have a 6-tuple $(Z,L,\cV,\phi,\cG,\tau)$. 
Let $\cV=V_H\otimes M$ for some $M\in\Pic(T)$.
We choose an affine covering $U_i$ of $T$ such that 
$M\otimes O_{U_i}$ is trivial. Let $Z_i:=Z_{T}\times {U_i}$. 
Then $\phi_i:=\phi_{|Z_i}:(Z_i,L_{Z_i})\to \bP(V_H)$ 
is a closed $\cG_{U_i}$-immersion and $\rho_{\rho_i,\tau}$ is 
equivalent to $U_H$. Hence there exists 
$A_i\in \GL(V_H\otimes O_{U_i})$ such that 
$U_H=A_i^{-1}\rho_{\rho_i,\tau}A_i$
by Lemma~\ref{lemma:irred repres UH}. 
We define a closed $\cG_{U_i}$-immersion 
$$\psi_i:(Z_i,L_{Z_i})\to (\bP(V_H)_{U_i},\bH_{U_i})$$
by $\psi_i^*=\phi_i^*A_i$. Hence we have 
$\rho(\psi_i,\tau)=U_H.$
Over $U_i\cap U_j$ we have two $\cG_{U_i\cap U_j}$-isomorphisms
$$\psi^*_k:V_H\otimes O_{U_i\cap U_j}
\simeq\cV\otimes O_{U_i\cap U_j},\quad (k=i,j).
$$ 
By Lemma~\ref{lemma:lemma of Schur}, 
there exists a unit $f_{ij}\in O_{U_i\cap U_j}^{\times}$ such that 
$\psi^*_i=f_{ij}\psi_j^*$. Hence 
$\psi_i=\psi_j$ over $U_i\cap U_j$ as a morphism to $\bP(V_H)_{U_i\cap U_j}$. 
Thus we have a $T$-smooth (resp. $T$-flat) 
PSQAS $(Z,\psi,\tau)$ such that $\rho(\psi,\tau)=U_H$. \par
The same argument proves the Lemma for a $T$-flat PSQAS. 
See Subsec.~\ref{subsec:T-flat PSQASes} 
for $T$-flat PSQASes. 
This completes the proof. 
\end{proof}

\begin{subdefn}
\label{subdefn:functor AgK}
We define 
a contravariant functor ${\cA}_{g,K}$ from the category 
of $\cO$-schemes to the category of sets by
\begin{align*}
{\cA}_{g,K}(T)=\ &\text{the set of all 
level-$\cG_H$ $T$-smooth PSQASes $(Q,\phi,\tau)$}\\
&\text{of relative dimension $g$ modulo $T$-isomorphism}\\
=\ &\text{the set of all rigid 
level-$\cG_H$ $T$-smooth PSQASes}\\
&\text{of relative dimension $g$ modulo $T$-isomorphism}
\end{align*}
by Lemma~\ref{sublemma:uniqueness of rigid T PSQAS}. 
\end{subdefn}


\subsection{Pro-representability}
\label{subsec:pro-representability}
Let $k$ be an algebraically closed field, 
and $W=W(k)$ the Witt ring of $k$ 
where we set $W=k$ if the characteristic of $k$ is zero. 
Hence $W$ is a complete 
local noetherian ring 
with maximal ideal $m_W$ and with 
$W/m_W=k$.
Let $\cC=\cC_W$  be
the category of local Artinian $W$-algebra 
with an isomorphism $k=R/m_R$ making 
the following diagram commutative:
\begin{equation*}
\CD
W@>>> R \\
@VV{}V @VV{}V \\
k@>{\simeq}>> R/m_R .
\endCD
\end{equation*}
Let $\hat\cC_W$ be the category of all complete local noetherian 
$W$-algebras $R$ such that $R/m_R^n\in\cC_W$ for every $n$.  
The morphisms in $\hat\cC_W$ are local $W$-algebra homomorphisms. 
A functor $F:\cC_W\to (Sets)$ is called {\em pro-representable} if there exists an $A\in\hat\cC_W$ such that 
$$F(R)=\Hom_{W\text{-hom.}}(A,R).
$$

\subsection{Deformation theory of abelian schemes}
\label{subsec:deform of AV}
We briefly review \cite{Oort72}. Let $k$ be an algebraically closed field. 
Let $\cC=\cC_W$.  We caution that 
$R\in\cC$ is not always a $k$-algebra. \par
Let $A$ be an abelian variety over $k$, $L_0$ an ample line bundle on $A$,
and $\lambda(L_0):A\to A^{\vee}:=\Pic^0_A$ the polarization morphism. 
\par
By Grothendieck and Mumford \cite[Theorems~2.3.3, 2.4.1]{Oort72} 
the quasi-polarized moduli functor $P$ 
of $(A,\lambda(L_0))$
is formally smooth 
if $\lambda(L_0):A\to A^{\vee}$
is separable. We will explain this. \par 
The deformation functor $M:=M(A)$ of $A$ 
is defined  over $\cC$ by
\begin{align*}
M(R)&=\left\{
(X,\phi_0); 
\begin{tabular}{l}
$X$ is a proper $R$-scheme\\
$\phi_0:X\otimes_R k\simeq A$
\end{tabular}
\right\}/\text{$R$-isom.}
\end{align*}

By Grothendieck 
\cite[Theorem~2.2.1]{Oort72}, $M$ is pro-represented by 
$$W(k)[[t_{i,j}; 1\leq i,j\leq g]].$$

The quasi-polarized moduli 
functor $P:=P(A,\lambda_0)$
of $(A,\lambda(L_0))$ over $\cC$
is defined as follows \cite[pp.~240-242]{Oort72} : 
\begin{align*}
P(R)&=\left\{
(X,\lambda,\phi_0); 
\begin{tabular}{l}
$(X,\lambda)$ is an abelian $R$-scheme\\
$\lambda:X\to X^{\vee}$ is a homomorphism\\
such that $\lambda=\lambda(\cL)$ for some $\cL\in\Pic(X)$\\
$\phi_0:(X,\lambda)\otimes_R k\simeq (A,\lambda_0)$
\end{tabular}
\right\}/\text{$R$-isom.}
\end{align*}
where $\lambda_0:=\lambda(L_0)$ and $X^{\vee}:=\Pic^0_{X/R}$. \par
Thus any  
$(Y,\lambda,\phi_0)\in P(R)$ always has a line bundle $L$ 
such that $\lambda=\lambda(L)$. This fact is used in Subsec.~\ref
{subsec:deform of sep pol AV}.\par

By \cite[Theorem~2.3.3]{Oort72}, 
$P(A,\lambda_0)$ is a pro-representable subfunctor of $M(A)$, that is,
the functor $P(A,\lambda_0)$ is pro-represented by 
$$\cO_W:=W[[t_{i,j}; 1\leq i,j\leq g]]/{\frak a}$$
for some ideal ${\frak a}$ where  
${\frak a}$ is generated by $\frac{1}{2}g(g-1)$ elements.

\subsection{Deformations in the separably polarized case}
\label{subsec:deform of sep pol AV}
We call $\lambda(L_0)$ (or $L_0$) a separable polarization if 
$\lambda(L_0):A\to A^{\vee}$
is a separable morphism. For instance, 
$\lambda(L_0)$ is separable if $k\ni 1/N$ where 
$N=\sqrt{|\ker\lambda(L_0)|}$. \par
Suppose that the polarization $\lambda_0$ is separable.  
The ideal ${\frak a}$ is generated by $t_{ij}-t_{ji}$ for any pair $i\neq j$ 
\cite[Remark, p.~246]{Oort72}:
$${\frak a}=(t_{ij}-t_{ji}; 1\leq i< j\leq g)
$$

 Hence $P(A,\lambda_0)$ is formally smooth 
of dimension $\frac{1}{2}g(g+1)$
over $W$. In this case $(A,\lambda_0)$ can be lifted as a 
formal abelian scheme $(\cX_{\formal},\lambda(\cL_{\formal}))$ 
over $\cO_W$, that is, there 
exists a system $(X_n,\lambda_n)$ of polarized abelian schemes over 
$\cO_{W,n}:=\cO_W/{\frak m}^{n+1}$ such that 
$$(X_{n+1},\lambda_{n+1})\otimes\cO_n\simeq (X_n,\lambda_n),
$$
where $\frak m$ is the maximal ideal of $\cO_W$. Then by 
\cite[III, {\bf 11}, 5.4.5]{EGA},
the formal scheme $\cX$ is algebraizable, that is, 
there exists a polarized abelian scheme $(X,\cL)$ over $\Spec\cO_W$ such that 
$$(X,\lambda(\cL))\otimes \cO_{W,n}\simeq (X_n,\lambda_n).$$

Let $K_{\suniv}=\ker(\lambda(\cL))$, 
$\cG_{\suniv}:=\cG(X,\cL):=\cL^{\times}_{K_{\suniv}}$ and  
$\cV_{\suniv}:=\Gamma(X,\cL).$ 
By \cite[pp.~115-117, pp.204-211]{Mumford12}, 
$\cL$ is $\cG_{\suniv}$-linearizable. In other words, 
$\cG_{\suniv}$ acts on $(X,\cL)$ by bundle automorphisms. 
Let $\tau_{\suniv}$ be the action of $\cG_{\suniv}$ on $(X,\cL)$. 
Then $\cV_{\suniv}$ is an $\cO_W$-free 
$\cG_{\suniv}$-module of rank $N$ via $\rho_{\tau_{\suniv},\cL}$.  \par
By the assumption $k\ni 1/N$, $\lambda(\cL):X\to X^{\vee}$ 
is separable, and
$K_{\suniv}$ is a constant finite symplectic Abelian group of order $N^2$ 
isomorphic to $H\oplus H^{\vee}$
because $K_{\suniv}\otimes_{\cO_W}k$ is so. \par
If $e_{\min}(K_{\suniv})\geq 3$, then $\cL$ is very ample 
because $\cL_0=L_0$ is very ample by Theorem~\ref{thm:embedding of (Q0,L0)} 
(Lefschetz's theorem in this case). Let $\phi_{\suniv}:X\to \bP(\cV_{\suniv})\simeq \bP(V_H)_{\cO_W}$ 
be the embedding 
of  $X$ into $\bP(\cV_{\suniv})$ such that 
$\rho(\phi_{\suniv},\tau_{\suniv})=U_H$. Thus we have a level-$\cG_H$ 
$\cO_W$-smooth PSQAS
$$(X,\cL,\cV_{\suniv},\phi_{\suniv},\cG_{\suniv},\tau_{\suniv}).
$$


 \begin{thm}\label{thm:fine moduli AgK} 
Let $K=H\oplus H^{\vee}$ and $N:=|H|$.  
If $e_{\min}(H)\geq 3$, then
the functor ${\cA}_{g,K}$ 
of level-$\cG_H$ smooth PSQASes over $\cO$-schemes 
is represented by the quasi-projective 
$\cO$-formally smooth scheme 
$A_{g,K}$.
\end{thm}
\begin{proof}By Lemma~\ref{sublemma:uniqueness of rigid T PSQAS},
for a $T$-smooth PSQAS $(Q,\phi,\tau)$ there exists 
a unique rigid level-$\cG_H$ $T$-smooth PSQAS 
$(Q,\psi,\tau)$ such that 
$(Q,\psi,\tau)$ is $T$-isomorphic to $(Q,\phi,\tau)$. 
Since $\cL$ is very ample by the assumption $e_{\min}(H)\geq 3$, 
$(Q,\psi,\tau)$ is embedded $\cG$-equivariantly into $(\bP(V_H),\bH)$, 
whose image is contained in $A_{g,K}$, because $\rho(\psi,\tau)=U_H$. 
This implies that there exists a unique morphism $f:T\to A_{g,K}$ such that 
$(Q,\psi,\tau)$ is the pull back by $f$ of the universal 
subscheme 
$$(Z_{g,K}\times_{H_{g,K}}A_{g,K},i,U_H).$$ 
It follows that ${\cal{A}}_{g,K}$ is represented by the quasi-projective 
$\bZ[\zeta_N,1/N]$-scheme 
$A_{g,K}$. \par
It remains to prove $A_{g,K}$ is 
formally smooth over $\bZ[\zeta_N,1/N]$. 
Let $k$ be any algebraically closed field with $k\ni 1/N$, and 
we choose any level-$\cG_H$ abelian variety over $k$ 
$$\sigma:=(A,L_0,\Gamma(A,L_0),\phi_0,\cG(A,L_0),\tau_0)\in\cA_{g,K}(k).$$ 
By Subsec.~\ref{subsec:deform of sep pol AV}, the quasi-polarized 
moduli functor $P(A,\lambda(L_0))$ 
is formally smooth 
because $\lambda(L_0):A\to A^{\vee}$
is separable by $k\ni 1/N$.\par 
We define a functor $F$ over $\cC$ by
\begin{align*}
F(R)&=\left\{
\xi:=(Z,L,\cV,\phi,\cG,\tau)\in \cA_{g,K}(R);
\xi\otimes_R k\simeq \sigma
\right\}
\end{align*}
where we do not fix the isomorphism $\xi\otimes_R k\simeq \sigma$ in contrast with $P(A,\lambda(L_0))$. 
Subsec.~\ref{subsec:deform of sep pol AV} shows that 
the map $h:P(A,\lambda(L_0))\to F$ sending 
 $(Z,L)=(X,\cL)\otimes_{\cO_W} R$ to 
$$(X,\cL,\cV_{\suniv},\phi_{\suniv},\cG_{\suniv},\tau_{\suniv})\otimes_{\cO_W} R
$$
is surjective because $\cV_{\suniv}$, $\phi_{\suniv}$, $\cG_{\suniv}$ and 
$\tau_{\suniv}$ are uniquely determined, . It follows from 
Lemma~\ref{sublemma:autom trivial} that $h$ is injective.
Hence $F=P(A,\lambda(L_0))$. Hence $A_{g,K}$ is 
formally smooth at $\sigma$.
\end{proof}

\begin{cor}$SQ_{g,K}$ is reduced.
\end{cor}
\begin{proof}Since $A_{g,K}$ is $\cO$-formally smooth, 
it is reduced. Since $SQ_{g,K}$ is the intersection 
of all closed $\cO$-subschemes 
containing $A_{g,K}$, it is the intersection 
of all closed reduced $\cO$-subschemes 
containing $A_{g,K}$ because $A_{g,K}$ is reduced. 
Hence $SQ_{g,K}$ is reduced. 
\end{proof}
\subsection{$T$-flat PSQASes}
\label{subsec:T-flat PSQASes}

\begin{subdefn}
\label{subdefn:Tflat PSQAS}
Let $T$ be any {\em reduced} $\cO$-scheme. 
A 5-tuple
$(Q,\cL,\cV,\phi,\cG,\tau)$ (or a triple $(Q,\phi,\tau)$ for brevity) is 
called {\em a projectively stable quasi-abelian $T$-flat scheme} 
(or just {\em a $T$-flat PSQAS})
of relative dimension $g$ 
with level-$\cG_H$ structure 
if the conditions (ii)-(v) in Definition~\ref{subdefn:Tsmooth PSQAS} 
and (i${}^*$), (vi${}^*$) are
true:
\begin{enumerate}
\item[(i${}^*$)] 
$Q$ is a projective  
$T$-scheme 
with the projection $\pi:Q\to T$ surjective flat,
\item[(vi${}^*$)] for any geometric point $t$ of $T$, the fiber at $t$
$(Q_t,\cL_t,\phi_t,\tau_t)$ is {\em a PSQAS} 
of dimension $g$ over $k(t)$ with level-$\cG_H$ structure.
\end{enumerate}

We also call $(Q,\phi,\tau)$ a level-$\cG_H$ $T$-PSQAS. 
\end{subdefn}
\begin{subdefn}
Let $(Q,\phi,\tau)$  be a level-$\cG_H$ $T$-flat PSQAS. 
Then $(\phi,\tau)$ 
is called {\em a rigid level-$\cG_H$ structure} if 
$\rho(\phi,\tau)=U_H$.
\end{subdefn}
\begin{subdefn}\label{subdefn:morp of T-flat PSQAS}
Let $(Q_i,\cV_i,\cL_i,\phi_i,\cG_i,\tau_i)$ be level-$\cG_H$ $T$-PSQASes and 
$\pi_i:Q_i\to T$ a flat morphism (structure morphism) 
with $T$ {\em reduced}.
Then $f:Q_1\to Q_2$ is called 
{\em a morphism of level-$\cG_H$ $T$-PSQASes} 
 if the conditions   
in Definition~\ref{subdefn:morp of T-PSQAS} are true.
\end{subdefn}

\begin{subdefn}
\label{subdefn:sch red}
The category $Sch_{\red}$ of reduced schemes 
is a subcategory of the category $Sch$ of schemes with 
\begin{align*}
\Obj(Sch_{\red})&= \text{reduced schemes}, \\
\Mor(Sch_{\red})&=\text{morphisms in the category of schemes.}
\end{align*}
\end{subdefn}

\begin{subdefn}
\label{subdefn:functor SQgK}
We define 
a contravariant functor ${\cS}{\cQ}_{g,K}$ from the category $Sch_{\red}$
of {\em reduced} $\cO$-schemes to the category of sets by
\begin{align*}
{\cS}{\cQ}_{g,K}(T)=\ &\text{the set of all 
level-$\cG_H$ $T$-flat PSQASes $(Q,\phi,\tau)$}\\
&\text{of relative dimension $g$ modulo $T$-isomorphism}\\
=\ &\text{the set of all rigid 
level-$\cG_H$ $T$-flat PSQASes}\\
&\text{of relative dimension $g$ modulo $T$-isomorphism}
\end{align*}
by Lemma~\ref{sublemma:uniqueness of rigid T PSQAS}.
\end{subdefn}

\begin{thm}\label{thm:fine moduli SQgK}
Suppose $e_{\min}(K)\geq 3$. Let $N:=\sqrt{|K|}$.
The functor ${\cal{SQ}}_{g,K}$ 
of level-$G(K)$ PSQASes $(Q,\phi,\tau)$ over  
{\em reduced} schemes
is represented by 
the projective reduced 
$\cO_N$-scheme 
$SQ_{g,K}$.  
\end{thm}
\begin{proof}This is proved 
in parallel to Theorem~\ref{thm:fine moduli AgK}. 
Properness of $SQ_{g,K}$ follows from 
Theorem~\ref{thm:refined stable reduction}.  
See \cite[Theorem~10.4]{Nakamura99} for a more precise statement.
Since $SQ_{g,K}$ is a proper subscheme of the projective scheme 
$\hilb^{\chi(n)}$ in Subsec.~\ref{subsec:Hilb}, it is projective.
\end{proof}

\section{The functor of TSQASes}
\label{sec:level GH structure of TSQAS}

\subsection{TSQASes over $k$}
We introduced 
two kinds of nice classes of degenerate abelian 
schemes, PSQASes and TSQASes 
in Theorem~\ref{thm:refined stable reduction}. 
\par
It is TSQASes that we discuss in this section.
They are nonsingular abelian varieties, 
or reduced even if singular, 
and therefore easier to handle than PSQASes.
However the very-ampleness criterion 
(Theorem~\ref{thm:refined stable reduction}~(4))
fails for $(P,\cL_P)$, and because of this defect, 
 we cannot expect the existence of the fine moduli scheme for TSQASes.

Let $H$ be any finite Abelian group, $N=|H|$,   
$k$ an algebraically closed field with $k\ni 1/N$, 
$K=H\oplus H^{\vee}$ and $\cO=\cO_N$.\par 
\begin{subrem}\label{subrem:chara GH action of P} 
Let $(Z,L)$ be any TSQAS over $k$. Hence there exist 
an Abelian group $H$ and
a flat family $(P,\cL)$ over a CDVR $R$ with $k=R/m$ given 
in Theorem~\ref{thm:refined stable reduction} 
such that $(Z,L)\simeq (P_0,\cL_0)$, $P_0$ is reduced,   
 and the generic fiber $(P_{\eta},\cL_{\eta})$ 
is an abelian variety with 
$\ker\lambda(\cL_{\eta}))\simeq K_H=H\oplus H^{\vee}$. 
Hence we have an action of $\cG_H$ on $(Z,L)$. See 
Remark~\ref{subrem:chara GH action}.  
We denote by $\cG(Z,L)$ the $\cG_H$-action on $(Z,L)$.  
This is determined by $(Z,L)$ 
uniquely up to an automorphism of $\cG_H$. 
In the totally degenerate case, 
the action of $\cG(Z,L)$ is explicitly written as $S_x$ and 
$T_a$ $(x\in X/Y, a\in X\times_{\bZ}\bG_m)$. 
See Definition~\ref{subdefn:Sz Tbeta}. 
\end{subrem}

\begin{subdefn}
\label{subdefn:chara GH action}
Let $(Z,L)$ be a TSQAS over $k$. 
We call  
$\tau:\cG_H\times (Z,L)\to (Z,L)$  
{\em a characteristic $\cG_H$-action,} 
or simply {\em characteristic}, if this action of $\cG_H$ 
induces the natural isomorphism in Remark~\ref{subrem:chara GH action of P}
$$\cG_H\overset{\cong}{\to} \cG(Z,L)\subset \Aut(L/Z).$$ 
\end{subdefn}

\begin{subdefn}Let $(Z,L)$ be a TSQAS over $k$. We define
$(Z,L,\phi^*,\cG_H,\tau)$ (denoted often $(Z,\phi^*,\tau)$ or 
$(Z,L,\phi^*,\tau)$) 
to be a level-$\cG_H$ TSQAS if 
if the conditions (i)-(iii) are 
true:
\begin{enumerate}
\item[(i)] 
$(Z,L)$ is a PSQAS $(P_0,\cL_0)$ over $k$ with $L$ ample,
\item[(ii)] $\tau:\cG_H\times (Z,L)\to (Z,L)$ is 
a characteristic $\cG_H$-action,
\item[(iii)] $\phi^*:V_H\otimes k\to H^0(Z,L)$ 
is a $\cG_H$-isomorphism.
\end{enumerate}
\end{subdefn}

\begin{subdefn}\label{subdefn:k-isom of TSQAS}
We define level-$\cG_H$ $k$-TSQASes 
$(Z_1,L_1,\phi_1^*,\tau_1)$ and $(Z_2,L_2,\phi_2^*,\tau_2)$ 
to be {\em isomorphic} if 
there exists a $\cG_H$-isomorphism 
$f:(Z_1,L_1)\to (Z_2,L_2)$
such that $f^*\phi_1^* =c \phi_2^*$ for some nonzero $c\in k$. \end{subdefn}

\subsection{$T$-smooth TSQASes}
\label{subsec:T-TSQASes}
Let $T$ be any $\cO$-scheme. 
In this subsection we define level-$\cG_H$ $T$-smooth TSQASes. 
The level-$\cG_H$ $T$-smooth TSQASes are essentially the same 
as level-$\cG_H$ $T$-smooth PSQASes in Subsec.~\ref{subsec:T-PSQASes}.
The only difference from Subsec.~\ref{subsec:T-PSQASes} is 
that we define them 
without any restriction on $e_{\min}(H)$. 
Since any smooth TSQAS over a field is an abelian variety, 
any level-$\cG_H$ $T$-smooth TSQAS is 
a level-$\cG_H$ abelian scheme over $T$ 
possibly with no zero section over $T$.

\begin{subdefn}
\label{subdefn:Tsmooth TSQAS}
A 5-tuple
$(P,\cL,\phi^*,\cG,\tau)$ (or a triple $(P,\phi^*,\tau)$ for brevity) is 
called {\em a $T$-smooth PSQAS}
of relative dimension $g$ 
with level-$\cG_H$ structure 
if the conditions (i)-(v) are
true:
\begin{enumerate}
\item[(i)] 
$P$ is a projective  
$T$-scheme 
with the projection $\pi:P\to T$ surjective smooth,
\item[(ii)]$\cL$ is a relatively ample line bundle of $P$,
\item[(iii)] $\cG$ is a $T$-flat 
group scheme, $\tau:\cG\times (P,\cL)\to (P,\cL)$ 
is an action of $\cG$ on $(P,\cL)$ as bundle automorphism, 
\item[(iv)]there exists a $\cG$-isomorphism 
$\phi^*:V_H\otimes_{\cO}M\overset{\cong}{\to}\pi_*\cL$ 
 for some $M\in\Pic(T)$ with  
trivial $\cG$-action, 
\item[(v)] for any geometric point $t$ of $T$, the fiber at $t$
$(P_t,\cL_t,\phi^*_t,\cG_t,\tau_t)$ 
is {\em a level-$\cG_H$ smooth TSQAS} 
of dimension $g$ over $k(t)$.
\end{enumerate}

We call $(\phi^*,\tau)$ a level-$\cG_H$ structure on $P$ 
if no confusion is possible. We also call $(P,\phi^*,\tau)$ 
a level-$\cG_H$ 
$T$-smooth TSQAS. 
\end{subdefn}
\begin{subdefn}\label{subdefn:rigid level GH TSQAS}
Let $(P,\phi^*,\tau)$  be a level-$\cG_H$ $T$-smooth TSQAS. 
Then $(\phi^*,\tau)$ 
is called {\em a rigid level-$\cG_H$ structure} 
if $\rho(\phi^*,\tau)=U_H$, 
where $\rho(\phi^*,\tau)$ is defined by
\begin{equation}\label{eq:rho(phi,tau)}
\rho(\phi^*,\tau)(g)(\theta):=(\phi^*)^{-1}\rho_{\tau,L}(g)(\theta)\phi^*
\end{equation}for $\theta\in\cV:=\phi^*V_H\otimes_{\cO} M$. 
If $\phi^*$ defines a morphism $\phi:Z\to\bP(V_H)_T$, then 
$\rho(\phi^*,\tau)=\rho(\phi,\tau)$ with the notation 
in Definition~\ref{subdefn:rigid level GH PSQAS}. 
\end{subdefn}
\begin{subdefn}
\label{subdefn:morp of T-TSQAS}
Let $(P_k,\cL_k,\phi^*_k,\cG_k,\tau_k)$ be 
a level-$\cG_H$ $T$-smooth TSQAS and 
$\pi_k:P_k\to T$ the projection (structure morphism). 
Then $f:P_1\to P_2$ is called 
{\em a morphism  of level-$\cG_H$ $T$-smooth TSQASes} 
if there exists  $M\in\Pic(T)$,  a $T$-morphism 
$f:P_1\to P_2$ and a group scheme $T$-morphism $h:\cG_1\to \cG_2$
such that  
\begin{enumerate}
\item[(i${}^{**}$)] $f^*\phi_2^*=c\phi_1^*$ 
for some unit $c\in H^0(O_T)^{\times}$, 
\item[(ii${}^{**}$)] the following diagram is commutative:
\begin{equation*}
\CD
\cG_1\times (P_1,\cL_1) @>{\tau_1}>> 
(P_1,\cL_1) \\
@VV{h\times f}V @VV{f}V \\
\cG_2\times (P_2,\cL_2\otimes_{O_T}\pi_2^*(M)) @>>{\tau_2}> 
(P_2,\cL_2\otimes_{O_T}\pi_2^*(M)).
\endCD
\end{equation*}\end{enumerate}

The same is true as in 
Remark~\ref{subrem:Remark to subdefn:morp of T-PSQAS} 
by replacing $\rho(\phi_k,\tau_k)$ by $\rho(\phi_k^*,\tau_k)$.
\end{subdefn}

\begin{sublemma}
\label{sublemma:uniqueness of rigid T TSQAS} 
For a level-$\cG_H$ $T$-smooth (resp. $T$-flat) TSQAS  $(Z,\phi^*,\tau)$, 
there exists a unique 
rigid level-$\cG_H$  $T$-smooth (resp. $T$-flat) TSQAS 
$(Z,\psi^*,\tau)$ such that 
\begin{enumerate}
\item $(Z,\psi^*,\tau)$ is isomorphic to $(Z,\phi^*,\tau)$,
\item $\psi^*$ is the $\cG_H$-isomorphism with $\rho(\psi^*,\tau)=U_H$, 
unique up to nonzero constant multiple.
\end{enumerate}
\end{sublemma}
\begin{proof}One can prove this in parallel to 
Lemma~\ref{sublemma:uniqueness of rigid PSQAS}. 
\end{proof}

\begin{subdefn}
\label{subdefn:functor AgK}
We define 
a contravariant functor ${\cA}_{g,K}$ 
(the functor of level-$\cG_H$ smooth TSQASes) 
from the category 
of $\cO$-schemes to the category of sets by
\begin{align*}
{\cA}^{\toric}_{g,K}(T)=\ &\text{the set of all 
level-$\cG_H$ $T$-smooth TSQASes $(Q,\phi^*,\tau)$}\\
&\text{of relative dimension $g$ modulo $T$-isomorphism}\\
=\ &\text{the set of all rigid 
level-$\cG_H$ $T$-smooth TSQASes}\\
&\text{of relative dimension $g$ modulo $T$-isomorphism}
\end{align*}
by Lemma~\ref{sublemma:uniqueness of rigid T TSQAS}. 
\end{subdefn}

\subsection{$T$-flat TSQASes}
\label{subsec:T-flat TSQASes}

\begin{subdefn}
\label{subdefn:Tflat PSQAS}
Let $T$ be any {\em reduced} $\cO$-scheme. 
A 5-tuple
$(P,\cL,\phi^*,\cG,\tau)$ (or a triple $(P,\phi^*,\tau)$ for brevity) is 
called {\em a $T$-flat TSQAS} 
of relative dimension $g$ 
with level-$\cG_H$ structure 
if the conditions (ii)-(iv) in Definition~\ref{subdefn:Tsmooth TSQAS} 
and (i${}^*$), (v${}^*$) are
true:
\begin{enumerate}
\item[(i${}^*$)] 
$P$ is a 
$T$-scheme 
with the projection $\pi:P\to T$ surjective flat,
\item[(v${}^*$)] for any geometric point $t$ of $T$, the fiber at $t$
$(P_t,\cL_t,\phi^*_t,\tau_t)$ is {\it a TSQAS} 
of dimension $g$ over $k(t)$ with level-$\cG_H$ structure.
\end{enumerate}

We also call $(P,\phi^*,\tau)$ a level-$\cG_H$ $T$-TSQAS. 
\end{subdefn}
\begin{subdefn}
Let $(P,\phi^*,\tau)$  be a level-$\cG_H$ $T$-flat TSQAS. 
Then $(\phi^*,\tau)$ 
is called {\em a rigid level-$\cG_H$ structure} if  $\rho(\phi^*,\tau)=U_H$.
\end{subdefn}
\begin{subdefn}\label{subdefn:morp of T-flat PSQAS}
Let $(P_i,\cL_i,\phi^*_i,\cG_i,\tau_i)$ be level-$\cG_H$ 
$T$-TSQASes and 
$\pi_i:P_i\to T$ a flat morphism (structure morphism) 
with $T$ {\em reduced}.
Then $f:P_1\to P_2$ is called 
{\em an isomorphism 
of level-$\cG_H$ $T$-TSQASes} if the conditions  
in Definition~\ref{subdefn:morp of T-TSQAS} are true.
\end{subdefn}

\begin{subdefn}
\label{subdefn:sch red}
The category $Space_{\red}$ of reduced algebraic spaces 
is a subcategory of the category $Space$ of algebraic spaces  with 
\begin{align*}
\Obj(Space_{\red})&= \text{reduced algebraic spaces}, \\
\Mor(Space_{\red})&=\text{morphisms in the category of algebraic spaces.}
\end{align*}
\end{subdefn}

\begin{subdefn}
\label{subdefn:functor SQgKtoric}
We define 
a contravariant functor ${\cS}{\cQ}^{\toric}_{g,K}$ from the category 
$Space_{\red}$
of {\em reduced} algebraic $\cO$-spaces to the category of sets by
\begin{align*}
{\cS}{\cQ}^{\toric}_{g,K}(T)=\ &\text{the set of all 
level-$\cG_H$ $T$-flat TSQASes $(P,\phi^*,\tau)$}\\
&\text{of relative dimension $g$ modulo $T$-isomorphism}\\
=\ &\text{the set of all rigid 
level-$\cG_H$ $T$-flat TSQASes}\\
&\text{of relative dimension $g$ modulo $T$-isomorphism}
\end{align*}
by Lemma~\ref{sublemma:uniqueness of rigid T TSQAS}.
\end{subdefn}

\section{The moduli spaces $A_{g,K}^{\toric}$ and $SQ_{g,K}^{\toric}$}
\label{sec:SQtoric}
Let $H$ be a finite Abelian group, $K=K_H:=H\oplus H^{\vee}$ 
and $N=|H|$, and let $\cO=\cO_N$.
In this section we recall from \cite[\S~9]{Nakamura10} 
how to construct the algebraic space 
$SQ_{g,K}^{\toric}$ parameterizing level-$\cG_H$ TSQASes. \par
The construction in 
Subsec.~\ref{subsec:HilbP(X/T)}--\ref{subsec:The scheme Udagger and U3} 
is carried out 
{\em without any change regardless of the value of} $e_{\min}(H)$. 
We do not assume $e_{\min}(H)\geq 3$ unless otherwise mentioned. 
\par
We summarize this section in Summary ~\ref{summary:summary} 
at the end.
\subsection{Preliminaries}
Let $k$ be any algebraically closed field  with $k\ni 1/N$.
In this subsection we list some basic properties of 
a level-$\cG_H$ TSQAS $(P_0,\cL_0)$ over $k$ that we use in what follows.

\begin{sublemma}
\label{sublemma:P,L}
Let $k$ be any algebraically closed field  with $k\ni 1/N$.
Let  $(P_0,\cL_0,\phi_0^*,\cG(P_0,\cL_0),\tau_0)$ 
be a level-$\cG_H$ TSQAS over $k$, and 
therefore a closed fiber of the TSQAS $(P,\cL)$ over a CDVR $R$ 
with the generic fiber $P_{\eta}$ an abelian variety. 
Then 
\begin{enumerate}
\item $P_0$ is nonsingular if and only if it is an abelian variety,
\item $P_0$ is reduced,
\item $\cL_0$ is ample, and $n\cL_0$ is very ample for $n\geq 2g+1$, 
\item $H^q(P_0,n\cL_0)=0$ for any $q>0, n>0$,
\item $\chi(P_0, n\cL_0)=n^g|H|$ for any $n>0$, 
\item the action $\cG(P_0,\cL_0)$ of $\cG_H$ on $(P_0,\cL_0)$ 
is characteristic, that is, 
it is induced from $\cG(P_{\eta},\cL_{\eta})$, where 
any of the latter induces a translation of an abelian variety $P_{\eta}$.
\end{enumerate}
\end{sublemma}
\begin{proof}(1) follows from 
Theorem~\ref{thm:refined stable reduction}.  
For (2)--(5), see \cite{AN99} or 
\cite[Theorem~2.11, p.~79]{Nakamura10}. 
(6) is proved (and defined) 
in a manner similar to Remark~\ref{subrem:chara GH action} and
Definition~\ref{subdefn:chara GH action}. 
\end{proof}

\begin{sublemma}
\label{sublemma:weight 1 mod N module}
Let $n$ be any positive integer, and $d=Nn+1$. 
 We define 
$U_{d,H}$ on the $\cO$-module $V_H$ in 
Definition~\ref{defn:Heisenberg group} by
\begin{equation}\label{eq:action UdH}
U_{d,H}(a,z,\alpha)v(\beta)=a^d\beta(z)^dv(\alpha+\beta).
\end{equation} 
We denote $V_H$ by $V_{d,H}$ if $\cG_H$ acts on $V_H$ via $U_{d,H}$. Then 
\begin{enumerate}
\item $V_{d,H}$ is an irreducible $\cG_H$-module of weight $d$, 
\item let $W$ be any $\cO$-free $\cG_H$-module of finite rank. 
If $\cG_H$ acts on $W$ with weight $d$: that is, 
the center $\bG_m$ of $\cG_H$ acts on $W$ by $a^{d}\id_W$, then 
$W$ is equivalent to $W_0\otimes_{\cO}V_{d,H}$ as $\cG_H$-module, where 
$W_0$ is an $\cO$-module with trivial $\cG_H$-action.
\end{enumerate}
\end{sublemma}
\begin{proof}
We denote the action of $g\in\cG_H$ on $W$ by $U(g)$, 
and we write $U(g)=U(a,z,\alpha)$ for $g=(a,z,\alpha)\in\cG_H$. 
Let $W(\chi)=\{w\in W; U(h)w =\chi(h) w\ \text{for any $h\in H$}\}$. 
By \cite[p.~89]{Nakamura10}, we have
\begin{equation}\label{eq:decomp of W}
W=\bigoplus_{\chi\in H^{\vee}} W(\chi),\quad  
W(\chi)=U(1,0,\chi)W(0).
\end{equation} Therefore $W(0)\neq 0$ if $W\neq 0$. \par
For any $w\in W(0)$, we define
 $v(\chi,w)=U(1,0,\chi)w$ for $\chi\in H^{\vee}$.  \par
By imitating \cite[p.~89]{Nakamura10}, we infer 
\begin{align*}
U(1,z,0)\cdot v(\chi,w)&=U(\chi(z)(1,0,\chi))w=\chi(z)^dv(\chi,w),\\
U(1,0,\alpha)\cdot v(\chi,w)&=U(1,0,\chi+\alpha)\cdot w=v(\chi+\alpha,w),
\end{align*}whence
\begin{equation}\label{eq:action UdH on W}
\begin{aligned}
U(a,z,\alpha)\cdot v(\chi,w)&
=U(a(1,0,\alpha)(1,z,0)(1,0,\chi))w\\
&=U(a\chi(z)(1,0,\chi+\alpha)(1,z,0))w\\
&=a^d\chi(z)^dv(\chi+\alpha,w).
\end{aligned}
\end{equation}

We define a homomorphism $F:W(0)\otimes V_{d,H}\to W$ by
\begin{equation}\label{eq:homom F}
F(w\otimes v(\chi))=v(\chi,w)
\end{equation}where $w\in W(0)$ and $v(\chi)\in V_{d,H}$. 
Here 
$W(0)$ in the left hand side of (\ref{eq:homom F})  
is regarded as a trivial $\cG_H$-module, while 
$W(0)$ in the right hand side of (\ref{eq:homom F}) 
is an $\cO$-submodule of $W$.
Then by (\ref{eq:action UdH}) and (\ref{eq:action UdH on W}),
$F$ is a $\cG_H$-homomorphism:
$$F(w\otimes U_{d,H}(g)(v(\chi)))=U(g)v(\chi,w).
$$ 

In view of (\ref{eq:decomp of W}), $W$ is spanned by $v(\chi,w)$  
for $w\in W(0)$ and $\chi\in H^{\vee}$. Hence $F$ is surjective.
By (\ref{eq:decomp of W}), $W$ and $W(0)\otimes V_{d,H}$ are 
$\cO$-modules of the same rank. Hence $F$ is an isomorphism. 
\end{proof}


\subsection{$\Hilb^{P}(X/T)$}
\label{subsec:HilbP(X/T)}
Let $(X,L)$ be a polarized 
$\cO$-scheme with $L$ very ample and  $P(n)$ an 
arbitrary polynomial.
Let $\Hilb^P(X)$ be the Hilbert scheme 
parameterizing all closed subschemes $Z$ of $X$ 
with $\chi(Z,nL_Z)=P(n)$. 
As is well known $\Hilb^P(X)$ is a projective 
$\cO$-scheme. \par
Let $T$ be a projective scheme, $(X,L)$
a flat projective $T$-scheme 
with $L$ an ample line bundle of $X$, and 
$\pi:X\to T$ the projection. 
Then for an arbitrary polynomial $P(n)$, 
let $\Hilb^{P}(X/T)$ be the scheme 
parameterizing all closed subschemes $Z$ of $X$ 
with $\chi(Z,nL_Z)=P(n)$ such that $Z$ 
is contained in fibers of $\pi$. 	
Then $\Hilb^{P}(X/T)$ 
is  a closed $\cO$-subscheme of $\Hilb^P(X)\times_{\cO} T$. 
See \cite[Chap.~9]{ACG11}.

\subsection{The scheme $H_1\times H_2$}
\label{subsec:The scheme H1xH2}
Choose and fix a coprime pair of natural 
integers $d_1$ and $d_2$ such that
$d_1>d_2\geq 2g+1$ and $d_{\nu}\equiv 1\mod N$.	
This pair does exist because it is enough to
choose prime numbers $d_1$ and $d_2$ large enough 
such that $d_{\nu}\equiv 1\mod N$ and $d_1>d_2$. 
We choose integers $q_{\nu}$ such that $q_1d_1+q_2d_2=1$. 
\par
We consider a $\cG_H$-module
 $$W_{\nu}(K):=W_{\nu}\otimes V_{d_{\nu},H}\simeq 
V_{d_{\nu},H}^{\oplus N_{\nu}}$$ 
where $N_{\nu}=d_{\nu}^g$ and 
$W_{\nu}$ is a free $\cO$-module of rank $N_{\nu}$ 
with trivial $\cG_H$-action.
Let $\sigma_{\nu}$ be 
the natural action of $\cG_H$ on $W_{\nu}(K)$.
In what follows we always consider $\sigma_{\nu}$.
\par
Let $H_{\nu}$ $(\nu=1,2)$ be the Hilbert scheme parameterizing all
closed polarized subschemes $(Z_{\nu},L_{\nu})$ 
of $\bP(W_{\nu}(K))$ such that
\begin{enumerate}
\item[(a)] $Z_{\nu}$ is $\cG_H$-stable,
\item[(b)]
$\chi(Z_{\nu},nL_{\nu})=n^gd_{\nu}^g|H|$, 
where $L_{\nu}=\bH(W_{\nu}(K))\otimes O_{Z_{\nu}}$ 
is the hyperplane bundle of $Z_{\nu}$ .
\end{enumerate}

Since (a) and (b) are closed conditions, 
$H_{\nu}$ is a closed (hence projective) subscheme 
of $\hilb^{\chi_{\nu}}(\bP(W_{\nu}(K))$ 
where $\chi_{\nu}(n)=n^gd_{\nu}^g|H|$. 

Let $\cO=\cO_N$. Let $X_{\nu}$ be the universal subscheme of $\bP(W_{\nu}(K))$ 
over $H_{\nu}$. Let $X=X_1\times_{\cO} X_2$ and 
$H_3=H_1\times_{\cO} H_2$. 
Let $p_{\nu}:X_1\times_{\cO} X_2\to X_{\nu}$ 
be the $\nu$-th projection, $\pi:X\to H_3$ the natural projection. 
Hence $X$ is a subscheme 
of $\bP(W_1(K))\times_{\cO}\bP(W_2(K))\times_{\cO}H_3$, flat	
over $H_3=H_1\times_{\cO} H_2$. \par
We note that $\bH(W_{\nu}(K))$
has a $\cG_H$-linearization $\{\psi^{(\nu)}_g\}$, 
which we fix once for all.  Since $\cG_H$ transforms 
any closed $\cG_H$-stable subscheme $Z$ of $\bP(W_{\nu}(K))$
onto itself, it follows that 
$\cG_H$ acts on $H_{\nu}$ trivially.  Hence,  
$\cG_H$ transforms any fiber $X_u$ of $\pi:X\to H_3$ onto $X_u$ itself.
\subsection{The scheme $U_1$}
\label{subsec:The scheme U1}
The aim of this and the subsequent subsections 
is to construct a new compactification 
of the	moduli space of abelian varieties as the quotient 
of a certain $\cO$-subscheme of $\Hilb^{P}(X/H_3)$ by 
$\PGL(W_1)\times\PGL(W_2)$.\par
Let $B$ be the pullback to $X$ 
of a very ample line bundle on $H_3$.
Let $M_{\nu}=p_{\nu}^*(\bH(W_i(K)))\otimes O_{X}$ and 
\begin{equation}
\label{eq:M}
M=d_2M_1+d_1M_2+B.
\end{equation}
Then $M$ is a very ample line bundle on $X$. 
Since $M_{\nu}$ is $\cG_H$-linearized 
and $B$ is trivially $\cG_H$-linearized, 
$M$ is $\cG_H$-linearized. 
\par
Let $P(n)=(2nd_1d_2)^g|H|$. 
Let $\Hilb^{P}(X/H_3)$ be the Hilbert scheme 
parameterizing all closed subschemes $Z$ of $X$ 
contained in the fibers of $\pi:X\to H_3$
with $\chi(Z,nM_Z)=P(n)$, and $Z^{P}$ be 
the universal subscheme of $X$ over it. We denote 
$\Hilb^{P}(X/H_3)$ by $H^{P}$ for brevity.
Now using the double polarization trick of Viehweg, 
we define $U_1$ to be the subset of $H^{P}$
consisting of all subschemes $(Z,M_Z)$ of $(X,M)$ 
with the properties	
\begin{enumerate}
\item[(i)] $Z$ is $\cG_H$-stable,
\item[(ii)] $d_2L_1=d_1L_2$, where $L_i=M_i\otimes O_Z$. 
\end{enumerate}

By Lemma~\ref{lemma:exists (A,L) with G(A,L)=G(K)}, 
 $U_1$ is a nonempty closed $\cO$-subscheme of 
$H^P$. See \cite[Subsec.~9.3]{Nakamura10}.

\subsection{The scheme $U_2$}
\label{subsec:The scheme U2}

Let $U_2$ be the open subscheme of $U_1$ 
consisting of all subschemes $(Z,M_Z)$ of $(X,M)$ such that 
besides (i)-(ii) the following are satisfied:
\begin{enumerate}
\item[(iii)] ${p_{\nu}}_{|Z}$ is an isomorphism  $(\nu=1,2)$,
\item[(iv)] $Z$ is reduced with $h^0(Z,O_Z)=1$, 
\item[(v)] $d_{\nu}L$ is very ample on $Z$, 
where $L=(q_1M_1+q_2M_2)\otimes O_Z$,
\item[(vi)] $\chi(Z,nL)=n^g|H|$ for $n>0$, 
\item[(vii)] $H^q(Z,nL)=0$ 
for $q>0$ and $n>0$,
\item[(viii)] 	
$H^0(p_{\nu}^*) : W_{\nu}(K)\otimes k(u)\to\Gamma(Z,d_{\nu}L)$ 
is surjective (hence an isomorphism by (vi) and (vii)) for $\nu =1,2$.
\end{enumerate}

Let $(Z,M_Z)\in H^{P}$.  
By (ii) and (v), we have $L=q_1L_1+q_2L_2$ for 
$L_i=M_i\otimes O_Z$. Since $d_1q_1+d_2q_2=1$, we have  
$L_{\nu}=d_{\nu}L$ by (ii). (iii) 
is an open condition by \cite[Chap.~9, Lemma~7.5]{ACG11}. 
It is clear that (iv)-(viii) are 
open conditions. 
It follows that $U_2$ is a nonempty open $\cO$-subscheme of $U_1$. 
 See \cite[Subsec.~9.5]{Nakamura10}.

\subsection{The schemes $U^{\dagger}_{g,K}$ and $U_3$}
\label{subsec:The scheme Udagger and U3} 
See \cite[Subsec.~9.7]{Nakamura10}.
First we note that
if $(Z,L)\in U_2$, then $L=q_1L_1+q_2L_2$. 
On each $L_{\nu}$ we have a $\cG_H$-action on 
$(Z,L_{\nu})$ induced 
from the $\cG_H$-action ($=\cG_H$-linearization) on $Z^P$ 
induced from those $\cG_H$-actions on $\bP(W_{\nu}(K))$. 
By Remark~\ref{subrem:linearization of tensor}, we have a 
$\cG_H$-linearization on $(Z,L)$.
In what follows, we mean this $\cG_H$-action on $Z$ or $(Z,L)$ 
by the (characteristic) $\cG_H$-action on 
$(Z,L)$ when $(Z,L)\in U_2$. \par
The locus $U_{g,K}$ 
of abelian varieties (with the zero not necessarily chosen) 
is an open subscheme of $U_2$. 
In fact,  $U_{g,K}$ is 
the largest open $\cO$-subscheme among all the open 
$\cO$-subschemes $H'$ of $U_2$ such that 
\begin{enumerate}
\item[($\alpha$)]
the projection $\pi_{H'}:Z^P\times_{H^P} H' \to H'$ 
is smooth over $H'$, 
\item[($\beta$)] 
at least one geometric fiber of $\pi_{H'}$ is an abelian variety 
for each irreducible component of $H'$. 
\end{enumerate}

In general, the subset $H''$ of $U_2$ 
over which the projection 
$\pi_{H''}:Z^P\times_{H^P} H'' \to H''$ is smooth 
is an open $\cO$-subscheme of $U_2$. 
By \cite[Theorem~6.14]{MFK94}, any geometric fiber of 
$\pi_{U_{g,K}}$ is a polarized abelian variety. 
See also \cite[p.~705]{Nakamura99} and \cite[p.~116]{Nakamura10}. 

Next we define $U^{\dagger}_{g,K}$ to be 
the subset of $U_{g,K}$ 
parameterizing all 
subschemes $(A,L)\in U_{g,K}$ such that 
\begin{enumerate}
\item[(ix)]  the $K$-action on $A$ induced 
from the $\cG_H$-action on $(A,L)$ 
is effective and contained in $\Aut^0(A)$.
\end{enumerate} 

We see that $U^{\dagger}_{g,K}$ is 
a nonempty open $\cO$-subscheme
of $U_{g,K}$.

Finally we define $U_3$ to be the closure 
of $U^{\dagger}_{g,K}$ in $U_2$. 
It is the smallest closed $\cO$-subscheme of $U_2$ 
containing $U^{\dagger}_{g,K}$. \par
We denote the pull back to $U^{\dagger}_{g,K}$ (resp. $U_3$) of the  
universal subscheme of $X$ over $H^P=\Hilb^P(X/H_3)$ by 
\begin{equation}
\label{eq:universal subschemes}
(A_{\univ},L_{{\univ}})\quad \text{resp.}\quad (Z_{{\univ}},L_{{\univ}}). 
\end{equation}

\theoremtextend

\begin{thm}\label{thm:geom fibers over U3}
Let $R$ be a CDVR, $S:=\Spec R$, and 
$\eta$ the generic point of $S$. Let 
$h$ be a morphism from $S$ into $U_3$.
Let $(Z,\cL)$ be the pullback by $h$ of the universal subscheme 
$(Z_{\univ},L_{\univ})$ (\ref{eq:universal subschemes}) such that 
$(Z_{\eta},\cL_{\eta})$ 
is a polarized abelian variety. 
Then after a finite base change
 if necessary, $(Z,\cL)$ is isomorphic to 
$(P,\cL_P)$ in Theorem~\ref{thm:refined stable reduction}.
In particular, $(Z_0,\cL_0)$ is a TSQAS over $k(0)$. 
\end{thm}
\begin{proof}
The outline of the proof of Theorem is as follows. The generic fiber 
$(Z_{\eta},\cL_{\eta})$ of $(Z,\cL)$ is an abelian variety. 
By Theorem~\ref{thm:refined stable reduction} 
there exists an $R^*$-TSQAS $(P,\cL_P)$ 
after a suitable base change $\Spec R^*$ of $\Spec R$.
So we have two flat families $(Z,\cL)_{R^*}$ 
and $(P,\cL_P)$ over $R^*$, which we can now compare. 
For each of $(Z,\cL)$ and $(P,\cL_P)$, 
we can find a natural level-$\cG_H$ structure extending 
a level-$\cG_H$ structure of 
$(Z_{\eta},\cL_{\eta})\ (=(P_{\eta},\cL_{P,\eta}))$.
Then we can prove they are isomorphic.  See \cite[Theorem~10.4]{Nakamura10} 
for the details when $e_{\min}(H)\geq 3$. The case $e_{\min}(H)\leq 2$ 
is proved by reducing to the case $e_{\min}(H)\geq 3$ 
by Claims in Subsec.~\ref{subsec:case emin leq 2}. 
See Claim~\ref{claim:final separatedness for e1 leq 2}.
\end{proof}

\begin{thm}\label{lemma:GL(Wk)orbit}Let $G=\PGL(W_1)\times\PGL(W_2)$ and 
$k$ an algebraically closed field with $k\ni 1/N$. Then 
\begin{enumerate}
\item 
$U_3(k)=\left\{(Z,L)\in U_2(k); 
\begin{matrix}\text{a level-$\cG_H$ TSQAS with}\\
\text{characteristic $\cG_H$ action}
\end{matrix}
\right\}$
\item  
let $(Z, L)\in U_3(k)$ and $(Z', L')\in U_3(k)$ where 
$L=M\otimes O_Z$ and $L'=M\otimes O_{Z'}$ with the notation 
of Subsec~\ref{subsec:The scheme U1} Eq.(\ref{eq:M}). 
Then the following are equivalent:
\begin{enumerate} 
\item
$(Z, L)$ is $\cG_H$-isomorphic to $(Z', L')$ with respect to their 
characteristic $\cG_H$-action in the sense of Remark~\ref{subrem:k-isom}, 
\item 
$(Z, L)$ and $(Z', L')$ have the same $G$-orbit.
\end{enumerate}
\end{enumerate}
\end{thm}
\begin{proof}(1) is a corollary of Theorem~\ref{thm:geom fibers over U3}. 
 By the first assertion, any $(Z, L)\in U_3(k)$ has a natural 
characteristic $\cG_H$-action. Thus (2) makes sense. 
See \cite[Lemma~11.1]{Nakamura10} for a proof of (2).
\end{proof}

\begin{thm}\label{thm:uniform geometric quotient}
Let $G=\PGL(W_1)\times\PGL(W_2)$. Then 
\begin{enumerate}
\item $U^{\dagger}_{g,K}$ and $U_3$ are $G$-invariant,
\item the action of $G$ on $U^{\dagger}_{g,K}$ is proper 
and free (resp. proper 
with finite stabilizer)  if $e_{\min}(H)\geq 3$ 
(resp. if $e_{\min}(H)\leq 2$), 
\item the action of $G$ on $U_3$ is proper 
with finite stabilizer.
\item the uniform geometric and 
uniform categorical quotient of $U_3$ 
(resp. $U^{\dagger}_{g,K}$) by $G$ 
exists as a separated algebraic $\cO$-space, 
which we denote by 
$SQ^{*\toric}_{g,K}$ (resp. 
$A^{\toric}_{g,K}$). 
\end{enumerate}
\end{thm}

See \cite[Sec.~10-11]{Nakamura10} for 
Theorems~\ref{thm:geom fibers over U3}
-\ref{thm:uniform geometric quotient} 
when $e_{\min}(H)\geq 3$. 

\subsection{The case $e_{\min}(H)\leq 2$}
\label{subsec:case emin leq 2}
Theorems~\ref{thm:geom fibers over U3}-\ref{thm:uniform geometric quotient} 
for $e_{\min}(H)\leq 2$ 
are proved in the same manner as 
in the case $e_{\min}(H)\geq 3$
by using the following Claims. 

\begin{claim}\label{claim:etale Z/nZ covering}
Let $k$ be an algebraically closed field with 
$k\ni 1/N$, $K=H\oplus H^{\vee}$ and $N=|H|$. 
Let $(P,L)$ be a TSQAS over $k$ with $L$ 
$\cG_H$-linearized and $\cG(P,L)\simeq \cG_H$, and 
$n$ any positive integer $(\geq 3)$ 
prime to both $N$ and the characteristic of $k$.  Then 
there exists a TSQAS
$(P^{\dagger},L^{\dagger})$ over $k$ with 
the pull back $L^{\dagger}$ of $L$
$\cG_{H^{\dagger}}$-linearized
which is an  \'etale Galois 
covering of $(P,L)$ with Galois group 
$H^{\dagger}/H\simeq (\bZ/n\bZ)^g$, where 
$H$ (resp. $H^{\dagger}$) is a maximal isotropic subgroup of 
$K:=K(P,L)=H\oplus H^{\vee}$ 
(resp. of $K^{\dagger}:=K(P^{\dagger},L^{\dagger})=H^{\dagger}\oplus 
(H^{\dagger})^{\vee}=K\oplus(\bZ/n\bZ)^{2g}$). 
\end{claim}
\begin{proof}We denote the given TSQAS 
$(P,L)$ by $(P_0,\cL_0)$. Let $R$ be a CDVR, 
$(P,\cL)$ an $R$-flat family such that 
\begin{enumerate}
\item[(i)] the generic fiber $(P_{\eta},\cL_{\eta})$ 
is a level-$\cG_H$ abelian variety,
\item[(ii)] 
the closed fiber $(P_0,\cL_0)$ of $(P,\cL)$ is the given TSQAS with 
torus part $T_0$ and abelian part $(A_0,\cM_0)$.
\end{enumerate}
Since $P_0$ is a $k(0)$-TSQAS with $T_0=\Hom(X,\bG_m)$ 
for some lattice $X$ of rank $g''$, 
there exists a sublattice $Y$ of $X$ such that 
$K(P_0,\cL_0)=K(A_0,\cM_0)\oplus (X/Y)\oplus (X/Y)^{\vee}$. 
See \cite[5.14]{Nakamura99} and Definition~\ref{subdefn:Sz Tbeta}. 
Therefore it is enough to construct an \'etale 
$H^{\dagger}/H\simeq(\bZ/n\bZ)^g$-covering 
$(A_0^{\dagger},M_0^{\dagger})$ of $(A_0,M_0)$ as above. \par
Hence we may assume $P_0$ is an abelian variety. 
In what follows we denote $(P_0,\cL_0)$ by $(A,L)$. 
Let $A[m]=\ker(m\id_A)$ for any positive integer $m$. 
By the assumption, 
$A[n^2]\simeq(\bZ/n^2\bZ)^{2g}$ and $N^2=|K(A,L)|$.
Let $L'$ be the pull back of $L$ by $n\id_A$. Then 
by \cite[p.~56, Corollary~3; p.~71~(iv)]{Mumford12} there exists 
$M\in\Pic^0(A)$ such that $L'=L^{n^2}\otimes M$. 
For a line bundle $F$ on $A$, 
we denote by $\phi_{F}$ the homomorphism $A\to A^{\vee}$ defined by 
$x\mapsto T_x^*F\otimes F^{-1}$. 
Then by \cite[p.~57, Corollary~4]{Mumford12} 
$\phi_{L'}=\phi_{L^{n^2}}=n^2\phi_{L}$. 
Since $T_x^*M=M$, we have 
$$K(A,L'):=\ker(\phi_{L'})
=\ker n^2\phi_{L}=K(A,L^{n^2})\supset A[n^2].$$ 
Since $n$ is prime to $N$, 
we have $A[n^2]\cap K(A,L)=\{0\}$, hence 
$$K(A,L')=K(A,L^{n^2})=K(A,L)\oplus A[n^2].$$
 
For a maximal isotropic subgroup $G^{\dagger}$
$(\simeq (\bZ/n^2\bZ)^g)$ of $A[n^2]$, we define 
$\Delta^{\dagger}:=(n\bZ/n^2\bZ)^g$. It is  
the unique subgroup of $G^{\dagger}$ isomorphic to $(\bZ/n\bZ)^g$. 
We set $A^{\dagger}:=A/\Delta^{\dagger}$, and 
$\pi:A\to A^{\dagger}$ the projection. Now we have 
a diagram with $\varpi\pi=n\id_A$:
$$A\overset{\pi}{\to}  A^{\dagger}=A/\Delta^{\dagger} \overset{\varpi}{\to} 
A/A[n]\simeq A.
$$
As a subgroup of $K(A,L')$, we have 
\begin{gather*}
A[n^2]=\{0\}\oplus G^{\dagger}\oplus (G^{\dagger})^{\vee},\\ 
A[n]=\{0\}\oplus \Delta^{\dagger}\oplus (G^{\dagger}/\Delta^{\dagger})^{\vee},
\end{gather*}where in particular $A[n]$ 
is a totally isotropic subgroup of $A[n^2]$.

Let $L^{\dagger}:=\varpi^*(L)$. Then $L'=\pi^*(L^{\dagger})$. 
Let $(\Delta^{\dagger})^{\perp}$ be the orthogonal complement 
of $\Delta^{\dagger}$ in $K(A,L')$. Then by \cite[p.~291]{Mumford66}
 $$K^{\dagger}:=K(A^{\dagger},L^{\dagger})\simeq
(\Delta^{\dagger})^{\perp}/\Delta^{\dagger},$$ 
where we see  
$(\Delta^{\dagger})^{\perp}=
K(A,L)\oplus G^{\dagger}\oplus (G^{\dagger}/\Delta^{\dagger})^{\vee}
$, where 
$(G^{\dagger}/\Delta^{\dagger})^{\vee}\simeq (n\bZ/n^2\bZ)^g$.  
Let $H$ be a 
maximal isotropic subgroup of $K(A,L)$. 
Let $H^{\dagger}:=H\oplus \{0\}\oplus 
(G^{\dagger}/\Delta^{\dagger})^{\vee}
\subset K^{\dagger}$. 
Then $H^{\dagger}$ is  a 
maximal isotropic subgroup of $K^{\dagger}$ with 
$(H^{\dagger})^{\vee}=H^{\vee}
\oplus (G^{\dagger}/\Delta^{\dagger})\oplus\{0\}$.  
 It follows  
\begin{equation}\label{eq:Kdagger}
K^{\dagger}\simeq 
K(A,L)
\oplus (G^{\dagger}/\Delta^{\dagger})
\oplus(G^{\dagger}/\Delta^{\dagger})^{\vee}
\simeq H^{\dagger}\oplus (H^{\dagger})^{\vee}.
\end{equation}  Hence the covering $\varpi:A^{\dagger}\to A$ 
is \'etale with Galois group 
$$A[n]/\Delta^{\dagger}\simeq (G^{\dagger}/\Delta^{\dagger})^{\vee}\simeq H^{\dagger}/H\simeq 
(\bZ/n\bZ)^g,$$ and 
$\cL^{\dagger}$ is $\cG_{H^{\dagger}}$-linearized by (\ref{eq:Kdagger}). 
This proves Claim~\ref{claim:etale Z/nZ covering}.
\end{proof}

\begin{claim}\label{claim:separatedness for e1 leq 2}
\rm{(See also \cite[Lemma~6.7]{Nakamura10})}\  
Let $R$ be a complete discrete valuation ring,	
$k(\eta)$ the fraction field of $R$ and $S:=\Spec R$. 
Let $(Z_i,\phi^*_i,\tau_i)$\ $(i=1,2)$ be rigid-$\cG_H$ $S$-TSQASes 
whose generic fibers are abelian varieties. 
If $(Z_i,\phi^*_i,\tau_i)$ are 
$k(\eta)$-isomorphic, then 
they are $S$-isomorphic.
\end{claim}

Claim~\ref{claim:separatedness for e1 leq 2} 
follows from the following Claim~\ref{claim:final separatedness for e1 leq 2}.

\begin{claim}
\label{claim:final separatedness for e1 leq 2}
With the same notation as above,  
let $(P,\cL)$ be an $S$-TSQAS 
with generic fiber $(P_{\eta},\cL_{\eta})$ an abelian variety. 
Then $(P,\cL)$ is the normalization of a modified Mumford family 
with generic fiber $(P_{\eta},\cL_{\eta})$ 
by a finite base change if necessary.
\end{claim}
\begin{proof}
Let $n$ be a positive integer $\geq 3$ prime to 
the characteristic of $k(0)$ and $|H|$.  
In view of Claim~\ref{claim:etale Z/nZ covering},
by a finite base change $S^{\dagger}$ 
of $S$ 
and then by taking 
the pull back of $(P,\cL)$ to $S^{\dagger}$, we have  
an \'etale $H^{\dagger}/H\simeq (\bZ/n\bZ)^{g}$-covering 
$(P^{\dagger}_0, \cL^{\dagger}_0)$ of $(P_0,\cL_0)$ such that 
$K(P^{\dagger}_0, \cL^{\dagger}_0)=H^{\dagger}\oplus (H^{\dagger})^{\vee}$. 
From now, we denote $S^{\dagger}$ by $S$, and $(P,\cL)\times_SS^{\dagger}$ by $(P,\cL)$. \par
Let $P_{\formal}$ be the formal completion of $P$ along $P_0$.  
By \cite[Corollaire~8.4]{SGA1}, there is a category equivalence between 
\'etale coverings of $P_0$ and \'etale coverings of $P_{\formal}$. 
Hence there exists a formal
scheme $(P^{\dagger}_{\formal}, \cL^{\dagger}_{\formal})$
 which is an \'etale $(\bZ/n\bZ)^{g}$-covering 
of $(P_{\formal},\cL_{\formal})$. Then there exists 
a projective $S$-scheme $(P^{\dagger}, \cL^{\dagger})$ 
algebraizing $(P^{\dagger}_{\formal}, \cL^{\dagger}_{\formal})$ which	
is an \'etale $(\bZ/n\bZ)^{g}$-covering 
of $(P,\cL)$ with $\cL^{\dagger}$ 
the pull back of $\cL$. 
It follows that the generic fiber
$(P^{\dagger}_{\eta}, \cL^{\dagger}_{\eta})$
is a polarized abelian variety, and 
$(P^{\dagger}_0,\cL^{\dagger}_0)$ is a reduced $k(0)$-TSQAS
and  $P^{\dagger}$ is normal 
by Claim~\ref{claim:normality}. \par
Since $n\geq 3$, by \cite[10.4]{Nakamura10}
$(P^{\dagger},\cL^{\dagger})$ is 
the normalization of a modified Mumford family 
with generic fiber $(P^{\dagger}_{\eta},\cL^{\dagger}_{\eta})$.
By \cite[Corollaire~8.4]{SGA1} $(P,\cL)$ is 
the quotient  of $(P^{\dagger},\cL^{\dagger})$ 
by $(\bZ/n\bZ)^{g}$, because $(P_0,\cL_0)$ is  
the quotient  of $(P^{\dagger}_0,\cL^{\dagger}_0)$ 
by $(\bZ/n\bZ)^{g}$. 
Hence $(P,\cL)$ is 
the normalization of a modified Mumford family 
with generic fiber $(P_{\eta},\cL_{\eta})$. This proves the Claim.
\end{proof}


\begin{summary}
\label{summary:summary}Let 
$k$ be an algebraically closed field with $k\ni 1/N$. 
Let $H^{P}:=\Hilb^{P}(X/H_3)$ 
be as in Subsec.~\ref{subsec:The scheme U1}. 
We define the schemes $U_k$, $U_{g,K}$ and $U^{\dagger}_{g,K}$ as follows:
\begin{align*}
U_1&=\{(Z,L_1,L_2)\in H^{P}; \text{(i)-(ii) are true}\},\\
U_2&=\{(Z,L)\in U_1; \text{(iii)-(viii) are true}\},\\
U_{g,K}(k)&=\{(Z,L)\in U_2(k); 
\text{$(Z,L)$ is an abelian variety over $k$}\},\\
U^{\dagger}_{g,K}(k)&=\{(Z,L)\in U_{g,K}(k);\text{(ix) is true}\},\\
U_3&=\text{the closure of $U^{\dagger}_{g,K}$ in $U_2$}.
\end{align*}
Then 
\begin{enumerate}
\item $U_1$ is a closed $\cO$-subscheme of $H^{P}$, while 
$U_2$, $U_{g,K}$ and $U^{\dagger}_{g,K}$ are 
nonempty $\cO$-subschemes of $U_1$ such that 
$U^{\dagger}_{g,K}\subset U_{g,K}\subset U_2$, and  
\begin{align*}
U^{\dagger}_{g,K}(k)&=\left\{
(A,L)\in U_2(k); \begin{matrix}
\text{an abelian variety over $k$ with}\\
\text{characteristic $\cG_H$-action\phantom{$ABCc$}}
\end{matrix}
\right\}
\\
U_3(k)&=\left\{(Z,L)\in U_2(k); 
\begin{matrix}\text{a level-$\cG_H$ TSQAS over $k$ with}\\
\text{characteristic $\cG_H$ action\phantom{$ABCc$}}
\end{matrix}
\right\},
\end{align*}
\item $(Z',L')\in U_3(k)$, $(Z,L)\in U_3(k)$ are $\cG_H$-isomorphic iff 
they are in the same $G$-orbit, where $G=\PGL(W_1)\times\PGL(W_2)$,
\item there exists a nice quotient $A^{\toric}_{g,K}$ 
of $U^{\dagger}_{g,K}$ by $G$,
\item there exists a nice quotient $SQ^{*\toric}_{g,K}$ of $U_3$ by $G$, 
\item let $SQ^{\toric}_{g,K}:=(SQ^{*\toric}_{g,K})_{\red}$.
\end{enumerate}
\end{summary}

See \cite[Corollaries~10.5, 10.6]{Nakamura10} for $U_3(k)$.  

\section{Moduli for TSQASes}
\label{sec:moduli of TSQAS}
Let $\cO=\cO_N$. In this section we prove
\begin{enumerate}
\item[(i)] 
$A^{\toric}_{g,K}$ is 
the coarse moduli algebraic $\cO$-space  
for the functor of level-$\cG_H$ smooth TSQASes 
over algebraic $\cO$-spaces for any $e_{\min}(K)$,
\item[(ii)] $A^{\toric}_{g,K}\simeq A_{g,K}$  
if $e_{\min}(K)\geq 3$, 
which is the fine moduli scheme.
\end{enumerate}

We also see 
\begin{enumerate}
\item[(iii)] 
$SQ_{g,K}^{\toric}$ is 
the coarse moduli algebraic $\cO$-space for the functor 
of level-$\cG_H$ flat TSQASes 
over reduced algebraic $\cO$-spaces,
\item[(iv)] if 
$e_{\min}(K)\geq 3$, there exists a natural morphism 
$\sq:SQ_{g,K}^{\toric}\to SQ_{g,K}$, 
which is surjective and bijective on $SQ_{g,K}^{\toric}$, 
and the identity on $A_{g,K}$, hence 
$SQ_{g,K}^{\toric}$ is a projective $\cO$-scheme.  
\end{enumerate}

 \begin{thm}\label{thm:fine/coarse moduli AgK} 
Let $K=H\oplus H^{\vee}$ and $N:=|H|$.  
\begin{enumerate}
\item If $e_{\min}(H)\geq 3$, then 
$A^{\toric}_{g,K}\simeq A_{g,K}$ and 
${\cA}^{\toric}_{g,K}$ is represented by the quasi-projective 
formally smooth $\cO$-scheme 
$A_{g,K}$,
\item if $e_{\min}(H)\leq 2$, then
${\cA}^{\toric}_{g,K}$ 
has a normal 
coarse moduli algebraic $\cO$-space 
$A^{\toric}_{g,K}$.
\end{enumerate}
\end{thm}
\begin{proof}We can prove this almost in parallel 
to Theorem~\ref{thm:fine moduli AgK}. \par

\medskip
Let $\cO=\cO_N$. Let $d_{\nu}$, $W_{\nu}$
and  $W_{\nu}(K)=W_{\nu}\otimes_{\cO} V_{d_{\nu},H}$ 
be the same as in Subsec.~\ref{subsec:The scheme H1xH2}. 
Similarly let $(X_{\nu},L_{\nu})$, 
$H_{\nu}$, $(X,L)$ and $H_3=H_1\times_{\cO} H_2$ be the same as in 
Subsections~\ref{subsec:The scheme U1}--\ref{subsec:The scheme U2}.
\par

\medskip
{\bf Step 1.}\quad 
Let $T$ be any $\cO$-scheme, 
and $(P,\cL,\phi^*,\cG,\tau)$ any level-$\cG_H$ $T$-smooth TSQAS with 
$\pi:P\to T$ the projection. Then we define a natural morphism 
$\bar\eta:T\to A^{\toric}_{g,K}$ as follows.\par
The sheaf $\pi_*(d_{\nu}\cL)$ is a vector bundle of rank $d_{\nu}^gN$ over $T$.
Let $U_i$ be an affine covering of $T$ 
which trivializes both $\pi_*(d_{\nu}\cL)$.
Then 
$$\Gamma(U_i,\pi_*(d_{\nu}\cL))=\Gamma(P_{U_i},d_{\nu}\cL)\simeq 
(\cW_{\nu})_{U_i}\otimes_{\cO} V_{d_{\nu},H}$$
for some  locally $O_T$-free module $\cW_{\nu}$ 
of rank $d_{\nu}^g$ with trivial $\cG$-action. 

Since 
$d_{\nu}\cL_t$ is very ample, we can choose 
closed $\cG$-immersions 
$$(\phi_{\nu})_{U_i}:P_{U_i}\to \bP(W_{\nu}(K))_{U_i}$$ by 
the linear system associated to $\pi_*(d_{\nu}\cL)_{U_i}$ 
such that 
\begin{equation}\label{eq:freedom of isom}
\rho((\phi_{\nu})_{U_i}^*,\tau_{U_i})=\id_{W_{\nu}}\otimes U_{d_{\nu},H}
\end{equation}
 We caution that $(\phi_{\nu})_{U_i}$ is 
not unique, there is freedom of isomorphisms by 
$\GL(W_{\nu},O_{U_i})$. 

By (\ref{eq:freedom of isom}) 
the image of $(\phi_{\nu})_{U_i}$ is $\cG$-invariant, so 
the image of $(\phi_{\nu})_t$ is $\cG_H$-invariant for any $t\in T$, 
 Since $\cL=q_1d_1\cL+q_2d_2\cL$, 
$\cL_{U_i}$ is $\cG_{U_i}$-linearized. Hence 
$(P_{U_i},\cL_{U_i})$ has a $\cG_{U_i}$-action, that is, fiberwise 
$(P_t,\cL_t)$ has a $\cG_H$-action. By the definition of 
level-$\cG_H$ TSQASes, this $\cG_H$-action on $(P_t,\cL_t)$ is characteristic. 
Hence the image of $(\phi_{\nu})_{U_i}$ is contained in $U^{\dagger}_{g,K}$ 
by Theorem~\ref{lemma:GL(Wk)orbit} or Summary~\ref{summary:summary}. 
It follows that  $(P_{U_i},\cL_{U_i})$ 
is the pull back by a morphism 
$U_i\to U^{\dagger}_{g,K}$ of the universal subscheme 
$(X,H_3)$ in Subsec~\ref{subsec:The scheme H1xH2}.\par
On $U_i\cap U_j$,  
$\Gamma(U_i,\pi_*(d_{\nu}\cL))$ and $\Gamma(U_j,\pi_*(d_{\nu}\cL))$ are 
identified by $\GL(W_{\nu}\otimes \Gamma(O_{U_i\cap U_j}))$.
Thus we have a morphism 
$$j: T\to U^{\dagger}_{g,K}/\PGL(W_1)\times \PGL(W_2)=A^{\toric}_{g,K},
$$where $G=\PGL(W_1)\times \PGL(W_2)$. 
This induces a morphism of functors
\begin{equation}
\label{eq:morphism of AgKtoric to AgKtoric}
f: \cA^{\toric}_{g,K}\to h_W,\quad W:=A^{\toric}_{g,K}.
\end{equation}

The argument so far is true regardless of the value of $e_{\min}(H)$.
\par

\medskip
{\bf Step 2.}
Now we assume $e_{\min}(H)\geq 3$. \par
\smallskip
{\em Step 2-1.}\quad
Any level-$\cG_H$ $T$-smooth TSQAS 
is a level-$\cG_H$ $T$-smooth PSQAS 
with $\cV=\pi_*(\cL)$, 
and vice versa. 
Hence the functors are the same : 
${\cA}^{\toric}_{g,K}={\cA}_{g,K}$.\par
\smallskip
{\em Step 2-2.}\quad 
Now we assume $e_{\min}(H)\geq 3$. 
There is the universal subscheme over $U^{\dagger}_{g,K}$
(\ref{eq:universal subschemes})
$$(A_{\univ},\cV_{\univ},L_{\univ},\phi_{\univ},\cG_{\univ},\tau_{\univ})$$ 
where $\cG_{\univ}=\cG_H\times U^{\dagger}_{g,K}$, $\tau_{\univ}=U_H$ 
(acting on $\bP(V_H)_{U^{\dagger}_{g,K}}$), 
$\cV_{\univ}=V_H\otimes O_{U^{\dagger}_{g,K}}$  and 
we choose a closed immersion 
$\phi_{\univ}:A_{\univ}\to \bP(V_H)_{U^{\dagger}_{g,K}}$, such that 
$\rho(\phi_{\univ},\tau_{\univ})=U_H$. This is a rigid 
level-$\cG_H$ $U^{\dagger}_{g,K}$-smooth PSQAS. 
Hence we have a morphism $\eta^{\dagger}:U^{\dagger}_{g,K}\to A_{g,K}$ 
because $A_{g,K}$ is the fine moduli scheme of $\cA_{g,K}$ 
by Theorem~\ref{thm:fine moduli AgK}. 
Since the morphism 
$\eta^{\dagger}$ is $G=\PGL(W_1)\times \PGL(W_2)$-invariant, 
we have a morphism
 $$\bar\eta:A^{\toric}_{g,K}\to A_{g,K}.$$\par
\smallskip{\em Step 2-3.}\quad 
Conversely since $A_{g,K}$ is the fine moduli scheme for $\cA_{g,K}$, 
there exists the universal level-$\cG_H$ PSQAS 
$$\pi_A :(Z_A,\cV_A,L_A,\phi_A,\cG_A,\tau_A)\to A_{g,K}.$$  Then 
we apply Step 1 to the universal level-$\cG_H$ PSQAS over $A_{g,K}$.
We have a morphism 
from $A_{g,K}$ to $A^{\toric}_{g,K}$, 
which is evidently the inverse of $\bar\eta$. 
This proves that $\bar\eta$ is an isomorphism. This proves 
the first assertion of Theorem~\ref{thm:fine/coarse moduli AgK} 
by Theorem~\ref{thm:fine moduli AgK}. See \cite[Lemma~11.5]{Nakamura10}.\par
\medskip
{\bf Step 3.}\quad 
We consider next the case $e_{\min}(H)\leq 2$. 
By Step~1 (\ref{eq:morphism of AgKtoric to AgKtoric}), we have a 
morphism of functors 
$f: \cA_{g,K}^{\toric}\to h_W$ where $W:=A_{g,K}^{\toric}$.  
To prove that $A^{\toric}_{g,K}$ is a 
coarse moduli algebraic $\cO$-space 
for ${\cA}^{\toric}_{g,K}$, 
it remains to prove 
\begin{enumerate}
\item[(a)]  $f(\Spec k) : \cA^{\toric}_{g,K}(\Spec k)\to A^{\toric}_{g,K}(\Spec k)$ is bijective 
for any algebraically closed field $k$ over $\cO$,
\item[(b)]  For any algebraic $\cO$-space $V$, and any 
morphism $g : \cA^{\toric}_{g,K}\to h_V$, 
there is a unique morphism $\chi:h_W\to h_V$ 
such that $g=\chi\circ f$,
\end{enumerate}where $W=A_{g,K}^{\toric}$, 
$h_V$ is the functor defined by $h_V(T)=\Hom(T,V)$.

The assertion (b) is proved similarly to Step~1 and Step~2-2.\par 
The assertion (a) follows  
from Theorem~\ref{lemma:GL(Wk)orbit}. In fact, let 
$$\sigma_j:=(Z_j,L_j,\phi_j^*,\cG_H,\tau_j)$$ be a level $\cG_H$ 
smooth $k$-TSQAS. Since $A^{\toric}_{g,K}$ is 
the orbit space of $U_{g,K}$ by 
$G:=\PGL(W_1)\times\PGL(W_2)$, $(Z_1,L_1)$ and $(Z_2,L_2)$ 
determine the same point of $A^{\toric}_{g,K}$ iff 
$(Z_1,L_1)$ and $(Z_2,L_2)$ have the same $G$-orbit. 
By Theorem~\ref{lemma:GL(Wk)orbit}, 
$(Z_1,L_1)$ and $(Z_2,L_2)$ have the same $G$-orbit iff 
$(Z_1,L_1)$ and $(Z_2,L_2)$ are $\cG_H$-isomorphic 
with respect to their characteristic $\cG_H$-action in the sense of 
Remark~\ref{subrem:k-isom}. 
Thus it suffices to prove 
that $\sigma_1\simeq\sigma_2$ iff $(Z_1,L_1)$ and $(Z_2,L_2)$ 
are $\cG_H$-isomorphic. \par
If $\sigma_1\simeq \sigma_2$, then 
by definition $(Z_1,L_1)\simeq (Z_2,L_2)$. \par
Conversely assume $(Z_1,L_1)\simeq (Z_2,L_2)$ $\cG_H$-isomorphic 
with respect to their characteristic $\cG_H$-action. 
Let $f:(Z_1,L_1)\to (Z_2,L_2)$ be the $\cG_H$-isomorphism. 
Hence $(f^*)^{-1}\rho_{\tau_1,L_1}(g)f^*=\rho_{\tau_2,L_2}(g)$. 
Meanwhile we can choose a  $\cG_H$-isomorphism 
$\phi_j^*:V_H\otimes k\to \Gamma(Z_j,L_j)$
such that $\rho(\phi_j^*,\tau_j)=U_H$. Let 
$h:=(\phi_1^*)^{-1}f^*\phi_2^*$. Then we see 
$U_Hh=hU_H$. 
Since $U_H$ is an irreducible representation of $\cG_H$, 
$h$ is a nonzero scalar. Hence $f^*\phi_2^*=c\phi_1^*$ for some unit $c$. 
It follows 
from Definition~\ref{subdefn:morp of T-TSQAS} that $\sigma_1\simeq\sigma_2$. 
This proves (a).  
Thus $A^{\toric}_{g,K}$ is a 
coarse moduli algebraic $\cO$-space 
for ${\cA}^{\toric}_{g,K}$. 
\par
\medskip
{\bf Step 4.}\quad Finally we prove  
that $A^{\toric}_{g,K}$ is reduced 
for $e_{\min}(H)\leq 2$. We use the same notation as in 
the proof of Theorem~\ref{thm:fine moduli AgK}. Let $k$ be any 
algebraically closed field with $k\ni 1/N$, $(A,L_0)$ 
be an abelian variety over $k$ with $L_0$ $\cG_H$-linearized, and 
$\tau_0$ be the $\cG_H$-action 
associated to the  $\cG_H$-linearization of $L_0$. 
Let $\sigma_0:=(A,L_0,\phi_0^*,\cG_H,\tau_0)$ be 
a rigid level-$\cG_H$ $k$-smooth TSQAS. \par
Let $\cC=\cC_W$  be
the category of local Artinian $W$-algebra with $k=R/m_R$. 
We define a subfunctor $F:=F_{\sigma_0}$ of $\cA_{g,K}^{\toric}$ by 
\begin{align*}
F(R)&=\left\{\sigma:=(Z,L,\phi^*,(\cG_H)_R,\tau)\in\cA_{g,K}^{\toric}(R);
\sigma\otimes k\simeq \sigma_0
\right\}
\end{align*}where $R\in\cC$ and the isomorphism 
$\sigma\otimes k\simeq \sigma_0$ is not fixed in $F$. \par
Let $(X,\cL)$, $K_{\suniv}=\ker(\lambda(\cL))$, 
$\cG_{\suniv}:=\cG(X,\cL):=\cL^{\times}_{K_{\suniv}}$,   
$\cV_{\suniv}:=\Gamma(X,\cL)$ and the action 
$\tau_{\suniv}$ of $\cG_{\suniv}$ on $(X,\cL)$ 
be the same as in Subsec~\ref{subsec:deform of sep pol AV}.
Since $\lambda(\cL):X\to X^{\vee}$ 
is separable, 
$K_{\suniv}$ is isomorphic to $(H\oplus H^{\vee})_{\cO_W}$, 
hence $\cG_{\suniv}\simeq (\cG_H)_{\cO_W}$. 
If $e_{\min}(K_{\suniv})\geq 3$, we choose the unique closed 
$\cG_H$-immersion $\phi_{\suniv}$ of $X$ into $\bP(\cV_{\suniv})
\simeq \bP(V_H)_{\cO_W}$ such that 
$\rho(\phi_{\suniv},\tau_{\suniv})=U_H$. If $e_{\min}(K_{\suniv})\leq 2$, 
then we choose the unique 
$\cG_H$-isomorphism $\phi_{\suniv}^*:(V_H)_{\cO_W}\to \Gamma(X,\cL)$ such that 
$\rho(\phi^*_{\suniv},\tau_{\suniv})=U_H$. 
   In any case
 we have a level-$\cG_H$ smooth TSQAS over $\cO_W$
$$(X,\cL,\cV_{\suniv},\phi^*_{\suniv},\cG_{\suniv},\tau_{\suniv}).
$$

Now we shall define a morphism of functors 
$h:P(A,\lambda(L_0))\to F$ over 
$\cC=\cC_W$. 
Let $R\in\cC$. 
By Subsec~\ref{subsec:deform of sep pol AV}, 
for $(Z,\lambda(L))\in P(A,\lambda(L_0))(R)$, $R\in\cC$, 
we have a unique morphism
$$\rho\in\Hom(\Spec R,\Spf \cO_W)=\Hom_{\hat{\cC}}(\cO_W,R)$$ such that 
$(Z,\lambda(L))=\rho^*(X,\lambda(\cL))$.
Then we define 
$$h(Z,\lambda(L))=\rho^*(X,\cL,\cV_{\suniv},\phi^*_{\suniv},\cG_{\suniv},\tau_{\suniv})\in F(R).
$$ 

One can check that this is well-defined.

Subsec.~\ref{subsec:deform of sep pol AV} shows that 
$h(R):P(A,\lambda(L_0))(R)\to F(R)$ is surjective for any $R\in\cC$.
In general, $h$ is not injective. 
Let 
\begin{equation*}
G_0:=\Aut(\sigma_0)=\{f\in\Aut(A); f(0)=0,\ f^*\sigma_0\simeq \sigma_0\},
\end{equation*}where $0$ is the zero of $A$. 
Since $f^*L_0\simeq L_0$ for any $f\in G_0$, we have 
$f^*(3L_0)\simeq 3L_0$. Since $3L_0$ is very ample, $G_0$ is 
an algebraic $k$-group. $G_0$
has trivial connected part 
because $f(0)=0$ for any $f\in G_0$.
Hence $G_0$ is a finite group scheme, acting
nontrivially on $P(A,\lambda(L_0))$.  
Then  
\begin{align*}
F(R)&=P(A,\lambda(L_0))(R)/G_0\\
&=\Hom(\cO_W/\frak a,R)/G_0\\
&=\Hom((\cO_W/\frak a)^{G_0-\inv},R)
\end{align*}
whence $F$ is pro-represented by 
$(\cO_W/\frak a)^{G_0-\inv}$, which is normal. 
This proves that the formal completion of any local ring of $A^{\toric}_{g,K}$ is normal. Hence it satisifies 
$(\text{R}_1)$ and $(\text{S}_2)$ by Serre's criterion. 
See Remark~\ref{subrem:Serre}. 
This implies 
that 
any local ring of $A^{\toric}_{g,K}$ satisfies 
$(\text{R}_1)$ and $(\text{S}_2)$. 
Hence $A^{\toric}_{g,K}$ is normal.  
\end{proof}

\begin{subrem}\label{subrem:Serre}
Let $A$ be a noetherian local ring.  Then 
$A$ is normal if and only if 
$(\text{R}_1)$ and $(\text{S}_2)$ are true for $A$, 
where 
\begin{enumerate}
\item  $(\text{S}_2)$ is true if and only if 
$\depth(A_p)\geq \inf(2,\text{ht}(p))$
for all $p\in\Spec(A)$,
\item $(\text{R}_1)$ is true if and only if $A$ is codimension one regular. 
\end{enumerate}

See \cite[Theorem~39]{Matsumura70} and	 
\cite[IV$_2$, 5.8.5 and 5.8.6]{EGA}.	
\end{subrem}

\begin{thm}\label{thm:coarse moduli}  
{\rm (\cite{Nakamura10})}
Let $N=|H|$ and $SQ^{\toric}_{g,K}=(SQ^{*\toric}_{g,K})_{\red}$. 
For any $K=H\oplus H^{\vee}$,
the functor ${\cal{SQ}}^{\toric}_{g,K}$ 
of level-$\cG_H$ TSQASes $(P,\phi^*,\tau)$ over reduced algebraic 
$\cO$-spaces is coarsely represented 
by a proper (hence separated) reduced algebraic $\cO$-space
$SQ^{\toric}_{g,K}$.  
\end{thm}
\begin{proof}We imitate the proof of 
Theorem~\ref{thm:fine/coarse moduli AgK}.   
Let $(P\overset{\pi}{\to}T,L,\phi^*,\cG,\tau)$ be a level-$\cG_H$ 
$T$-flat TSQAS with $T$ reduced. Then by Step 1 of 
Theorem~\ref{thm:fine/coarse moduli AgK}, we have a morphism 
$$j:T\to U_3/G=SQ_{g,K}^{*\toric},
$$where $G=\PGL(W_1)\times\PGL(W_2)$. Hence we have a morphism 
$$j_{\red}:T_{\red}=T\to (SQ_{g,K}^{*\toric})_{\red}=:SQ_{g,K}^{\toric}.
$$
 This induces a morphism of functors
\begin{equation}
\label{eq:morphism of SQgKtoric to SQgKtoric}
f: {\cS}{\cQ}^{\toric}_{g,K}\to h_W, \quad 
W=SQ_{g,K}^{\toric}.
\end{equation} 

As in Theorem~\ref{thm:fine/coarse moduli AgK}~Step~3, 
it remains to prove 
\begin{enumerate}
\item[(a)]  $f(\Spec k) : {\cS}{\cQ}^{\toric}_{g,K}(\Spec k)\to 
SQ^{\toric}_{g,K}(\Spec k)$ is bijective 
for any algebraically closed field $k$ over $\cO$,
\item[(b)]  For any algebraic $\cO$-space $V$, and any 
morphism $g : {\cS}{\cQ}^{\toric}_{g,K}\to h_V$, 
there is a unique morphism $\chi:h_W\to h_V$ 
such that $g=\chi\circ f$,
\end{enumerate}where $h_V$ is the functor defined by $h_V(T)=\Hom(T,V)$. 
For a reduced space $T$, $h_V(T)=h_{V_{\red}}(T)$, that is, 
$h_V=h_{V_{\red}}$ over $Space_{\red}$. Hence we may assume $V$ is reduced.

We shall prove (b). Let $g:{\cS}{\cQ}_{g,K}^{\toric}\to h_V$ 
be any morphism for a reduced algebraic $\cO$-space $V$. 
The universal subscheme $(Z_{\univ},L_{\univ})$ has a natural 
$\cG_H$-action which is characteristic for any fiber 
$(Z_{\univ,u},L_{\univ,u})$ $(u\in U_3)$. We choose 
$\phi^*_{\univ}=\id_{V_H\otimes O_{U_3}}$. Thus we have 
a rigid level-$\cG_H$ $U_3$-flat TSQAS 
$(Z_{\univ},L_{\univ},\phi^*_{\univ},\cG_H,\tau_{\univ})$ 
over $U_3$. Hence by $g:{\cS}{\cQ}_{g,K}^{\toric}\to h_V$ we have a 
morphism $\widetilde\chi:U_3\to V$, which turns out to be $G$-invariant. 
Hence we have a morphism 
${\bar \chi}: SQ_{g,K}^{*\toric}\to V$, 
hence $\chi:={\bar \chi}_{\red}: SQ_{g,K}^{\toric}\to V_{\red}=V$. 
It is clear that $g=\chi\circ f$.  
 \par
By the same argument as in the proof of 
Theorem~\ref{thm:fine/coarse moduli AgK}~Step~3~(a), we see 
${\cS}{\cQ}_{g,K}^{\toric}(\Spec k)
=SQ_{g,K}^{*\toric}(k)=SQ_{g,K}^{\toric}(k)$. This proves (a). 
This completes the proof. 
\end{proof}

\begin{thm} {\rm (\cite{Nakamura10})} Suppose $e_{\min}(K)\geq 3$. 
 Then 
\begin{enumerate}
\item both $SQ_{g,K}$ and 
$SQ^{\toric}_{g,K}$ are compactifications of $A_{g,K}$,
\item there exists a bijective $\cO$-morphism 
$$\sq:SQ^{\toric}_{g,K}\to SQ_{g,K}$$ 
extending the identity of $A_{g,K}$, 
\item their normalizations are isomorphic : 
$(SQ^{\toric}_{g,K})^{\norm}\simeq 
(SQ_{g,K})^{\norm}$.
\end{enumerate}
\end{thm}

\begin{cor}$SQ^{\toric}_{g,K}$ is a projective scheme if $e_{\min}(K)\geq 3$.
\end{cor}
\begin{proof}
Since $SQ^{\toric}_{g,K}$ is finite over $SQ_{g,K}$ and 
$SQ_{g,K}$ is a scheme, 
$SQ^{\toric}_{g,K}$ is a scheme by \cite[Theorem~4.1, p.~169]{K71}, 
hence it is a projective scheme because  $SQ_{g,K}$ 
is projective by (\ref{eq:SQgK}).
\end{proof}

\section{Morphisms to Alexeev's complete moduli spaces}
\label{sec:The morphisms to Alexeev's complete moduli space}

In this section 
\begin{enumerate}
\item[(i)] we briefly review 
Alexeev \cite{Alexeev02},
\item[(ii)] then report that 
\begin{enumerate}
\item[(a)] any $T$-flat TSQAS has a canonical semi-abelian action if any generic fiber is smooth,
\item[(b)] 
$SQ^{\toric}_{g,1}\simeq\barAP_{g,1}^{\main}$. 
\end{enumerate}\end{enumerate}
\begin{defn} \cite{Alexeev02}\ \ Let $k$ be an algebraically closed field. 
A $g$-dimensional 
semiabelic $k$-pair of degree $d$ is 
a quadruple $(G,P,\cL,\Theta)$ such that
\begin{enumerate}
\item[(i)] $P$ is a connected seminormal {\it complete} $k$-variety, 
and any irreducible component of $P$ is $g$-dimensional,
\item[(ii)]  $G$ is a semi-abelian $k$-scheme acting on $P$,
\item[(iii)] there are only finitely many $G$-orbits,
\item[(iv)] the stabilizer subgroup of every point of $P$ is connected, 
reduced and lies in the torus part of $G$,
\item[(v)] $\cL$ is an ample line bundle on $P$ with $h^0(P,\cL)=d$, 
\item[(vi)] 
$\Theta$ is an effective Cartier divisor of $P$ with $\cL=O_P(\Theta)$ 
which does not contain any $G$-orbits.
\end{enumerate}
\end{defn}

Recall that a variety $Z$ 
is said to be {\em seminormal} if any bijective morphism 
$f:W\to Z$ with $W$ reduced is an isomorphism. 

\begin{defn}Let $T$ be a scheme. 
A $g$-dimensional 
semiabelic $T$-pair of degree $d$ is 
a quadruple $(G,P\overset{\pi}{\to} T,\cL,\Theta)$ such that
\begin{enumerate}
\item[(i)] 
$G$ is a semi-abelian group $T$-scheme 
of relative dimension $g$, 
\item[(ii)] $P$ is a proper flat	
$T$-scheme, on which $G$ acts, 
\item[(iii)] $\cL$ is a $\pi$-ample line bundle on $P$ 
with $\pi_*(\cL)$ locally free, 
\item[(iv)]	
any geometric fiber $(G_t,P_t,\cL_t,\Theta_t)$ $(t\in T)$ is a stable 
semiabelic pair of degree d.
\end{enumerate}
\end{defn}

\begin{defn}We define two functors: 
for any scheme $T$
\begin{align*}
\barcAP_{g,d}(T)&=\left\{
\begin{matrix}
(G,P\overset{\pi}{\to} T,D);&\text{semi-abelic $T$-pair of degree $d$}
\end{matrix}
\right\}/\text{$T$-isom.},\\
{\cA}{\cP}_{g,d}(T)&=
\left\{
\begin{matrix}
(G,A\overset{\pi}{\to} T,D);&\text{semi-abelic $T$-pair of degree $d$}\\
&\text{$G$ is an abelian $T$-scheme\phantom{aaa}}
\end{matrix}
\right\}/\text{$T$-isom.}.
\end{align*}
\end{defn}

\begin{thm}{\rm (Alexeev \cite[5.10.1]{Alexeev02})}
\begin{enumerate}
\item 
The component $\barcAP_{g,d}$ of the moduli stack of semiabelic pairs 
containing the moduli stack ${\cA}{\cP}_{g,d}$ of abelian pairs as well as 
pairs of the same numerical type is a proper Artin stack 
with finite stabilizer,
\item It has a proper coarse moduli
algebraic space $\barAP_{g,d}$  over $\bZ$. 
\end{enumerate}
\end{thm}

\begin{defn}
In order to compare $\barAP_{g,d}$ with $SQ^{\toric}_{g,K}$ 
we consider the pullback 
of $\barAP_{g,d}$ to $\cO_d$, which we denote $\barAP_{g,d}$ 
by abuse of notation.
Let $\barAP_{g,d}^{\main}$ be
the closure of $AP_{g,d}$ in $\barAP_{g,d}$. 
$\barAP_{g,d}^{\main}\neq\barAP_{g,d}$ in general.
\end{defn}

\subsection{The semi-abelian group action on a $T$-TSQAS}
\label{subsec:semiabelian action of TSQAS}
The purpose of this subsection to construct 
a semiabelian group action on any $T$-flat TSQAS. See \cite{Nakamura14}.

\begin{sublemma}
\label{sublemma:local structure of P0 along Z(tau)}Let $(P_0,\cL_0)$ be a 
totally degenerate TSQAS over $k$. Let 
$X$ be a lattice of rank $g$ associated to $P_0$, 
 $\Del_B$ the Delaunay decomposition 
of $X_{\bR}$ also associated to $P_0$, and 
$\Del_B^{(d)}$ the set of all $d$-dimensional 
Delaunay cells in $\Del_B$. 
Let $\tau\in\Del_B^{(g-1)}$ and $\sigma_i\in\Del_B^{(g)}$ 
$(i=1, 2)$ be Delaunay cells  
such that $\tau=\sigma_1\cap\sigma_2$. 
Let $Z(\sigma_i)=\overline{O(\sigma_i)}$ be
the irreducible component of $P_0$ corresponding to $\sigma_i$. Then 
$P_0$ is, along $O(\tau)$, 
isomorphic to the subscheme of $O(\tau)\times\bA^2_k$ given by 
$$\Spec \Gamma(O_{O(\tau)})[\zeta_1,\zeta_2]/(\zeta_1\zeta_2),$$
where $\bA^2_k=\Spec k[\zeta_1, \zeta_2]$: 
the two-dimensional affine space over $k$. Here 
$Z(\sigma_i)$ is given by $\zeta_i=0$, and $P_0$ is, along $O(\tau)$, the union of $Z(\sigma_1)$ and $Z(\sigma_2)$, while $O(\tau)\ (\simeq \bG_{m,k}^{g-1})$ is given by $\zeta_1=\zeta_2=0$, 
which is a Cartier divisor of each $Z(\sigma_i)$.
\end{sublemma}

\begin{subrem}Instead of proving  
Lemma~\ref{sublemma:local structure of P0 along Z(tau)} here,  
we revisit Case~\ref{subcase:first} to illustrate the situation. 
In this case,  $P_0=Q_0$, and we recall 
the open affine subset $U_0(0)$ of 
$P_0$:
\begin{align*}
(U_0)_0&=\Spec R[qw_1, qw_2, qw_1^{-1}, qw_2^{-1}]\otimes k(0)\\
&\simeq\Spec k(0)[u_1, u_2, v_1, v_2]/(u_1v_1, u_2v_2),
\end{align*}where $(U_0)_0=U_0\otimes k(0)$. \par
Let $\tau=[0, 1]\times\{0\}\in\Del^{(1)}_B$. Then 
there are exactly two Delaunay cells $\sigma=\sigma_i$ $(i=1, 2)$
such that $\tau\subset\sigma$ and $\sigma\in\Del^{(2)}_B$, where
$$\sigma_1=[0, 1]\times [0, 1],\quad \sigma_2=[0, 1]\times[-1, 0].$$

We see 
\begin{align*}
O(\tau)
&\simeq\Spec k(0)[u_1^{\pm 1}, u_2, v_1, v_2]/(u_2, v_1, v_2)
\simeq \Spec k(0)[u_1^{\pm}].
\end{align*} 
Let $(U_0)_0(\tau)$ be the subset of $(U_0)_0$ where $u_1$ is invertible. 
Then we have 
\begin{align*}
(U_0)_0(\tau)&=\Spec k(0)[u_1^{\pm}, u_2, v_2]/(u_2v_2),\\
Z(\sigma_1)&=\Spec k(0)[u_1^{\pm 1}, u_2, v_2]/(u_2),\\
Z(\sigma_2)&=\Spec k(0)[u_1^{\pm 1}, u_2, v_2]/(v_2).
\end{align*} 

This is what is meant by ``along $O(\tau)$'' 
in Lemma~\ref{sublemma:local structure of P0 along Z(tau)}.
\end{subrem}

\begin{defn}\label{defn:log rational one forms}
Let $\Omega^1_{P_0}$ be the sheaf of 
germs of regular one-forms over $P_0$, and 
$\Theta_{P_0}:={\cH}{om}_{O_{P_0}}(\Omega^1_{P_0},O_{P_0})
={\cD}er(O_{P_0})$. Let $\Sing (P_0)$ be the singular locus of $P_0$,
and $Z$ the union of all 
semi-abelian orbits in $P_0$ of codimension at least two. 
Let $\pi:P\to S$ be a simplified Mumford family with $P_0$ the closed fiber.
Then we  define ${\wOmega}_{P\setminus Z}$
to be the sheaf of germs 
of relative rational one-forms $\phi$ over $P\setminus Z$ such that 
$\phi$ is 
regular outside $\Sing(P_0)$,
and the restriction $\phi_{|P_0}$ of $\phi$ to $P_0$ 
has log poles along every 
$(g-1)$-dimensional irreducible component of 
$\Sing(P_0\setminus Z)$.
We define 
\begin{align*}{\wOmega}_{(P\setminus Z)/S}:&
=({\wOmega}_{P\setminus Z})/O_{P\setminus Z}(dq/q),\\
{\wOmega}_{P/S}:&=i_*({\wOmega}_{(P\setminus Z)/S}),
\end{align*}where $i$ is the inclusion $i:P\setminus Z\hookrightarrow P$, 
and $q$ is a uniformizer of $R$.  
\end{defn}
\begin{defn}
\begin{gather*}
{\wOmega}_{P_0\setminus Z}:={\wOmega}_{(P\setminus Z)/S}\otimes k(0),\quad
{\wOmega}_{P_0}:=j_*({\wOmega}_{P_0\setminus Z}),\\
\dTheta_{P_0}:={\cH}om_{O_{P_0}}(\wOmega_{P_0},O_{P_0}),\quad
\dOmega_{P_0}:={\cH}om_{O_{P_0}}(\dTheta_{P_0},O_{P_0}).
\end{gather*}where
$j$ is the inclusion $j:P_0\setminus Z\hookrightarrow P_0$. 
\end{defn}

\begin{sublemma}\label{sublemma:dTheta}
Let $P_0$ be a (not necessarily totally degenerate) $k(0)$-TSQAS 
of dimension $g$. 
Then we have 
$\dTheta_{P_0}\simeq O^{\oplus g}_{P_0}$, 
$\dOmega_{P_0}\simeq 
O^{\oplus g}_{P_0}$.
\end{sublemma}

We note that by \cite[p.~112]{Rim72},
the tangent space of the automorphism group $\Aut(P_0)$ 
is given by $H^0(P_0,\Theta_{P_0})$. 

\begin{thm}
\label{thm:dAut_T(P)}
Let $T$ be a reduced scheme, 
$(P\overset{\pi}{\to} T,\cL)$ a level-$\cG_H$ $T$-flat TSQAS 
with any generic fiber smooth. 
Let $T'$ be the maximal open subscheme of $T$ 
such that $\pi$ is smooth. Then
\begin{enumerate}
\item  the fiberwise identity component $\Aut^0_{T'}(P')$ of $\Aut_{T'}(P')$ is an abelian $T'$-scheme, and 
$P':=P\times_TT'$ is an abelian $T'$-torsor, 
\item let $\dAut_T(P)$ be the closure of $\Aut^0_{T'}(P')$ in $\Aut_T(P)$ 
with reduced structure. Then 
$\dAut_T(P)$ is a $T$-flat group scheme,
\item  let $\ddAut_T(P)$ be the minimal open subgroup scheme of $\dAut_T(P)$.
It is the fiberwise identity component of $\Aut_T(P)$, 
which is a semi-abelian $T$-scheme,
\item the tangent space of $(\dAut_T(P))_t$ at $\id_{P_t}$ is 
$H^0(P_t,\dTheta_{P_t})$ for  any point $t\in T$. 
\end{enumerate}
\end{thm}


\begin{thm}
\label{cor:isom SQg1 and APg1main}
$SQ^{\toric}_{g,1}\simeq\barAP_{g,1}^{\main}$.
\end{thm}
\begin{subrem}Assume Theorem~\ref{thm:dAut_T(P)}. Then 
Theorem~\ref{cor:isom SQg1 and APg1main} is proved 
as follows. 
The scheme $U_{g,1}^{\dagger}$ is reduced, as is shown 
in the same manner as in Theorem~\ref{thm:fine/coarse moduli AgK}, hence 
the closure $U_3$ of $U_{g,1}^{\dagger}$ is also reduced. 
Over $U_3$ we have a universal family 
$$(Z_{\univ},L_{\univ})_{U_3}:=(Z_{\univ},L_{\univ})\times_{H^P} U_3.$$
Since $U_3$ is reduced and any fiber of $(Z_{\univ},L_{\univ})_{U_3}$ is a TSQAS by Theorem~\ref{thm:uniform geometric quotient}, 
we can apply Theorem~\ref{thm:dAut_T(P)}.\par
Since $A_{g,1}\simeq AP_{g,1}$ by $d=1$, it is reduced 
by Theorem~\ref{thm:fine/coarse moduli AgK}.
Hence the closure $\barAP_{g,1}^{\main}$ 
of $AP_{g,1}$ in $\barAP_{g,1}$ is reduced because 
it is the intersection of all closed algebraic subspaces of 
$\barAP_{g,1}$ containing $AP_{g,1}=(AP_{g,1})_{\red}$, hence 
it is the intersection of all closed reduced algebraic subspaces 
of $\barAP_{g,1}$ containing $(AP_{g,1})_{\red}$. \par 
It follows from Theorem~\ref{thm:fine/coarse moduli AgK} 
that we have a $G$-morphism 
from $U_3$ to $\barAP_{g,1}^{\main}$ where $G=\PGL(W_1)\times\PGL(W_2)$. 
By the universality of the categorical quotient, we have a morphism
$\sqap: SQ^{\toric}_{g,1}\to \barAP_{g,1}^{\main},
$ which is an isomorphism over $A_{g,1}$. 
Since $SQ^{\toric}_{g,1}$ is proper, 
$\sqap$ is surjective. The forgetful map 
$$\barAP_{g,1}^{\main}\ni (G,P,\cL,\Theta)\mapsto (P,\cL)\in SQ^{\toric}_{g,1}$$is the left inverse of $\sqap$. This proves 
$SQ^{\toric}_{g,1}\simeq\barAP_{g,1}^{\main}$ because both 
$SQ^{\toric}_{g,1}$ and $\barAP_{g,1}^{\main}$ are reduced. 
\end{subrem}

\begin{prop}
\label{prop:morphism sqap}
Let $N=\sqrt{|K|}$. We define a map $\sqap$ by
$$SQ^{\toric}_{g,K}\ni (P,\cL,\phi^*,\tau)\times [v] 
\mapsto (\ddAut(P),P,\cL,\Div\phi^*(v))\in \barAP_{g,K},$$
where $v\in V_H$, $\Div\phi^*(v)$ is a 
Cartier divisor of $P$ defined by $\phi^*(v)$. 
Then 
there exists a nonempty Zariski open subset $U$ of 
$\bP(V_H)$ such that 
\begin{enumerate}
\item $\sqap\otimes \cO_{N^3}$ is a finite morphism
 with $\Gal=\Aut_c(\cG_H)$ acting freely,  
\item for any $u\in U$, 
$\sqap : SQ^{\toric}_{g,K}\times\{u\}\to \barAP_{g,N}$ is 
injective.
\end{enumerate}
\end{prop}

Details will appear in \cite{Nakamura14}.

\section{Related topics}
\label{sec:related topics}

\subsection{Stability}
\label{subsec:stability}

Let us look at the following example. Let $X=\Spec \bC[x,y]$ 
and $\bG_m=\Spec \bC[s,s^{-1}]$. Then $\bG_m$ acts on $X$ by 
$(x,y)\mapsto (sx,s^{-1}y)$.  Let $(a,b)\in X$ and let 
$O(a,b)$ be the $\bG_m$-orbit of $(a,b)$.  The (categorical) 
quotient of $X$ by $\bG_m$ is given by 
\begin{equation*}X/\!/\bG_m=\Spec \bC[t],\quad (t=xy). 
\end{equation*}

Any closed $\bG_m$-orbit is either $O(a,1)$ $(a\neq 0)$ or $O(0,0)$. 
Hence by mapping  $t=a$ (resp. $t=0$) to the orbit 
 $O(a,1)$ (resp. $O(0,0))$,  the quotient 
$X/\!/\bG_m$ is identified with the set of closed orbits. 
This is a very common phenomenon.  
The same is true in general.

\begin{subthm}\label{thm:Seshadri-Mumford}$\op{(Seshadri-Mumford)}$\quad
Let $X=\Proj B$ be a projective scheme over 
a closed field $k$, and $G$ a reductive algebraic $k$-group 
acting linearly on $B$ (hence on $X$). 
Then there exists an open subscheme $X_{ss}$ of $X$ 
consisting of all semistable points in $X$,
and a quotient $Y$ of $X_{ss}$ by $G$, that is, 
$Y=\Proj(R)$, 
where $R$ is the graded subring of $B$
of all $G$-invariants.  
To be more precise, there exist a 
$G$-invariant morphism 
$\pi$ from $X_{ss}$ onto $Y$ such that 
\begin{enumerate}
\item[(1)] For any $k$-scheme $Z$ on which $G$ acts, and for 
any $G$-equivariant morphism $\phi : Z\to X$ there exists a unique morphism
$\bar\phi : Z\to Y$ such that $\bar\phi=\pi\phi$,
\item[(2)] 
For given points $a$ and $b$ of $X_{ss}$ 
\begin{center}$\pi(a)=\pi(b)$\/ if and only 
if\/ $\overline{O(a)}\cap \overline{O(b)}\neq \emptyset$
\end{center}
where the closure is taken in $X_{ss}$,
\item[(3)] $Y(k)$ is regarded as the set of $G$-orbits closed in $X_{ss}$.
\end{enumerate}
\end{subthm}

See \cite[p.38, p.40]{MFK94} and \cite[p.~269]{Seshadri77}.  

A reductive group in Theorem~\ref{thm:Seshadri-Mumford} is by definition an
algebraic group 
whose maximal solvable normal subgroup is an algebraic torus; 
for example $\SL(n)$ and $\bG_m$ are 
reductive. \par

The following is well known. 
\begin{subthm}\label{thm:stable curve} 
{\rm (\cite{Gieseker82}, \cite{Mumford77}) }
For a connected curve $C$ 
of genus greater than one with dualizing sheaf $\omega_C$,  
the following are equivalent:
\begin{enumerate}
\item $C$ is a stable curve, (moduli-stable)
\item the $n$-th Hilbert point of $C$ 
embedded by $|\omega^m_C|$\ $(m\geq 10)$ is GIT-stable for $n$ large,
\item the Chow point of $C$ embedded by $|\omega^m_C|$\  
$(m\geq 10)$ is GIT-stable.
\end{enumerate}
\end{subthm}
\begin{proof}The proof goes as 
$(2)\Longrightarrow (1)\Longrightarrow (3)\Longrightarrow(2).$ \par
We explain only who proved these and where. \par
By \cite[Chap.~2]{Gieseker82}, let $\pi:Z_{U_C}\to U_C$ 
be the universal curve 
such that 
\begin{enumerate}
\item[(i)] 
$X_h:=\pi^{-1}(h)$ $(h\in U_C)$ is a connected curve of genus $g$  and 
degree $d=n(2g-2)$ embedded by the linear system 
$\omega_{X_h}^{n}$ into $\bP^N$ ($N=d-g$), 
\item[(ii)] the $m_0$-th Hilbert point $H_{m_0}(X_h)$ of $X_h$ 
is $\SL(N+1)$-semistable,
\end{enumerate}where $m_0$ is a fixed positive integer large enough.\par
Then by \cite[Theorem~1.0.1, p.~26]{Gieseker82}, 
$X_h$ is a semistable curve, that is, 
a reduced connected curve with nodal singularities only, any of 
whose nonsingular rational irreducible components 
meets the other irreducible components of $X_h$ at two or more points. 
For any semistable curve $X$,
$\omega_X$ is ample if and only if $X$ is a stable curve. 
Hence  (2) implies (1). \par
By \cite[Theorem~5.1]{Mumford77}, 
if $C$ is a stable curve, $\Phi_n(C)$, 
the image of $C$ by the linear system $\omega_C^n$, is Chow-stable. 
Thus (1) implies (3). (3) implies (2) by \cite{Fogarty69} and 
\cite[Prop.~2.18, p.~65]{MFK94}. See 
\cite[p.~215]{MFK94}.
\end{proof}

We have an analogous theorem for PSQASes. 

\s\begin{subthm}\label{thm:stability of PSQAS}  
Let $K=H\oplus H^{\vee}$, $N=|H|$, 
$N=|H|$, and $k$ an algebraically closed field with 
$k\ni 1/N$. \par
Suppose $e_{\min}(H)\geq 3$,  and 
  $(Z,L)$ is a closed subscheme of $\bP(V)$. 
Suppose moreover 
that $(Z,L)$ is smoothable into an abelian variety 
whose Heisenberg group is isomorphic to $\cG_H$.
Then the following are equivalent:\par
\begin{enumerate}
\item $(Z,L)$ is a level-$\cG_H$ PSQAS, (moduli-stable)
\item  any Hilbert point of $(Z,L)$ of large degree 
has a closed $\SL(V_H)$-orbit in the semistable locus, 
(GIT-stable), 
\item $(Z,L)$ is stable (that is, mapped onto itself) 
under the action of (a conjugate of) $\cG_H$, ($\cG_H$-stable).
\end{enumerate}
\end{subthm}

See \cite[Theorem~11.6]{Nakamura99} 
and \cite[Theorems~10.3, 10.4]{Nakamura04}.

\begin{subrem}\label{subrem:stability of cubics}
In Table~\ref{table1} 
we mean by GIT-stable that the cubic has a closed $\PGL(3)$-orbit 
in the semistable locus. See \cite{Nakamura04} for details.

By Table~\ref{table1}, a planar cubic is GIT-stable 
if and only if 
it is either a smooth elliptic curve or a 3-gon. 
This is a special case 
of Theorem~\ref{thm:stability of PSQAS}.
\end{subrem}


\begin{table}[h]
\caption{Stability of cubics}\label{table1}
\centering
\begin{tabular}
{llc}
\hline\noalign{\s} 
curves (sing.)& \hskip 1.5cm stability & $\op{stab. gr.}$  \\
\noalign{\s}\hline\noalign{\s}
smooth elliptic  & GIT-stable & finite \\
3 lines, no triple point & GIT-stable & 2 dim \\
a line$+$a conic, not tangent & semistable, not GIT-stable & 1 dim\\
irreducible, a node & semistable, not GIT-stable & $\bZ/2\bZ$ \\
3 lines, a triple point&not semistable&1 dim\\ 
a line$+$a conic, tangent & not semistable & 1 dim\\
irreducible, a cusp & not semistable & 1 dim\\
\noalign{\s}\hline
\end{tabular}
\end{table}

\subsection{Arithmetic moduli}
Katz and Mazur \cite{KM85} 
constructed an integral model $X(n)$ of the  moduli scheme of
 elliptic curves with level $n$-structure. 
Level structure is generalized as $A$-generators of the group of 
$n$-division points for $A=(\bZ/n\bZ)^{\oplus 2}$. 
For any $n\geq 3$, $X(n)$  is a regular $\bZ$-flat scheme such that 
$X(n)\otimes \bZ[1/n,\zeta_n]\simeq SQ_{1,A}$.
If $n=3$, $X(3)\otimes\bF_3$ is a union of four copies of $\bP^1$, 
intersecting at the  
unique supersingular elliptic curve over $\bF_9$. \par
This $X(n)$ is the model that we wish to generalize 
to the higher dimensional case, using our PSQASes or TSQASes. 
This will be discussed somewhere else.

\subsection{The other compactifications}
It is still unknown whether $AP_{g,1}^{\main}$ (or $SQ^{\toric}_{g,1}$) 
is normal or not. Therefore it is not yet known whether 
$AP_{g,1}^{\main}$ (or $SQ^{\toric}_{g,1}$) is the Voronoi compactification, 
one of the toroidal compactifications associated to 
the second Voronoi cone decomposition. 
There will exist a flat family of PSQASes or TSQASes 
over the Voronoi compactification. 
This will define, 
by the universality of the target, 
a morphism from the Voronoi compactification to 
$\barAP_{g,1}^{\main}$ (or $SQ^{\toric}_{g,1}$) or 
$SQ^{\toric}_{g,K}$ for some $K$ 
once we check the family is algebraic. 
The author conjectures that $SQ^{\toric}_{g,K}$ is normal, hence 
isomorphic to the Voronoi(-type) compactification.


\end{document}